\documentclass[11pt,reqno]{amsart}
\usepackage{amsmath,amssymb}
\usepackage{amsthm}
\usepackage{upref}
\usepackage{mathrsfs}
\usepackage{ascmac}	
\usepackage{fancybox}
		
\numberwithin{equation}{section}
	\newcommand{\vare}{\varepsilon}
	\newcommand{\N}{{\mathbb N}}
	
	\newcommand{\R}{{\mathbb R}}

	\newcommand{\F}{{\mathcal F}}

	\def\Vect#1{\mbox{\boldmath$#1$} }
	\def\pd#1#2{\dfrac{\partial#1}{\partial#2}}
	\def\spd#1#2{\frac{\partial#1}{\partial#2}}

	\renewcommand{\labelenumi}{\rm (\theenumi)}

	\def\Vect#1{\mbox{\boldmath$d$} }

	\DeclareMathOperator{\supp}{supp}

	\newcommand{\Norm}[2]{{\left\|#1\right\|_#2}}

	\newtheorem{thm}{Theorem}
	\newtheorem{lemma}{Lemma}
	
	\newtheorem*{lemmaA}{Lemma A}
	\newtheorem*{lemmaB}{Lemma B}
	\newtheorem*{lemmaC}{Lemma C}
	\newtheorem*{lemmaD}{Lemma D}
	\newtheorem*{lemmaE}{Lemma E}

	\newtheorem*{CorA}{Corollary A}

	\newtheorem{prop}{Proposition}

	\newtheorem{Cor}{Corollary}
	
	\theoremstyle{remark}
	\newtheorem{notation}{{\sc Notation}}
	\newtheorem{remark}{\rm\sc Remark}

\title[]
{Asymptotic profile of solutions \\
	for semilinear wave equations \\
	with structural damping 
}

\author[Taeko Yamazaki]{Taeko Yamazaki}

\address{Department of Mathematics, Faculty of Science and Technology, \\
	Tokyo University of Science,\\
	 Chiba, 278-8510, Japan}

\email{yamazaki$\_$taeko@ma.noda.tus.ac.jp}

\begin{document}
		
\begin{abstract}
This paper is concerned with the initial value problem for semilinear wave equation with structural damping 
$u_{tt}+(-\Delta)^{\sigma}u_t -\Delta u =f(u)$, 
where $\sigma \in (0,\frac{1}{2})$ 
and	$f(u) \sim  |u|^p$ or $u |u|^{p-1}$ with $p> 1 + {2}/(n - 2 \sigma)$.  
We first show the global existence for initial data small in some weighted Sobolev spaces on $\R^n$ ($n \ge 2$).  
Next, we show that the asymptotic profile of the solution above is given by  
a constant multiple of the fundamental solution of the corresponding parabolic equation, provided the initial data belong to weighted $L^1$ spaces.   
\end{abstract}
	
\thanks{This work is supported in part by Grant-in-Aid for Scientific Research (C) 17K05338 of JSPS}

\keywords{Semilinear wave equation; Structural damping; Asymptotic profile; Diffusion phenomena}

\subjclass{primary 35L70, secondary 35L15, 35B40}

\maketitle
			
\section{Introduction}
				
In this paper, we consider 
the unique global existence of solutions and diffusion phenomina for
	 the Cauchy problem of the semilinear wave equation with structural damping 
	 (damping term depends on the frequency) for $\sigma \in (0,\frac{1}{2})$: 	 
		\begin{eqnarray}\label{NW}
		\begin{cases}
		u_{tt} -\Delta u +(-\Delta)^{\sigma}u_t = f(u),
		&  t \geq 0,\; x\in \mathbb{R}^n,
		\\
		u(0,x)=u_0(x), \quad
		u_t(0,x)=u_1(x), & x\in \mathbb{R}^n, 
			\end{cases}
		\end{eqnarray}			
where $f \in C^{[\bar{s}],1 }(\R)$ ($1 \le \bar{s},\; [\bar{s}] < p$) satisfies   
\begin{equation}\label{fass}
\begin{cases}
&|\dfrac{d^j}{du^j} f(u) | \le C |u|^{p - j} \quad(0 \le j \le [\bar{s}]),
\\
&|\dfrac{d^j}{du^j} (f(u) - f(v)) | \le C |u - v|(|u| + |v|)^{p - [\bar{s}] } 
\quad( j = [\bar{s}] ),
\end{cases}
\end{equation} 	
for a positive constant $C$. Here, $[\bar{s}]$ denotes the integer part of $\bar{s}$.  		

For linear wave equations with structural damping: 
	\begin{eqnarray}\label{LW}
\begin{cases}
u_{tt}+(-\Delta)^{\sigma}u_t -\Delta u = 0,
&  t \geq 0,\; x\in \mathbb{R}^n,
\\
u(0,x)=u_0(x), \quad
u_t(0,x)=u_1(x), & x\in \mathbb{R}^n, 
\end{cases}
\end{eqnarray}
with $\sigma \in (0,\frac{1}{2})$, 
Narazaki and Reissig \cite{NR} gave some $L^p-L^q$ ($1 \le p \le q \le \infty$) estimates of the solutions. 
D'Abbicco and Ebert \cite{AE1} 
 showed the diffusion phenomena,  
by giving the $L^p-L^q$ decay estimates     
of the difference between the low frequency part of the solution of \eqref{LW} and that of the corresponding parabolic equation
\begin{equation}\label{H}
v_t + (-\Delta)^{1 -\sigma} v = 0,
\quad   t \geq 0,\; x\in \mathbb{R}^n,
\end{equation}
with initial data $(-\Delta)^{\sigma}u_0 + u_1$.  
Ikehata and Takeda \cite{IT} 
showed that a constant multiple of the fundamental solution of the  parabolic equation \eqref{H} gives the asymptotic profile of the solutions of \eqref{LW} with $(u_0, u_1)  \in ( L^1 \cap H^1) \times (L^1 \cap L^2)$ (see Remark \ref{ITremark}).  
		
For semilinear structural damped wave equation  
\eqref{NW} with $\sigma \in (0, \frac{1}{2})$, D'Abbicco and Reissig \cite{D1} first showed global existence and decay estimates of the solution  of   
\eqref{NW} with small initial data for space dimension $1 \le n \le 4$ and $p \in [2,n/[n-2]_+]$ such that 
	\begin{equation}
	\label{pass0}
	p > p_{\sigma}:= 1 + \frac{2}{n - 2\sigma}.   
	\end{equation}
	They showed the results by using $(L^1 \cap L^2)-L^2$ estimates of solutions of  the linear wave equation with structural damping \eqref{LW}.      
 		\cite{D1} considers also for $\sigma \in [\frac{1}{2},1]$ and shows that $p_\sigma$ is critical in the case $\sigma = 1/2$.   
Using the $L^p-L^q$ decay estimate ($1 \le p \le q \le \infty$) of solutions of 
the linear wave equations with structural damping \eqref{LW} by \cite{AE1} for low frequency part,  
D'Abbicco and Ebert \cite{AE3} (see also \cite{AE2}) showed the unique existence of  solutions of \eqref{NW} for small initial data  in some Sobolev spaces and gave the decay estimates of the solutions, in the following two cases:
\begin{equation}\label{AEass1}
p_\sigma < p, \quad n < 1 + 2 \max \left\{m \in \N; m < \frac{1+ 2 \sigma}{1 - 2\sigma} \right\}, \quad 
\end{equation}
or
\begin{equation}\label{AEass2}
p_\sigma < p < 1 + \frac{2(1 + 2\sigma)}{[n - 2(1 +2\sigma)]_+}, \quad
\left[\frac{n}{2} \right]\left[ \frac{2}{p} - 1 \right]_+(1 - 2 \sigma)
< 1 + 2\sigma. 
 \end{equation}
In \cite{AE3}, they also treated the case where $-\Delta u$ is replaced by $(- \Delta)^\delta u$ with $\delta > 0$. 
	
The assumption \eqref{AEass1} and \eqref{AEass2} for $p < 2$ restrict the space dimension from above. 
The first purpose of this paper is to remove restriction of the space dimension $n$ from above   
for every $\sigma \in (0, \frac{1}{2})$. 	

The second purpose is to give the asymptotic profile of the solutions of \eqref{NW} as $t \to \infty$, if small initial data belongs to some weighted 
$L^1$ spaces.    		
We show that a constant multiple of the fundamental solution of the parabolic equation \eqref{H} gives the asymptotic profile of \eqref{NW} (Theorem \ref{thmdiff}).  
As as far as the author knows, there seems to be no results on the asymptotic profile for semilinear wave equation with structural damping \eqref{NW} for $\sigma \in (0,\frac{1}{2})$.  

In the case $\sigma = 0$, the asymptotic profile for semilinear damped wave equation is investigated. 
Since we treat nonlinear term not necessarily absorbing, we only refer to the results for non-absorbing type nonlinear  term. Then if $1 < p \le p_0$ where 
\begin{equation*}
p_0 := 1 + \frac{2}{n} \ :\text{Fujita Exponent},    
\end{equation*}
then 
the solution of the semilinear damped wave equation blows up when $f(u) = |u|^p$ and the integrals of initial data on $\R^n$ are positive 
(see \cite{LZ, ToYo, Zh, IO} ).   
On the other hand, in the case $p > p_0$, small data global existence is widely studied, (see \cite{Ma, NO, ToYo, N1, Hy, HO, Na, ITani, NN, IIW}, for example, 
and the references therein).   
The asymptotic profiles of the solutions are obtained as follows.    
Galley and Raugel \cite{GL} ($n = 1$), Hosono and Ogawa \cite{HO} ($n = 2$), 
 showed that the asymptotic profile of the solutions  
is given by a constant multiple of the heat kernel $G(t,x)$, provided the initial data belong to some Sobolev spaces. 
(See also Kawakami and Takeda \cite{KT} for higher order asymptotic expansion in the case $n \le 3$.) 
For general space dimensions,   
Hayashi, Kaikina and Naumkin \cite{Hy} 
proved the unique existence of global solution $u \in C([0,\infty);H^{\bar{s}} \cap H^{0,\delta})$ 
for small initial data belonging to some weighted $L^1$ spaces, 
and showed that a constant multiple of the heat kernel  gives 
the asymptotic profile of the solutions (see Remark \ref{KarchHayashi}).   

We consider the equation in weighted Sobolev spaces as in \cite{Hy}. 
The high frequency part of the structural damped wave equation has a good regularizing property.  
However, unlike the damped wave equation ($\sigma = 0$), the Fourier transform of the kernel of the linear structural damped wave equation is singular at the origin. 
This fact causes the difficulty when we treat the equation in weighted Sobolev spaces.
To get around this difficulty, we estimate the low frequency part in a new way employing Lorentz spaces (Lemma \ref{weightlemma}).  
For the estimate of nonlinear term, we use the method of \cite{Hy, IIW}.  

\medskip

This paper is organized as follows.  
\begin{itemize}
	\item In section 2, we list some notations and state main results.  
	\item In section 3, we list known preliminary lemmas.  
	\item In section 4, we estimate kernels.  
	\item In section 5, we prove Theorem 1. 
	\item In section 6, we estimate a nonlinear term. 
	\item In section 7, we estimate a convolution term.  
	\item In section 8, we prove Proposition 1 and Theorems 2 and 3.  That is, we prove the global existence of the solution of semilinear wave equation with structural damping, and give the asymptotic profile of the solutions. 
\end{itemize}
	
	\section{Main Results}			
	
	Before stating our results, we list some notations.  

	\begin{notation} We write  $\varphi(x) \lesssim \psi(x)$ on $I$ if there exists a positive constant $C$ such that
			\begin{align*}
			&\ \varphi(x) \leq C\psi(x) \; \text{for every} \; x \in I.  
			\end{align*}			
			We write  $\varphi(x) \sim \psi(x)$ on $I$, if $\varphi(x) \lesssim \psi(x)$ and 
			$\psi(x) \lesssim \varphi(x)$ on $I$. 			
		\end{notation}
	
	\begin{notation}
		For $a \in \R$, $[a]_+: = \max \{a, 0\}$. 
		\end{notation}
	
\begin{notation}
For every $q \in [1,\infty]$, we abbreviate $\R^n$ in $L^q(\R^n)$, and $L^q$ norm  is denoted by $\|\cdot\|_q$. 
\end{notation}

\begin{notation}
Let	$ H^{s,\delta} = H^{s,\delta}(\R^n)$ denote the weighted Sobolev space equipped with the norm
	\begin{equation*}
	\| u \|_{{H^{s,\delta}}}  = \| \langle  x \rangle^\delta ( 1 - \Delta)^{s/2} \|_{{L^2}}. 
	\end{equation*}	
$H^{s,0}$ equals $H^s$.  
Let	$ \dot H^{s} = \dot H^{s}(\R^n)$ denote the homogeneous Sobolev space equipped with the norm
\begin{equation*}
\| u \|_{{\dot H^{s}}}  = \|  ( - \Delta)^{s/2} \|_{{L^2}}. 
\end{equation*}		
\end{notation}

			\begin{notation}[see {\cite[section 1.3]{BL}}, for example]
		Let $q \in (1,\infty)$ and $r \in [1,\infty]$.  
		Let $\mu$ be the Lebesgue measure on $\R^n$.  
		The distribution function $ m(\tau, \varphi)$ is defined by 
		$$
		m(\tau, \varphi) := \mu(\{x; |\varphi(x)| > \tau \}). 
	$$
	The Lorentz space $L_{q,r}$ consists of   all locally integrable function $\varphi$ on $\R^n$ such that
	\begin{align*}
&	\| \varphi \|_{{q,r}} 
:= \left(\int_0^\infty (t^{1/q} \varphi^*(t)^r \frac{dt}{t}\right)^{1/r} < \infty
	\quad \text{when} \quad r  < \infty,
	\\
&	\| \varphi \|_{{q,\infty}} 
	:= \sup_t t^{1/q} \varphi^*(t)   = \sup_\tau \tau m(\tau,\varphi)^{1/q}< \infty,	
	\end{align*}
	where 
	$
	\varphi^*(t) = \inf \{\tau; m(\tau,\varphi) \le t \} 
	$ (the rearrangement of $\varphi$).   
		\end{notation}

				\begin{notation}
			For $\kappa \in (0,n)$,  Riesz potential is the operator
				\begin{eqnarray*}
					I_{\kappa}f(x) := \frac{1}{|x|^{n-\kappa}}* f 
				=	\int_{\mathbb{R}^n} \frac{f(y)}{|x-y|^{n-\kappa}} dy
				=  C_{n,\kappa} \mathcal{F}^{-1}(|\xi|^{- \kappa} \hat f(\rho))	. 
				\end{eqnarray*}
			\end{notation} 
			
First we give the asymptotic profile of the solutions to linear wave equation with structural damping.  
			
\begin{thm} \label{thm_lin_diff} 
Let $n \ge 1$. 
Let $(u_0, u_1)  \in ( L^1 \cap L^2) \times (L^1 \cap H^{-2 \sigma})$ 
such that $|\cdot|^{\theta_j} u_j \in L^1$ with $\theta_j \in [0,1]$ $(j = 0,1)$.  
Let $u \in C([0,\infty); H^1) \cap C^1((0,\infty);L^2)$ be a unique global solution of  
\eqref{LW}.  				
	Then the following holds. 
			\begin{equation}
	\begin{aligned}\label{lin_diff}
		&	\| u(t,\cdot) - \vartheta_0 H_{\sigma}(t, \cdot) 
	-	\vartheta_1 G_{\sigma}(t, \cdot) 	  \|_2 
	\\
	&	\lesssim
	\langle t \rangle 
	^{{
			\frac{1}{1 - \sigma}\left(-\frac{n}{4} +  2\sigma - 1 \right)
		 }}
	\|u_0\|_1
	+	\langle t \rangle 
	^{{\max \{
			\frac{1}{1 - \sigma}\left(-\frac{n}{4} +  3\sigma - 1 \right),
			\frac{1}{\sigma}		 \left(-\frac{n}{4} +  \sigma \right) 
			\} }}
	\| u_1 \|_{{{1}}} 
	\\
	& \quad	 + 
	e^{-\vare_{\sigma} t} (\| u_0 \|_2+  \|u_1 \|_{{H^{-2 \sigma} }})
	\\
	& \quad +	t ^{ \frac{1}{1-\sigma}(-\frac{n}{4} -\frac{\theta_0}{2} )}
	\| | \cdot |^{\theta_0}  u_0 \|_1 
	+ t ^{ \frac{1}{1-\sigma}(-\frac{n}{4} + \sigma -\frac{\theta_1}{2} )}\| | \cdot |^
	{\theta_1}  u_1 \|_1,
	\end{aligned}
	\end{equation}
	where
	\begin{align}
H_{\sigma}(t, x) &:=\F^{-1}[e^{- |\xi|^{2(1 - \sigma)} t}](x), 
\nonumber	\\
		G_{\sigma}(t, x) &:=\F^{-1}[|\xi|^{-2 \sigma} e^{- |\xi|^{2(1 - \sigma)} t}](x)
	=  C_{n,2 \sigma}^{-1} I_{2\sigma} H_{\sigma}(t, x), 
\label{Gdef}
	\\
\vartheta_0 &:=  \int_{\R^n}u_0 (y) dy,
\quad	\vartheta_1 :=  \int_{\R^n} u_1(y) dy.  
\label{theta1def}
	\end{align}
\end{thm}	

\begin{remark}
	\begin{equation*}
	\max \Big\{
	\frac{1}{1 - \sigma}\big(-\frac{n}{4} +  3\sigma - 1 \big),
	\frac{1}{\sigma}		 \big(-\frac{n}{4} +  \sigma \big) 
	\Big\} 
	= \begin{cases}
	\frac{1}{1 - \sigma}\left(-\frac{n}{4} +  3\sigma - 1 \right) &
	\text{if} \; 8 \sigma \le n,
	\\
	\frac{1}{\sigma}		 \left(-\frac{n}{4} +  \sigma \right) &
	\text{if} \; 8 \sigma \ge n. 
	\end{cases}
	\end{equation*}	
\end{remark}

\begin{remark}\label{Gdecay}
The function $G_\sigma(t,x)$ is the fundamental solution of the parabolic equation \eqref{H}.  
We easily see that 
	\begin{equation}\label{Gsigmadecay}
	\|  G_{\sigma}(t,\cdot) \|_{{2}}
	= 	\| \hat G_{\sigma}(t,\cdot)
		 \|_{{2}} 
	\sim t^{ \frac{1}{1-\sigma}(-\frac{n}{4} + \sigma)},
	\quad 
\| H_{\sigma}(t,\cdot)\| = 
\| \hat H_{\sigma}(t,\cdot)
\|_{{2}}
\sim {t}^{  -\frac{n}{4(1-\sigma)}},
\end{equation}
(see \eqref{Gest02} and \eqref{Hest02}). 
Putting $\theta_0 = [\theta - 2 \sigma]_+ $ and $\theta_1 = \theta$ ($\theta \in (0,1]$) in 
 \eqref{lin_diff}, and taking \eqref{Gsigmadecay} and the assumption $\sigma \in (0,\frac{1}{2})$ into consideration, we obtain   
	\begin{equation*}
\begin{aligned}
&	\| u(t,\cdot) - \vartheta_1 G_{\sigma}(t, \cdot) 	  \|_2 
\\
&	\lesssim
 t 
^{{\max \left\{
		\frac{1}{1 - \sigma}\left(-\frac{n}{4} +  \sigma - 
		\min \{1 - 2\sigma,  \frac{\theta}{2} \} \right),
		\frac{1}{\sigma}		 \left(-\frac{n}{4} +  \sigma \right) 
		\right\} }}
\\
& \quad \times ( 
\| u_0 \|_2+  \|u_1 \|_{{H^{-2 \sigma} }} 
+ \| \langle \cdot \rangle^{{[\theta - 2 \sigma]_+} } u_0 \|_1 
+ \| \langle \cdot \rangle^{\theta}  u_1 \|_1). 
\end{aligned}
\end{equation*}
Thus, the decay order of $\|u(t,\cdot) -  \vartheta_1 G_{\sigma}(t, \cdot)\|_2 $ is larger than that of  $\|G_{\sigma}(t,\cdot)\|_2$ itself, and therefore,  $\vartheta_1 G_{\sigma}(t,x)$ gives the asymptotic profile of the solution if $\vartheta_1 \neq 0$.  

If $u_1 = 0$, then \eqref{lin_diff} means
	\begin{equation}
\begin{aligned}
&	\| u(t,\cdot) - \vartheta_0 H_{\sigma}(t, \cdot) 
 \|_2 
\lesssim
t^{{
		\frac{1}{1 - \sigma}\left(-\frac{n}{4} - \min \{ 1 -  2\sigma, \frac{\theta_0}{2}\} \right)
}}
( \| u_0 \|_2 +  \|\langle \cdot \rangle^{\theta_0} u_0 \|_1). 
\end{aligned}
\end{equation}
Thus, the decay order of $\|u(t,\cdot) -  \vartheta_0 H_{\sigma}(t, \cdot)\|_2 $ is larger than that of  $\|H_{\sigma}(t,\cdot)\|_2$ itself if $\theta_0 > 0$, and therefore,  $\vartheta_0 H_{\sigma}(t,x)$ gives the asymptotic profile if $\vartheta_0 \neq 0$.  
\end{remark}

\begin{remark}\label{ITremark}	
Ikehata and Takeda \cite[Theorem 1.2]{IT} showed 
			\begin{equation*}
		\| u(t,\cdot) - \vartheta_1 G_{\sigma}(t,\cdot)  \|_2 
		= o( t^{ \frac{1}{1-\sigma}(-\frac{n}{4} + \sigma)})
		\end{equation*}
		as $t \to \infty$ for initial data in $(u_0, u_1) \in ( L^1 \cap H^1) \times (L^1 \cap L^2)$.   
		
			If $u_1 = 0$, 	Karch \cite[Corollary 4.1]{K}  showed 
		\begin{equation*}
		\| u(t,\cdot) - \vartheta_0 H_{\sigma}(t,\cdot)  \|_2 
		= o( t^{ -\frac{n}{4(1-\sigma)}})
		\end{equation*}
		as $t \to \infty$ for $u_0  \in L^1 $. 
	\end{remark}

	\begin{thm}[Global existence of the solution]\label{existence2}
		Let $n \ge 2$, and 
		\begin{equation}
	\label{pass1}
	p > p_{\sigma}:= 1 + \frac{2}{n - 2 \sigma}.   
	\end{equation}
	Assume that $\bar{s} \ge 1$ and $[\bar{s}] < p$. 
	If $ 2 \bar{s} < n$,  assume moreover that
	\begin{alignat}{2}\label{pass2}
	& p \le  1 + \frac{2}{n - 2\bar{s}}. 
		\end{alignat}
		Assume that $f \in C^{[\bar{s}],1 }(\R)$ satisfies \eqref{fass}. 
Let $q_j\; (j=0,1) $ be numbers such that
\begin{equation}
\left\{
\begin{aligned}
&  q_0 = \frac{2n}{n + 2 - 4\sigma}, \; 
&	& q_1 = \frac{2n}{n +2}
\qquad \quad \text{if} \quad 
p_\sigma < p < 1 + \frac{4}{n + 2 - 4\sigma},&
\\
&1 < q_0 < \frac{n(p-1)}{2},\; 
&	& 1 < q_1 <  \frac{n(p-1)}{2(1 + (p-1)\sigma)}
&
\\
&  \; 
& & 
\qquad \qquad \qquad \qquad \text{if} \quad 
1 + \frac{4}{n + 2 - 4\sigma} < p \le 1 + \frac{4}{n}, &
\\
&	1 < q_0 < 2,\;  
& & 1 < q_1 < \frac{2n}{n + 4 \sigma}
\qquad \quad \text{if} \quad 1 + \frac{4}{n} < p. &
\end{aligned}
\right. 
\end{equation}

	\renewcommand{\labelenumi}{(Case \arabic{enumi}). }
		\begin{enumerate}
		\item In the case 
		$p_\sigma < p \le 1 + \frac{4}{n + 2 - 4\sigma}$,
			let $\delta$ be a number satisfying
		\begin{equation}\label{deltaass1}
		2 \left(\frac{1}{p-1} + \sigma \right) - \frac{n}{2} - 1 < \delta \le 
		\frac{2}{p-1} - \frac{n}{2}. 
		\end{equation} 
			Then there exists a positive number $\vare$ such that if initial data 
		\begin{equation}\label{initial1}
		u_0 \in 
		H^{\bar{s}} 
		\cap H^{0,\delta}, \langle \cdot \rangle^\delta u_0 \in L^{q_0,2 }, \quad 
		u_1 \in \dot H^{\bar{s}-1}, \langle \cdot \rangle^\delta u_1 \in L^{{ q_1,2 }}
		\end{equation}	
		satisfy
		\begin{equation}
		\begin{aligned}\label{initial2}
		&\|\langle \cdot \rangle^\delta u_0 \|_{q_0,2 }
		+ \|\langle \cdot \rangle^\delta u_0 \|_2
		+  \|  u_0 \|_{{H^{\bar{s}}}}
		+ 		\|\langle \cdot \rangle^\delta u_1 \|_{{ q_1,2 }}
		+ \|  u_1 \|_{{\dot H^{s-1} }}
				\le \varepsilon,
		\end{aligned}
		\end{equation}
		then initial value problem \eqref{NW}  has a unique global  solution \\
		$u \in C([0,\infty); H^{\bar{s}}\cap H^{0,\delta}) \cap C^1((0,\infty);H^{\bar{s} - 1}) $. 
		
		\item
		In the case $p >1 +  \frac{4}{n + 2 - 4\sigma}$,
		there exists a positive number $\vare$ such that if initial data 
		\begin{equation*}
		u_0 \in H^{\bar{s}} \cap L^{q_0,2}, 
		\quad 
		u_1 \in H^{\bar{s}-1} \cap L^{q_1,2},
		\end{equation*}	
		satisfy
		\begin{align}\label{initial3}
		&\| u_0 \|_{q_0,2} + \| u_0 \|_{H^{\bar{s}}}
				 + 
		\|u_1 \|_{q_1,2}
		+ \|u_1\|_{{ H^{\bar{s}-1}}}  
		\le \varepsilon,
		\end{align}	
then initial value problem \eqref{NW}  has a unique global  solution \\
$u \in C([0,\infty); H^{\bar{s}}) \cap C^1((0,\infty);H^{\bar{s} - 1}) $. 					
	\end{enumerate}
		\end{thm}		
			
	\begin{remark}
	We note that $L^{q_j} = L^{q_j, q_j} \subset L^{q_j, 2}$ by Lemma A given later, since $q_j \le 2$.   
\end{remark}

	\begin{remark}
		If the space dimension $n = 2$, then 
			$ 1 + \frac{4}{n + 2 - 4\sigma} \le p_{\sigma}$,
			and therefore, (Case~1) does not occur. 
		\end{remark}
	
We prove Theorem \ref{existence2} by using the following proposition.  

\begin{prop}[Global existence of the solution] \label{existence}
	Let $n \ge 2$ and $r \in [1,\frac{2n}{n + 4 \sigma})$.  
Let 
	\begin{equation}
	\label{pass}
	p > p_{\sigma,r}:= 1 + \frac{2r}{n - 2r\sigma}.   
	\end{equation}
	Assume that $\bar{s} \ge 1$ and $[\bar{s}] < p$. 
If $ 2 \bar{s} < n$,  assume moreover \eqref{pass2}. 
Assume that $f \in C^{[\bar{s}],1 }(\R)$ satisfies \eqref{fass}. 
	Let $\delta$ be a non-negative constant satisfying 
	\begin{equation} 
	\label{delta}
	n(\frac{1}{r} - \frac{1}{2}) - 1 \le 
	\delta < n(\frac{1}{r} - \frac{1}{2}) - 2\sigma. 
	\end{equation}	
	and	
	\begin{equation}
	\label{delta2}
		\begin{cases}
	&	n(\frac{1}{p} - \frac{1}{2}) <  \delta, \quad \text{if}\quad  r = 1,\\
	& 	n(\frac{1}{pr} - \frac{1}{2}) \le \delta \quad \; \text{if}  \quad r \in (1,\frac{2n}{n + 4 \sigma}). 
	\end{cases}\end{equation}
	Let
	\begin{equation}\label{hat_qdef}
	\hat q_0 = \frac{nr}{n - r(\delta + 2\sigma)}, \quad 
	\hat q_1 =  \frac{nr}{n - r \delta}. 
	\end{equation}
	Then there exists a positive number $\vare$ such that if initial data 
	\begin{equation}\label{initialass_0}
	u_0 \in H^{\bar{s}} \cap H^{0,\delta}, 
	\langle \cdot \rangle^\delta u_0 \in L^{{\hat q_0,2}}, \quad 
	u_1 \in \dot H^{\bar{s}-1}, \langle \cdot \rangle^\delta u_1 \in 
	L^{{\hat q_1,2 }}	
	\end{equation}	
	satisfy
	\begin{equation}\label{initialass_1}
	\begin{aligned}
	&	  \|\langle \cdot \rangle^\delta u_0 \|_{{\hat q_0,2}}
	+ \|\langle \cdot \rangle^\delta u_0 \|_{{{{2}} }}
	+  \|  u_0 \|_{{H^{\bar{s}}}}
		\\	
	&  + 
	\|u_1\|_{{{1}}}
	+ \|\langle \cdot \rangle^\delta u_1 \|_{{\hat q_1,2}}
		+ \|  (-\Delta)^{\frac{\bar{s}}{2}} (1 - \Delta)^{-\frac{1}{2}}  u_1 \|_{{2}} 
	\le \vare,
	\end{aligned}
	\end{equation}		
	in the case $r = 1$, 	and
		\begin{equation}
	\begin{aligned}\label{initialass_r}
	&\|\langle \cdot \rangle^\delta u_0 \|_{{{\hat q_0},2 }}
	+ \|\langle \cdot \rangle^\delta u_0 \|_{{2}}
	+ \|u_0\|_{{H^{\bar{s}}}} 
	\\
	& \quad + 
		\|\langle \cdot \rangle^\delta u_1 \|_{{{{ \hat q_1   ,2}} }}
	+ \|  (-\Delta)^{\frac{\bar{s}}{2}} (1 - \Delta)^{-\frac{1}{2}}  u_1 \|_{{2}}
	\le \varepsilon,
	\end{aligned}		\end{equation}
	in the case $r \in (1,2]$, 	
	then initial value problem \eqref{NW} has a unique global solution 
	$u \in C([0,\infty); H^{\bar{s}} \cap H^{0,\delta}))	\cap C^1((0,\infty);H^{\bar{s}-1}) $. 
	
	Furthermore, the solution satisfies estimate:
	\begin{equation}
	\begin{aligned}\label{sol_est}
	\sup_{t > 0} 
	&	 \Big(
	\langle t \rangle ^{{\frac{1}{1 - \sigma} \left(\frac{n}{2}(\frac{1}{r} - \frac{1}{2}) -  \sigma \right) }}
	\Norm{ u(t,\cdot)}{{2}}
	+
	\langle t \rangle ^{{\frac{1}{1 - \sigma} \left(\frac{n}{2}(\frac{1}{r} - \frac{1}{2}) - \frac{\delta}{2} - \sigma \right) }}
	\Norm{|\cdot|^\delta u(t,\cdot)}{{2}}			
	\\
	&\quad +
	\langle t \rangle ^{{\frac{1}{1 - \sigma} \left(\frac{n}{2}(\frac{1}{r} - \frac{1}{2}) -  \sigma  + \frac{\bar{s}}{2}
			\right) }}
\|(-\Delta)^{\frac{\bar{s}}{2}} u(t,\cdot)\|_{{2}}
	\Big) < \infty. 
	\end{aligned}
	\end{equation}
\end{prop}		

\begin{remark}
	The assumption $r < \frac{2n}{n + 4\sigma}$ means that 
	$n(\frac{1}{r} - \frac{1}{2}) - 2\sigma > 0$.  		
The inequality
	\begin{equation*}
	n(\frac{1}{pr} - \frac{1}{2}) <	n(\frac{1}{r} - \frac{1}{2}) - 2\sigma
	\end{equation*}	
is equivalent to 
	\begin{equation*}
	p > 1 + \frac{2 \sigma r}{n - 2 \sigma r},
	\end{equation*}
	which holds by \eqref{pass} since $ \sigma < 1$. 
	Hence, we can take a non-negative number $\delta$ satisfying the assumption \eqref{delta} and \eqref{delta2}.  						
\end{remark}

\bigskip

If initial data belong to weighted $L^1$ space, the asymptotic profile of the solution is given by a constant multiple of the fundamental solution of the parabolic equation \eqref{H}.   

\begin{thm}[Asymptotic profile] \label{thmdiff} 
Assume the assumption of Proposition \ref{existence} with $r = 1$. 
	Let $\vare$ be a positive constant given by Proposition \ref{existence} for $r = 1$, and 
assume that initial data satisfy \eqref{initialass_0} and \eqref{initialass_1}.
		Let $\nu$  be an arbitrary number satisfying
	\begin{equation}
\label{nudef}
	0 < \nu  <
	\min 	\big \{
\frac{n}{4}(p-2) + \frac{1}{2}p \delta,     \delta 
\big \}. 
	\end{equation}
	Assume moreover that
\begin{equation}\label{nudef2}
	\nu < 	\frac{\delta}{2 \bar{s}}(n - \frac{p}{2}(n - 2 \bar{s}))	
	\qquad 	\text{if} \quad \bar{s} < \frac{n}{2}. 
\end{equation}	
	Then there is a constant $C$ depending on 
\begin{equation*}\begin{aligned}
&
\|\langle \cdot \rangle^{[\theta - 2\sigma]_+} u_0 \|_1+
 \|\langle \cdot \rangle^\delta u_0 \|_{{{{\frac{n}{n - (\delta + 2\sigma)}}} }}
+ \|\langle \cdot \rangle^\delta u_0 \|_{{{{2}} }}
+  \|  u_0 \|_{{H^{\bar{s} }}}
\\
& 
+ \|\langle \cdot \rangle^\theta u_1 \|_1 + \|\langle \cdot \rangle^\delta u_1 \|_{{{{ \frac{n}{n -  \delta}}} }}
+ \|  (-\Delta)^{\frac{\bar{s}}{2}} (1 - \Delta)^{-\frac{1}{2}}  u_1 \|_{{2}} 
 \end{aligned}
\end{equation*}
such that the solution 
$u \in C([0,\infty); H^{\bar{s}} \cap H^{0,\delta}) \cap C^1((0,\infty);H^{\bar{s} - 1}) $ 
of \eqref{NW}, which is given by Proposition \ref{existence}, satisfies the following:
\begin{equation}\begin{aligned}
	\Big\|
	&
	u(t,\cdot) - 
	\varTheta G_{\sigma}(t, \cdot) 
		\Big\|_{{2}}
\\
& \le C 
  t^{{ \max \left\{
		\frac{1}{1 - \sigma} \left(-\frac{n}{4} + \sigma  
				- \min \{( p - 1)\left(\frac{n}{2} - \sigma \right)  - 1, 1 - 2 \sigma, \nu, \frac{\theta}{2} \} 
		\right),
		\frac{1}{\sigma} \left(-\frac{n}{4} +  \sigma \right) 
		\right\}
		 }},	           
\label{heatdiff}
	\end{aligned}\end{equation}
	where $G_{\sigma}$ is defined by \eqref{Gdef} and 
	\begin{align}\label{thetadef}
	 \varTheta &:=  \int_{\R^n} u_1(y)dy  + \int_0^\infty \int_{\R^n}f(u(\tau,y))dy d\tau. 
	\end{align}
	\end{thm}	
	
\begin{remark}
The right-hand sides of \eqref{nudef} and \eqref{nudef2} are positive.  
In fact, the assumption \eqref{delta2} implies 
		$\frac{n}{2}(p-2) + p \delta  > 0$, 
		and \eqref{pass2} implies 
		$	n - \frac{p}{2}(n - 2 \bar{s}) > 0$. 
		Hence we can take $\nu$ satisfying \eqref{nudef} and \eqref{nudef2}.  
	\end{remark}
	
\begin{remark}\label{Gdecay2}
Since \eqref{Gsigmadecay} holds, \eqref{heatdiff} implies that $G_\sigma$ gives the asymptotic profile of the solution 
if $\varTheta \neq 0 $.  
\end{remark}

\begin{remark} \label{KarchHayashi}
	In the case $\sigma =  0$, 
Hayashi,  Kaikina and Naumkin \cite{Hy} 
showed the existence of global solution 
$u \in C([0,\infty);H^{\bar{s}} \cap H^{0,\delta})$ 
of  the semilinear damped wave \eqref{NW} with $\sigma = 0$ for small initial data
$
u_0 \in H^{\bar{s}} \cap H^{0,\delta}, u_1 \in H^{\bar{s}-1} \cap H^{0,\delta}
$
with $\delta > \frac{n}{2}$,  
and showed
\begin{align*}
\Big\|
& u(t,\cdot) -  \tilde \varTheta G_0(t, \cdot) 
\Big\|_{{q}}
\le C t^{{ -\frac{n}{2}(1 - \frac{1}{q}) 
		- \min \{\frac{n}{2}(p - 1) - 1, \frac{\delta}{2}- \frac{n}{4}, \nu \}	
	}},
	\end{align*}	
	for $2 \le q	\le \frac{2n}{n - 2 \bar{s}}$,	where 		
		$ \tilde \varTheta =  \int_{\R^n} (u_0(y) +u_1(y))dy  + \int_0^\infty \int_{\R^n}f(u(\tau,y))dy d\tau$,  
		$G_0$ is the heat kernel (\eqref{Gdef} with $\sigma = 0$) and 
	$0  <  \nu < 1$.  
\end{remark}
	
							\section{Preliminary lemmas}
				
			We list some properties for weak $L^p$ and Lorentz spaces which 	are used in this paper 
	(see \cite[section 1.3]{BL}, \cite{O}, for example).  
			
		\begin{lemmaA}
					Let $q \in (0,\infty)$.  Then
			\begin{align*}
			&L^{q,q} = L^q, \quad L^{q,\infty} = L_q^*,
			\label{Lorentz1}
			\\
			&		L^{q,\rho_1} \subset L^{q,\rho_2} \quad \text{if} \quad 1 \le \rho_1 \le \rho_2 \le \infty.   
				\end{align*}
		\end{lemmaA}
	
		\begin{lemmaB}					
		Assume that $\mu, \rho, \nu \in (1,\infty)$ and $\tilde \mu, \tilde \rho, \tilde \nu \in [1,\infty]$  
								satisfy
								$$
								\frac{1}{\mu} = \frac{1}{\rho}	 + \frac{1}{\nu}, \quad 
								\frac{1}{\tilde \mu  } = \frac{1}{\tilde \rho}	 + \frac{1}{\tilde\nu}.  
								$$
													Then
								\begin{equation*}
								\Norm{ f g}{{{\mu,\tilde \mu}}} \lesssim 	
								\Norm{f}{{{\rho, \tilde \rho}}}
								\Norm{g}{{{\nu, \tilde \nu}}},			
								\end{equation*}
								provided the right-hand side is finite.  
							\end{lemmaB}
							
The next corollary immediately follows from Lemma B.  						\begin{CorA}					
								Let  $\omega >0$, 
															$\mu, \nu \in (1,\infty)$  and  
								$\tilde \mu \in [1,\infty]$.  
														If
								$$
								\frac{1}{\mu} = \frac{\omega}{n} + \frac{1}{\nu},
								$$
								then the following hold. 
																\begin{align*}
								\Norm{| x |^{-\omega} f}{{{\mu, \tilde \mu}}} 		
								&\lesssim 			\| | x |^{-\omega}\|_{{{\frac{n}{\omega}, \infty} }}
								\Norm{f}{{{\nu,\tilde \mu}}}
								\lesssim 
								\Norm{f}{{{\nu,\tilde \mu}}}
								\end{align*}	
														\end{CorA}

							\begin{lemmaC}
								Let $q \in (2,\infty)$, 
								and let $q^\prime $ be the dual exponent of $q$, that is, $\frac{1}{q} + \frac{1}{q^\prime} = 1$.  
								Let $\nu \in [1,\infty]$.  
								Then 
								$$
								\|\F[\varphi]\|_{{{q,\nu} }} 
								\lesssim 
								\|\varphi \|_{{{q^{\prime},\nu}}}. 
								$$			
							\end{lemmaC}
											
							\begin{lemmaD}[Young's inequality]
								Let $q, \rho \in(1,2]$ such that  $\frac{1}{q} + \frac{1}{\rho} = \frac{3}{2}$.  
								Let $s,t \in [2,\infty)$ such that  $\frac{1}{s} + \frac{1}{t} = \frac{1}{2}$.  
								Then
								\begin{align*}
								&\Norm{\varphi*\psi }{{2}} 
								\lesssim			\Norm{\varphi}{{{\rho,s}}}
								\Norm{\psi}{{{q,t} }}. 
								\end{align*}				
								\end{lemmaD}	
							
							\begin{lemmaE}[sharp Sobolev embedding theorem]		
								Let $ q \in [2, \infty)$ and $s \ge 0$.  If 
								\begin{equation*}
								\frac{n}{2} -s \le \frac{n}{q},
								\end{equation*}
								then
								$$
								H^s(\R^n) \subset L^{q,2}(\R^n). 
								$$
				 \end{lemmaE}
		
	\section{Decay estimate for the kernels}
			
In this section, we estimate the kernel of the following linear wave equation with structural damping \eqref{LW}. 		
				
By Fourier transform, the equation \eqref{LW} is transformed to
				\begin{align*}
									\hat u_{tt}+|\xi|^{2\sigma} \hat u_t  +  |\xi|^2 \hat u=0 \quad (t > 0),  
					\qquad \hat u(0) =\hat u_0, \quad \hat u_t(0) =\hat u_1.  
				\end{align*}
		Hence the solution $u$ of \eqref{LW} is expressed as 	
			\begin{align}
			u(t,x)&= (K_0(t,\cdot)\ast u_0)(x) + (K_1(t,\cdot)\ast u_1)(x),\label{udef}
			\end{align}
		where
							\begin{align}
\widehat K_0(t,\xi)&=\frac{1}{\lambda_+(|\xi|) - \lambda_-(|\xi|)}(\lambda_+(|\xi|) e^{\lambda_-(|\xi|) t}-\lambda_-(|\xi|) e^{\lambda_+(|\xi|) t}),\label{K0def}
\\
\widehat K_1(t,\xi)&=\frac{1}{\lambda_+(|\xi|) - \lambda_-(|\xi|)}(e^{\lambda_+(|\xi|) t}- e^{\lambda_-(|\xi|) t}),
\label{K1def}
\\
\lambda_{\pm}(|\xi|) 
&=	\frac{1}{2}	\left( -|\xi|^{2\sigma} \pm \sqrt{|\xi|^{4\sigma} - 4|\xi|^2 }\right) 
					\label{lambda_def1}
						\\
					&= \begin{cases}
						\frac{1}{2}  |\xi|^{2 \sigma}\left(-1\pm \sqrt{1-4|\xi|^{2(1 - 2\sigma)} } \right) \ \ 
						\; \text{if} \quad |\xi|^{1 - 2\sigma} < \frac{1}{2},
												\\						
					 \frac{1}{2}|\xi|^{2\sigma} 	
					 \left( -1 \pm  i \sqrt{4|\xi|^{2(1 - 2\sigma)} -1} \right)  \ \ 	 \text{if} \quad |\xi|^{1 - 2\sigma} > \frac{1}{2}. 
							\end{cases}\label{lambda_def2}				
				\end{align}
	
				We divide $K_0$ and $K_1$ into
		\begin{align}
\widehat{K_1^{\pm}}(t,\xi) 
&:=  \pm \frac{e^{\lambda_{\pm}(|\xi|)t}}{\lambda_+(|\xi|) - \lambda_-(|\xi|)},
\label{K1+-def} 
\\
\widehat{K_0^{\pm}}(t,\xi) 
&:=  \mp \frac{
\lambda_{\mp}(|\xi|)	e^{\lambda_{\pm}(|\xi|)t}}{\lambda_+(|\xi|) - \lambda_-(|\xi|)} 
= -\lambda_{\mp}(\xi)\widehat{K_1^{\pm}}(t,\xi). 
\label{K0+-def}		\end{align}
				
Let  $\chi_{low}(\xi) \in C^\infty(\R^n)$ be a function such that $\chi_{low}(\xi) = 1$ for $|\xi| \le  2^{{ - \frac{3}{1 - 2 \sigma}}}$ and 
$\chi_{low}(\xi) = 0$ for $|\xi| \ge   2^{{ - \frac{2}{1 - 2 \sigma}}}$. 
								Let  $\chi_{high}(\xi) \in C^\infty(\R^n)$ be a function such that $\chi_{high}(\xi) = 1$ for $|\xi| \ge  2$ and 
			$\chi_{high}(\xi) = 0$ for $|\xi| \le  1$. 
		
			We put  
			$$
			\chi_{mid}(\xi) := 1 - \chi_{low}(\xi) - \chi_{high}(\xi), \quad
					\chi_{mh}(\xi) :=  1 - \chi_{low}(\xi) = \chi_{mid}(\xi) + \chi_{high}(\xi). 
			$$ 
			Here we note that 
			\begin{equation}
			\label{mid}
		\supp \chi_{mid} \subset \{\xi ; |\xi| \in [2^{{-\frac{3}{1 - 2\sigma} }}, 2] \}, \quad
			\supp \chi_{hm} \subset  \{\xi ; |\xi| \in [2^{{-\frac{3}{1 - 2\sigma} }}, \infty) \}. 
			\end{equation}
			We put
			\begin{equation*}\begin{aligned}
			K_{j,low}(t,x) &:= \F^{-1}[ \hat{K_j}(t,\cdot)\chi_{low}(\cdot) ],
			\\
		K_{j,mid}(t,x) &:= \F^{-1}[ \hat{K_j}(t,\cdot)\chi_{mid}(\cdot) ], 
		\\
		K_{j,high}(t,x) &:= \F^{-1}[ \hat{K_j}(t,\cdot)\chi_{high}(\cdot) ],
			\\
			K_{j,mh}(t,x) &:=	K_{j,mid}(t,x) +K_{j,high}(t,x), 
								\end{aligned}
			\end{equation*}			
	for $j = 0,1$.  
	Dividing the kernel into  
	\begin{equation*}
	\begin{aligned}
	K_j
		=K_{j,low} + K_{j,mid }  + K_{j,high}	= 	K_{j,l} +	K_{j,mh} 
	\end{aligned}
	\end{equation*}
	for $j = 0,1$, 	we estimate each part.  							
				
					\subsection{Estimate of the kernels for low frequency part}
				
			In this subsection, we consider low frequency region: 
	$|\xi| \leq 2^{{ - \frac{2}{1 - 2 \sigma}}}$.  				
	
				\begin{lemma}\label{weightlemma}
			Let $\alpha > -\frac{n}{2}$ and $\beta > 0$.  
			Let $a >  2^{{ - \frac{2}{1 - 2 \sigma}}}$.    
			Let $g(t,\rho)$ be a smooth function on $[0,\infty) \times (0,a)$ satisfying 
			\begin{align}\label{fg}
			& 
			\left|\pd{^k}{\rho^k} g(t,\rho)
		  \right| 						
			\lesssim
			\rho^{\alpha - k}
				e^{-\frac{1}{2}\rho^\beta t}
			\end{align}
			on $[0,\infty) \times (0,a)$ for every $k = 0,1,\cdots$.  Put
			$$
			K(t,x):= \F^{-1}[g(t,|\xi|) \chi_{low}](x). 
			$$
			Then for every $q_j \in [1,2)$ $(j = 0,1)$ and $\vartheta   \in [ 0,\frac{n}{2} + \alpha)$ satisfying					\begin{equation}\label{qass}
			\frac{1}{q_1} \ge \frac{1}{2} + \frac{\vartheta - \alpha  }{n}, \;
			\frac{1}{q_2} \ge \frac{1}{2} - \frac{\alpha}{n}, \;
			\end{equation}	
the following holds. 
\begin{align}\label{Kgamma0}
\Norm{| x |^\vartheta    \left(K(t,\cdot) * \varphi(\cdot) \right)}{{2}}
\lesssim
&\langle t \rangle ^{{ \frac{1}{\beta}
		\left(-n(\frac{1}{q_1} - \frac{1}{2}) + \vartheta   - \alpha \right) }} \|\varphi\|_{q_1}^\prime
\nonumber	\\
&+
\langle t \rangle ^{{ \frac{1}{\beta}
		\left(-n(\frac{1}{q_2} - \frac{1}{2}) - \alpha \right) }}	
\| | x |^\vartheta   \varphi\|_{q_2}^\prime,		
\end{align}	
where $\| \cdot \|^\prime_q$ denote 
\begin{equation}\label{normprime}
\| \cdot \|^\prime_q =
\begin{cases}
\|\cdot \|_1 \quad &\text{if} \quad q = 1
\\
\|\cdot \|_{q,2} \quad &\text{if} \quad q \in (1,2]. 
\end{cases}
\end{equation}				
\end{lemma}
	
Before proving Lemma \ref{weightlemma}, we state two corollaries:
\begin{Cor}\label{Kweightlemma}
					Let $\alpha > -\frac{n}{2}$ and $\beta > 0$.  
						Let $a >  2^{{ - \frac{2}{1 - 2 \sigma}}}$.    
			Let $\upsilon$ and $\lambda$ be smooth functions on some interval $(0,a)$ such that
					\begin{align}
					|\upsilon^{(j)}(\rho) | &\lesssim \rho^{\alpha -j}, \label{fassump} \\
					|\lambda^{(j)}(\rho) | &\lesssim \rho^{\beta -j}, \quad 
					-\lambda(\rho) \sim \rho^{\beta}
					\label{gassump}			
					\end{align}
					on $[0,\infty) \times (0,a)$ for every $j = 0,1,\cdots$.  Put 
				\begin{equation}\label{Kdef}
					K(t,x):= \F^{-1}[\upsilon(|\xi|) e^{\lambda(|\xi|)t} \chi_{low}]. 
					\end{equation}
				Then the conclusion of Lemma \ref{weightlemma} holds.   	
\end{Cor}
							
				In fact, we easily see	that
					\begin{align}\label{cor1}
					& 
					\left|\pd{^k}{\rho^k} 
									\left(	\upsilon(\rho)e^{\lambda(\rho)t}\right)
					\right| 						
					\lesssim
					\rho^{\alpha - k}
										\left(\sum_{j=0}^k (\rho^{\beta}t)^j \right)
					e^{-\rho^\beta t}
					\lesssim
					\rho^{\alpha - k}
									e^{-\frac{1}{2}\rho^\beta t}
					\end{align}
					on $[0,\infty) \times (0,a)$ for every $k = 0,1,\cdots$.   
					Hence, $g(t,\rho) = \upsilon(\rho)e^{\lambda(\rho)t}$ satisfies the assumption 	
					\eqref{fg} of Lemma \ref{weightlemma}, and thus the conclusion holds.  
		
\begin{Cor}\label{Kweightlemmadiff}
	Let $\alpha, \beta, \gamma$ be numbers such that $\alpha - \beta + \gamma > -\frac{n}{2}$, $\beta > 0$ and $\gamma > 0$.  
		Let $a >  2^{{ - \frac{2}{1 - 2 \sigma}}}$.    
	Let $\upsilon$ and $\lambda$ be smooth functions on $(0, a)$ such that
	\begin{align}
	|\upsilon^{(j)}(\rho) | &\lesssim \rho^{\alpha -j}, \label{fassumpdiff} \\
	|\lambda^{(j)}(\rho) | &\lesssim \rho^{\beta -j}, \quad 
	-\lambda(\rho) \sim \rho^{\beta}\label{gassumpdiff}	
	\\	
	|\mu^{(j)}(\rho) | &\lesssim \rho^{\gamma -j}, \quad 
	-\mu(\rho) \sim \rho^{\gamma}
	\label{muassump}		
	\end{align}
	on $(0,a)$ for every $j = 1,2,\cdots$.  Put
	$$
	K(t,x):= \F^{-1}[\upsilon(|\xi|) e^{\lambda(|\xi|)t} 
	(1 - e^{\mu(|\xi|)t})\chi_{low}]. 
	$$
Then for every $q_j \in [1,2)$ $(j = 1,2)$ and $\vartheta   \in [ 0,\frac{n}{2} + \alpha - \beta + \gamma)$ satisfying
\begin{equation}\label{qassdiff}
\frac{1}{q_1} \ge \frac{1}{2} + \frac{\vartheta - \alpha + \beta - \gamma }{n}, \;
\frac{1}{q_2} \ge \frac{1}{2} + \frac{-\alpha + \beta - \gamma}{n}, \;
\end{equation}	
the following holds. 	
	\begin{align}\label{Kgamma0diff}
\Norm{| x |^\vartheta    \left(K(t,\cdot) * \varphi(\cdot) \right)}{{2}}
\lesssim
&\langle t \rangle ^{{ \frac{1}{\beta}
		\left(-n(\frac{1}{q_1} - \frac{1}{2}) + \vartheta   - \alpha + \beta - \gamma \right) }} \|\varphi\|_{q_1}^\prime
\nonumber	\\
&+
\langle t \rangle ^{{ \frac{1}{\beta}
		\left(-n(\frac{1}{q_2} - \frac{1}{2}) - \alpha +  \beta - \gamma \right) }}	
\| | x |^\vartheta   \varphi\|_{q_2}^\prime. 		
\end{align}	
\end{Cor}

	\begin{remark}
D'Abbicco and Ebert \cite{AE1} considered  the  kernels:
	$$	K(t,x) = \F^{-1}[\upsilon(|\xi|) e^{\lambda(|\xi|)t} \chi_{low}](x),
	$$
	 where $\upsilon$ and $\lambda$ 
	satisfy the assumptions \eqref{fassump} and \eqref{gassump} for $\alpha > -1$ (see \cite[Lemma 3.1]{AE1}), and
$$
K(t,x) = \F^{-1}\left[\upsilon(|\xi|) e^{\lambda(|\xi|)t} 
\frac{1 - e^{\mu(|\xi|)t}}{\mu(|\xi|)t}\chi_{low}\right](x), 
$$
where 
$\upsilon$, $\lambda$ and $\mu$ satisfy \eqref{fassumpdiff}, \eqref{gassumpdiff} and \eqref{muassump}
for $\alpha > -1, \beta > 0, \gamma > 0$ (see \cite[Lemma 3.2]{AE1}),
and showed $L^p-L^q$ estimates of $\varphi \mapsto K(t,\cdot)*\varphi$ for $1 \le p \le q \le \infty$ such that
\begin{enumerate}
	\item $p \neq q$ if $\alpha = 0$ and $\upsilon$ is not a constant, 
	\item $\frac{1}{p} - \frac{1}{q} \ge - \frac{\alpha}{n}$  if $\alpha \in (-1,0)$,
\end{enumerate}
by using the description of kernels by Bessel functions. 

	In this paper, we show weighted $L^2$ estimates of $K(t,\cdot)*\varphi$ 
in a way different from \cite{AE1} by employing Lorentz spaces.
	\end{remark}

		\begin{proof}[Proof of Corollary \ref{Kweightlemmadiff}]		
			By the Leibniz rule, we have 
			\begin{align}
			&					\pd{^k}{\rho^k}
			\left(	\upsilon(\rho)e^{\lambda(\rho)t}(1 - e^{\mu(\rho)t})			
			\right)
			= \sum_{j=0}^k C_{k,j} 		
			\pd{^j}{\rho^j} (\upsilon(\rho)e^{\lambda(\rho)t})
			\pd{^{k-j}}{\rho^{k-j}}(1 - e^{\mu(\rho)t}). 
				\label{cor2}\end{align}
			By the assumption \eqref{muassump}, we have
			\begin{align}\label{cor3}
	\left|	\pd{^{k-j}}{\rho^{k-j}}(1 - e^{\mu(\rho)t})	\right| 
	& = \left| \pd{^{k-j}}{\rho^{k-j}} e^{\mu(\rho)t}\right|
	\nonumber \\
	\lesssim & \rho^{-(k-j)}
			\left(\sum_{i=1}^{k-j} (\rho^{\gamma}t)^i \right)
			e^{-\rho^\gamma t}	
			= \rho^{-k+j} \rho^\gamma t
			\sum_{i=0}^{k-j-1}(\rho^{\gamma}t)^i
			e^{-\rho^\gamma t}
		\nonumber		\\
			 \lesssim & \rho^{-k+j+\gamma}  t	e^{-\frac{\rho^\gamma t}{2}}
			\le \rho^{-k+j+\gamma}  t,
			\end{align}
			if $j \le k - 1$,  and
			\begin{equation}\label{cor4}
	\left|	\pd{^{k-j}}{\rho^{k-j}}(1 - e^{\mu(\rho)t})	\right| 
	=		| 1 - e^{\mu(\rho)t}| = |\mu(\rho)t  e^{\theta \mu(\rho)t} | \lesssim \rho^\gamma t,
			\end{equation}
			with $\theta \in (0,1)$ if $j = k$.  
			From  \eqref{cor2}, \eqref{cor1} with $k$ replaced by $j$, \eqref{cor3} and \eqref{cor4}, 
			it follows that   
			\begin{align*}
				& 
		\left|
		\pd{^k}{\rho^k} 
		\left(	\upsilon(\rho)e^{\lambda(\rho)t}(1 - e^{\mu(\rho)t})			
		\right)
		\right| 						
		 \lesssim \rho^{\alpha + \gamma - k}t	
		e^{-\frac{1}{2}\rho^\beta t}
		\lesssim \rho^{\alpha + \gamma - \beta - k}	
			e^{-\frac{1}{4}\rho^\beta t}
			\end{align*}
			on $(0,a)$.  
			Hence, $g(t,\rho) = \upsilon(\rho)e^{\lambda(\rho)t}(1 - e^{\mu(\rho)t})$ satisfies  the assumption 			
			 \eqref{fg} with $\alpha $ replaced by $\alpha - \beta + \gamma$, and therefore, Lemma \ref{weightlemma} implies the assertion.  
		\end{proof}

	Now we prove Lemma \ref{weightlemma}.  
			\begin{proof}[Proof of Lemma \ref{weightlemma}]		
												
	(Step 1)  Let $k$ be a non-negative integer and $\nu \in (0,\infty)$.  We  show that
				\begin{equation}	\label{Knu}
				\begin{aligned}               
				\left\|
				(-\Delta)^{\frac{k}{2}}  \hat K(t,\cdot)				
				\right\|_{{{\nu} }}						
				&	\lesssim	
				\langle t \rangle^{\frac{1}{\beta}(-\alpha + k - \frac{n}{\nu})}
							\end{aligned}
				\end{equation}
for every $t \ge 0$ if  $(-\alpha+ k)\nu < n$, and
					\begin{equation}	\label{Kinfty}
		\begin{aligned}
		\left\|
		(-\Delta)^{\frac{k}{2}}  \hat K(t,\cdot)				
		\right\|_{{{\nu,\infty} }}						
		&	\lesssim	
		\langle t \rangle^{\frac{1}{\beta}(-\alpha + k - \frac{n}{\nu})} = 1,		
		\end{aligned}
		\end{equation}
for every $t \ge 0$ if  $(-\alpha+ k)\nu =n$. 
				
		First, we assume that $(-\alpha+ k)\nu < n$. 
					Using the assumption \eqref{fg} and changing variables by $t^{1/\beta}\rho = r$, 
						we have		
			\begin{align}
			\left\|	(-\Delta)^{\frac{k}{2}}  \hat K(t,\cdot)			
			\right\|_{{{\nu} }}^{\nu}	
			&\lesssim 
			\left\|	\sum_{|\gamma| = k} \partial_{\xi}^\gamma  
			\hat K(t,\xi)			
			\right\|_{{\nu}}^{\nu}	
			\nonumber \\		
			&\lesssim	
			\int_{0}^a \rho^{(\alpha - k)\nu}
			e^{{-\frac{1}{2}\nu \rho^{\beta} t }}
			\rho^{n-1} d\rho	
			\label{K1+der1}		\\						
			&	= 
			t^{\frac{1}{\beta}((-\alpha + k) \nu - n)}
			\int_{0}^{t a} r^{-(-\alpha + k)\nu}
			e^{{-\frac{1}{2}\nu r^{\beta}  }}
				r^{n-1} dr				
			\nonumber			\\						
			&	\lesssim	
			t^{\frac{1}{\beta}((-\alpha + k) \nu - n)}. 
			\label{K1+der2}		\end{align}
				By \eqref{K1+der1}, we have 
			\begin{align*}
			\left\|	(-\Delta)^{\frac{k}{2}}  \hat  K(t,\cdot)				
			\right\|_{{{\nu} }}^{\nu}							
			&
			\lesssim	
				\int_{0}^a \rho^{(\alpha - k)\nu}
			\rho^{n-1} d\rho	< \infty,
			\end{align*}
					for $0 < t \le 1$,
	which together with \eqref{K1+der2} yields 	\eqref{Knu}.   

Next we assume that  $(-\alpha+ k)\nu = n$. 
			By \eqref{fg}, we have	
	\begin{align*}
			\left| \partial_{\xi}^\gamma \hat K(t,\xi)			
		\right|
		&
		\lesssim	
					|\xi|^{\alpha - |\gamma|}
									\end{align*}
					for every $t \ge 0$.  		Hence, 
						\begin{equation*}
						s 	\mu ( \{ \xi; | \partial_{\xi}^\gamma  \hat K(t,\xi)|	 > s \} )^{{\frac{1}{\nu}}}
						\lesssim s^{{1 -\frac{n}{(-\alpha +k)\nu} }} = 1,
						\end{equation*}
					if $|\gamma| = k$,  and therefore, 
									\begin{equation}	
						\begin{aligned}
						\left\|
						(-\Delta)^{\frac{k}{2}}  \hat K(t,\cdot)
						\right\|_{{\nu,\infty}}		
					\lesssim 
					\sum_{|\gamma| = k}	\left\|	 \partial_{\xi}^\gamma  
					\hat K(t,\xi)			
									\right\|_{\nu,\infty}	
					\nonumber 					
						&	\lesssim	1, 
						\end{aligned}
						\end{equation}
						for every $t \ge 0$,  
			that is, \eqref{Kinfty} holds in the case $(-\alpha+ k)\nu = n$.

					(Step 2)  Let $\vartheta   \in [0,\delta]$ and $\kappa \in (1,2]$.  
					We prove that 
					\begin{equation}
					\label{Kweight11}
					\begin{aligned}
								\Norm{|\cdot|^{\vartheta  } K(t,\cdot) }{{{\kappa, 2}}}
					&	\lesssim	
										\langle t \rangle^{\frac{1}{^\beta}(\vartheta - \alpha  - 	n( 1 - \frac{1}{\kappa}))}, 			
					\end{aligned}
					\end{equation}
				for every $t > 0$	if $-\alpha +\vartheta   < n(1 - \frac{1}{\kappa})$, and 					
				\begin{equation}\label{Kweight1}
				\begin{aligned}
				\Norm{|\cdot|^{\vartheta  } K(t,\cdot) }{{{\kappa, \infty}}}
				&	\lesssim	
					\langle t \rangle^{\frac{1}{^\beta}(\vartheta - \alpha  - 	n( 1 - \frac{1}{\kappa}))} = 1, 
				\end{aligned}
				\end{equation}
				for every $t > 0$	if	$-\alpha +\vartheta   = n(1 - \frac{1}{\kappa})$. 
				
						Let $\omega$ be a non-negative number 
						such that $\vartheta   + \omega$ becomes an integer and that
					\begin{equation}  \label{omega}
				n\left( \frac{1}{\kappa} -\frac{1}{2} \right) \le \omega < \frac{n}{\kappa}. 
					\end{equation}
					Since $n \ge 2$, we can take $\omega$ satisfying above conditions.  
					Let $\nu$ and its dual exponent $\nu^\prime$ be the numbers defined by
					\begin{equation}\label{kappa_nu}
					\frac{1}{\kappa} = \frac{\omega}{n}	 + 1 - \frac{1}{\nu}
						= \frac{\omega}{n} + \frac{1}{\nu^\prime}. 
					\end{equation}
		Since $0 < 1/\nu^\prime = 1/\kappa - \omega/n \le 1/2$ by the assumption \eqref{omega}, 
	we have 
	\begin{equation}\label{nu<nuprime}
	1 < \nu \le 2 \le \nu^\prime < \infty. 
	\end{equation} 
									
Now we prove \eqref{Kweight11}	under the assumption $-\alpha +\vartheta   < n(1 - \frac{1}{\kappa})$.  		  
									By Corollary A and Lemmas A and C together with the relation \eqref{nu<nuprime}, 
										we  have
										\begin{align}
						 \Norm{|\cdot|^{\vartheta  } K(t,\cdot) }{{{\kappa, 2}}}
						&=\Norm{|\cdot|^{-\omega}|\cdot|^{\omega + \vartheta  }	K(t,\cdot) }{{{\kappa, 2}}}
						\lesssim 
						\Norm{|\cdot|^{-\omega}}
						{{{{\frac{n}{\omega},\infty}} }}
							\Norm{	|\cdot|^{\vartheta   + \omega} K(t,\cdot) }{{{\nu^\prime,2}}}
						\nonumber	\\
						&
						\lesssim 
						\Norm{	|\cdot|^{\vartheta   + \omega}	K(t,\cdot) }{{{\nu^\prime,2}}}
						= 	\Norm{	
							\F^{-1}[(-\Delta)^{\frac{\vartheta  +\omega}{2}} \hat K(t,\cdot)]  }{{{\nu^\prime,2}}}
						\nonumber	\\					&					
					\lesssim	\Norm{	
						(-\Delta)^{\frac{\vartheta  +\omega}{2}} \hat K(t,\cdot) }{{{\nu,2}}}
						\nonumber	\\
						&\lesssim	\Norm{	(-\Delta)^{\frac{\vartheta  +\omega}{2}} \hat K(t,\cdot) }{{{\nu,\nu}}}
							=\Norm{						(-\Delta)^{\frac{\vartheta  +\omega}{2}} \hat K(t,\cdot) }{{{\nu}}}. 							
					\label{KLmu}	\end{align}														
								Since the assumption $-\alpha +\vartheta   < n(1 - \frac{1}{\kappa})$ and \eqref{kappa_nu} imply that  $- \alpha + \vartheta + \omega < \frac{n}{\nu}$,  we can take   $k = \vartheta   + \omega$ in \eqref{Knu}.  
					 Substituting the inequality into \eqref{KLmu}, we obtain  \eqref{Kweight11}.  
			
				Next we prove \eqref{Kweight1}	under the assumption $-\alpha +\vartheta   = n(1 - \frac{1}{\kappa})$.  
				By Corollary A and Lemma C together with  \eqref{nu<nuprime}, we have 

\begin{align}
			\Norm{|\cdot|^{\vartheta  } K(t,\cdot) }{{{\kappa,\infty}}}
			&=\Norm{|\cdot|^{-\omega}
							|\cdot|^{\omega + \vartheta  }
					K(t,\cdot) }{{{\kappa,\infty}}}
\nonumber \\
&	\lesssim 
				\Norm{	|\cdot|^{\vartheta   + \omega}
					K(t,\cdot) }{{{\nu^\prime,\infty}}}
				= 	\Norm{	
					\F^{-1}[(-\Delta)^{\frac{\vartheta  +\omega}{2}} \hat K(t,\cdot)] }{{{\nu^\prime,\infty}}}
\nonumber\\
&	\lesssim 	\Norm{	
					(-\Delta)^{\frac{\vartheta  +\omega}{2}} \hat K(t,\cdot) }{{{\nu,\infty}}}. 											\label{kappainfty}
				\end{align}				
					Since the assumption $-\alpha +\vartheta   = n(1 - \frac{1}{\kappa})$ means 
					$ - \alpha + \vartheta  + \omega = \frac{n}{\nu}$, 
					 we can take   $k = \vartheta   + \omega$ in 
				\eqref{Kinfty}.  Substituting the inequality into \eqref{kappainfty}, we obtain \eqref{Kweight1}.  
		
		(Step 3) 			
		We define $r_j \in (1,2]$ by $\frac{1}{q_j} + \frac{1}{r_j} = \frac{3}{2}$ $(j = 1,2)$.  		
		We estimate each term of the right-hand side of  
		\begin{align}\label{final}
		\Norm{|\cdot|^{\vartheta  }(K(t,\cdot) * \varphi) }{{2}} 
		& \lesssim 
		\Norm{(|\cdot|^{\vartheta  } K(t,\cdot) )* \varphi}{{2}}
		+\Norm{K(t,\cdot) *(|\cdot|^\vartheta     \varphi)}{{2}}.  
		\end{align}				
If $q_1 \in (1,2)$,	Lemma D yields			
		\begin{align}\label{Young21}
						\Norm{(|\cdot|^{\vartheta  } K(t,\cdot))* \varphi}{{2}}
									& \lesssim 
			\Norm{|\cdot|^{\vartheta  } K(t,\cdot) }{{{r_1,\infty}}}
			\Norm{\varphi}{{{q_1,2} }}.	
	           				\end{align}	
The assumption \eqref{qass} implies 
$-\alpha +\vartheta   \le  n (\frac{1}{q_1} - \frac{1}{2}) = n(1 - \frac{1}{r_1})$.  
Hence, 	noting that $\Norm{|\cdot|^{\vartheta  } K(t,\cdot) }{{{r_1, \infty}}} \le 
\Norm{|\cdot|^{\vartheta  } K(t,\cdot) }{{{r_1, 2}}}$ and substituting \eqref{Kweight11}  or \eqref{Kweight1} with $(\vartheta  , \kappa) =(\vartheta  , r_1)$ into \eqref{Young21}, we obtain
\begin{align}
\label{Young211}
		\Norm{(|\cdot|^{\vartheta  } K(t,\cdot) )* \varphi}{{2}}
& \lesssim 
\langle t \rangle ^{{ \frac{1}{\beta}
		\left(-n(\frac{1}{q_1} - \frac{1}{2}) + \vartheta   - \alpha \right) }} \|\varphi\|_{q_1}^\prime.
\end{align}
	In the case $q_1 = 1$,  Young's inequality yields
\begin{align}\label{Young333}
\Norm{(|\cdot|^{\vartheta  } K(t,\cdot))* \varphi}{{2}}
& \lesssim 
\Norm{|\cdot|^{\vartheta  } K(t,\cdot) }{2}
\Norm{\varphi}{{{1} }}		
= \Norm{|\cdot|^{\vartheta  } K(t,\cdot) }{{2,2}}
\Norm{\varphi}{{{1} }}.	
\end{align}				
The assumption $\vartheta < \frac{n}{2} + \alpha$ means  
$ - \alpha + \vartheta <  \frac{n}{2} = n(1 - \frac{1}{2})$.   
Hence, substituting \eqref{Kweight11}  
with $(\vartheta  , \kappa) =(\vartheta, 2)$ into  \eqref{Young333},
we see that \eqref{Young211} holds also for $q_1 = 1$.  

We can estimate the second term of \eqref{final} in the same way: 
If $q_2 \in (1,2)$, Lemma D yields
\begin{align}\label{Young22} 
\Norm{K(t,\cdot) *(|\cdot|^\vartheta     \varphi)}{{2}}
& \lesssim 
\Norm{K(t,\cdot)}{{{r_2,\infty} }}
\Norm{|\cdot|^\vartheta     \varphi}{{{q_2,2}}}.  
\end{align}
The assumption \eqref{qass} implies 
$-\alpha \le  n (\frac{1}{q_2} - \frac{1}{2}) = n(1 - \frac{1}{r_2})$.  
Hence, 	substituting \eqref{Kweight11}  or \eqref{Kweight1} with $(\vartheta, \kappa)  =(0, r_2)$ into \eqref{Young22}, we obtain
\begin{align}                 
\label{Young212}
\Norm{K *(|\cdot|^\vartheta     \varphi)}{{2}}
& \lesssim 
\langle t \rangle ^{{ \frac{1}{\beta}
		\left(-n(\frac{1}{q_2} - \frac{1}{2}) - \alpha \right) }}	
\| | x |^\vartheta   \varphi\|_{q_2}^\prime.                            
\end{align}	
In the case $q_2 = 1$,  Young's inequality yields
\begin{align}
\label{Young3333}
\Norm{K(t,\cdot) *(|\cdot|^\vartheta     \varphi)}{{2}}
& \lesssim 
\Norm{K(t,\cdot)}{{2}}
\Norm{|\cdot|^\vartheta     \varphi}{1} 
= \Norm{K(t,\cdot)}{{2,2}}
\Norm{|\cdot|^\vartheta     \varphi}{1}.                            
\end{align}				
Since $-\alpha <  \frac{n}{2} = n (1 - \frac{1}{2})$, we have \eqref{Kweight11}  
with $(\vartheta, \kappa) =(0, 2)$, which together with \eqref{Young3333} yields \eqref{Young212} with $q_2 = 1$. 

Hence, \eqref{Young211} and \eqref{Young212} hold for every case.  
Substituting \eqref{Young211} and \eqref{Young212} into \eqref{final}, we obtain \eqref{Kgamma0}. 
	 		\end{proof}				
	
		\begin{lemma}\label{Kweight}
			Assume that $0 \le s_2 \le s_1 $ and
					$\vartheta \ge 0$  satisfy 
			$\vartheta - s_1 + s_2 < \frac{n}{2} - 2 \sigma$. 	
		If $q_j \in [1,2)$  $(j = 1,2,3,4)$ satisfy 
		\begin{equation}\label{qass2}
			\frac{1}{q_1} \ge \frac{1}{2} + \frac{2\sigma + \vartheta - s_1 + s_2 }{n}, \;
			\frac{1}{q_2} \ge \frac{1}{2} + \frac{2 \sigma - s_1 + s_2}{n}, \;
			\frac{1}{q_3} \ge \frac{1}{2} + \frac{ \vartheta - s_1 + s_2 }{n}, \;
				\end{equation}	
	then the following hold provided the right-hand sides are finite:
		\begin{align}
	&		\label{K1lowest+} 
		\Norm{|\cdot|^{\vartheta}(-\Delta)^{\frac{s_1}{2}}   
			\left(\F^{-1}[\hat K_1^+(t,\cdot) \chi_{low}]  * \varphi \right)}{{2}}
	\nonumber	\\
	& \quad \lesssim 
			\langle t \rangle ^{{ \frac{1}{1 - \sigma} \left(-\frac{n}{2}(\frac{1}{q_1} - \frac{1}{2}) + \frac{\vartheta - s_1 + s_2 }{2} + \sigma \right) }} \| (-\Delta)^{\frac{s_2}{2}}\varphi \|_{{{q_1}}}^\prime
	\nonumber	\\
	&	\qquad	+
			\langle t \rangle ^{{ \frac{1}{1 - \sigma} 
			\left(-\frac{n}{2}(\frac{1}{{q_2}} - \frac{1}{2}) - \frac{s}{2} +  \sigma \right) }} 
			\||\cdot|^{\vartheta}(-\Delta)^{\frac{s_2}{2}}\varphi \|_{{{{q_2}} }}^\prime,	
	\\
		&	\Norm{|\cdot|^{\vartheta} (-\Delta)^{\frac{s_1}{2}}   \left(\F^{-1}[\hat K_1^-(t,\cdot) \chi_{low}]  * \varphi \right)}{{2}}
	\nonumber \\
	& \quad		\lesssim 
			\langle t \rangle ^{{ \frac{1}{ \sigma} \left(-\frac{n}{2}(\frac{1}{q_1} - \frac{1}{2}) + \frac{{\vartheta - s_1 + s_2} }{2} + \sigma \right) }} \| (-\Delta)^{\frac{s_2}{2}}\varphi \|_{{{q_1}}}^\prime
			\nonumber	\\
			& \qquad +
			\langle t \rangle ^{{ \frac{1}{\sigma} \left(-\frac{n}{2}(\frac{1}{{q_2}} - \frac{1}{2}) - \frac{s}{2} +  \sigma \right) }} 
			\||\cdot|^{\vartheta} (-\Delta)^{\frac{s_2}{2}}\varphi \|_{{{{q_2}} }}^\prime,	
			\label{K1lowest-}
	\\		
		&	\Norm{|\cdot|^{\vartheta} (-\Delta)^{\frac{s_1}{2}}   \left(K_{1,low}(t,\cdot) * \varphi \right)}{{2}}
	\nonumber	\\
		&\quad 	\lesssim
			\langle t \rangle ^{{ \frac{1}{1 - \sigma} \left(-\frac{n}{2}(\frac{1}{q_1} - \frac{1}{2}) + \frac{{\vartheta - s_1 + s_2} }{2} + \sigma \right) }} \| (-\Delta)^{\frac{s_2}{2}}\varphi \|_{{{q_1}}}^\prime
			\nonumber	\\
			& \qquad +
			\langle t \rangle ^{{ \frac{1}{1 - \sigma} \left(-\frac{n}{2}(\frac{1}{{q_2}} - \frac{1}{2}) - \frac{s}{2} +  \sigma \right) }} \||\cdot|^{\vartheta} (-\Delta)^{\frac{s_2}{2}}\varphi \|_{{{{q_2}} }}^\prime,
			\label{K1lowweight}		
					\\
			\label{K0lowest+} 
		&	\Norm{|\cdot|^{\vartheta} (-\Delta)^{\frac{s_1}{2}}   \left(\F^{-1}[\hat K_0^+(t,\cdot) \chi_{low}]  * \varphi \right)}{{2}}
		\nonumber	
		\\
			& \quad  \lesssim \langle t \rangle ^{{ \frac{1}{1 - \sigma} \left(-\frac{n}{2}(\frac{1}{q_3} - \frac{1}{2}) + \frac{{\vartheta - s_1 + s_2} }{2}  \right) }} \| (-\Delta)^{\frac{s_2}{2}}\varphi \|_{{{q_3}}}^\prime
			\nonumber	\\
			& \qquad +
			\langle t \rangle ^{{ \frac{1}{1 - \sigma} \left(-\frac{n}{2}(\frac{1}{{q_4}} - \frac{1}{2}) - \frac{s}{2} \right) }} \||\cdot|^{\vartheta} (-\Delta)^{\frac{s_2}{2}}\varphi \|_				{{q_4}}^\prime,	
				\\
			\label{K0lowest-} 
		&	\Norm{|\cdot|^{\vartheta} (-\Delta)^{\frac{s_1}{2}}   \left(\F^{-1}[\hat K_0^-(t,\cdot) \chi_{low}]  * \varphi \right)}{{2}}
		\nonumber		\\
		& \quad 	\lesssim \langle t \rangle ^{{ \frac{1}{\sigma} \left(-\frac{n}{2}(\frac{1}{q_3} - \frac{1}{2}) + \frac{{\vartheta - s_1 + s_2} }{2} + 2 \sigma - 1 \right) }} \| (-\Delta)^{\frac{s_2}{2}}\varphi \|_{{{q_3}}}^\prime
			\nonumber	\\
	& \qquad +
			\langle t \rangle ^{{ \frac{1}{\sigma} \left(-\frac{n}{2}(\frac{1}{{q_4}} - \frac{1}{2} ) - \frac{s}{2}+ 2 \sigma - 1)  \right) }} \||\cdot|^{\vartheta} (-\Delta)^{\frac{s_2}{2}}\varphi \|_{{{{q_4}} }}^\prime,	
			\\
	&		\Norm{|\cdot|^{\vartheta} (-\Delta)^{\frac{s_1}{2}}   \left(K_{0,low}(t,\cdot) * \varphi \right) }{{2}}
	\nonumber
	\\	& \quad 	\lesssim
		 \langle t \rangle ^{{ \frac{1}{1 - \sigma} \left(-\frac{n}{2}(\frac{1}{q_3} - \frac{1}{2}) + \frac{{\vartheta - s_1 + s_2} }{2}  \right) }} \| (-\Delta)^{\frac{s_2}{2}}\varphi \|_{{{q_3}}}^\prime
			\nonumber	\\
			& \qquad +
			\langle t \rangle ^{{ \frac{1}{1 - \sigma} \left(-\frac{n}{2}(\frac{1}{{q_4}} - \frac{1}{2}) - \frac{s}{2} \right) }} \||\cdot|^{\vartheta} (-\Delta)^{\frac{s_2}{2}}\varphi \|_{{{{q_4}} }}^\prime,	
			\label{low0weight}		
		\end{align}		
		where $\| \cdot \|^\prime_q$ is defined by \eqref{normprime}. 
	\end{lemma}		
	
	\begin{proof}
Since  
\begin{equation*}
\begin{aligned}
&\Norm{|\cdot|^{\vartheta}(-\Delta)^{\frac{s_1}{2}}   
	\left(\F^{-1}[\hat K(t,\cdot) \chi_{low}]  * \varphi \right)}{{2}}
\\
&=\Norm{|\cdot|^\vartheta 
	\F^{-1} [|\xi|^{s_1-s_2} \hat K(t,\xi) \chi_{low} |\xi|^{s_2} \hat \varphi(\xi)]  
	}{{2}}
\\
&= \Norm{|\cdot|^{\vartheta}(-\Delta)^{\frac{s_1-s_2}{2}}   
\F^{-1} [\hat K(t,\cdot)\chi_{low}]*(-\Delta) ^{{\frac{s_2}{2} }} \varphi }{2},
\end{aligned}     
\end{equation*}		
the conclusion reduces to the case $s_2 = 0$ by taking $s_1 - s_2$ and 
$(-\Delta) ^{s_2} \varphi $ as $s_1$ and $\varphi$ respectively.  

Let $\lambda_\pm$ be the functions defined by \eqref{lambda_def1}.  
From \eqref{lambda_def2}, it follows that 
\begin{align}
-2	\rho^{2 - 2 \sigma} \le 
\lambda_+(\rho) 
&= -\frac{2  \rho^{2 - 2 \sigma}}
{1 +\sqrt{1-4\rho^{2(1 - 2\sigma)} } } 
\le  - \rho^{2 - 2 \sigma},  
\label{lambda+}	\quad 
\\
- \rho^{2 \sigma} \le \lambda_- (\rho) & \le  -\frac{1}{2} \rho^{2 \sigma},
\label{lambda-}
\\
\lambda_+(\rho) - \lambda_- (\rho) & 
= \rho^{2 \sigma} \sqrt{1-4\rho^{2(1 - 2\sigma)} } 
\sim \rho^{2 \sigma},
\label{lambda+-}
\end{align}
on the support of $\chi_{low}$. 
	 
We first prove	\eqref{K1lowest+}.  
It is written that
\begin{align*}
	(-\Delta)^{\frac{s_1}{2}}  
	\left(\F^{-1}[\hat K_1^+(t,\cdot) \chi_{low}]  * \varphi \right)
	&= \F^{-1}[|\xi|^{s_1} \hat K_1^+(t,\cdot) \chi_{low}]* 
		\varphi,
\end{align*}
where $K_1^+$ is defined by \eqref{K1+-def}.  
By definition,  	
$K(t,x) = \F^{-1}[|\xi|^{s_1} \hat K_1^+(t,\cdot) \chi_{low}]$		
has the form \eqref{Kdef} with 
		$$
		\upsilon(\rho) = \frac{\rho^{s_1}}{\lambda_+(\rho) - \lambda_-(\rho)} \chi_{low}, \quad
		\lambda(\rho) = \lambda_{+}(\rho).  
		$$
By using \eqref{lambda+} and \eqref{lambda+-}, we easily see that 	 
$\upsilon$ and $\lambda$ above	satisfy  the assumption 
of Corollary \ref{Kweightlemma} with $\alpha =  s_1 - 2 \sigma$, $\beta = 2(1 -\sigma)$.   
The assumption $n \ge 2$ and $0 < 2 \sigma < 1$ implies $\alpha > -\frac{n}{2}$ and  $\beta > 0$, 
that is, $\alpha$ and $\beta$ satisfy the assumption of Corollary \ref{Kweightlemma}.   
Definition of $\alpha$ and \eqref{qass2} means \eqref{qass} (here we note that we assume $s_2 = 0$).  
Hence, applying  
Corollary \ref{Kweightlemma}, we obtain \eqref{K1lowest+}.     	
		
$
K(t,x) = \F^{-1}[|\xi|^{s_1} \hat K_1^-(t,\cdot) \chi_{low}]
$		
$(K_1^-$ is defined by \eqref{K1+-def}) has the form \eqref{Kdef} with 	
		$$
		\upsilon(\rho) = \frac{-\rho^{s_1}}{\lambda_+(\rho) - \lambda_-(\rho)} \chi_{low}, \quad
		\lambda(\rho) = \lambda_{-}(\rho). 
		$$
	By using \eqref{lambda+}, we easily see that $\upsilon$ and $\lambda$ above	satisfy the assumption of Corollary \ref{Kweightlemma} for $\alpha =s_1 - 2 \sigma (> -\frac{n}{2}), \beta = 2 \sigma (> 0 )$  and therefore,  
		\eqref{K1lowest-} holds in the same way as in the proof of \eqref{K1lowest+}. 
		 
		Since $\sigma < 1- \sigma$  by the assumption that $\sigma \in (0,1/2)$, 
the estimate \eqref{K1lowweight} follows from \eqref{K1lowest+} and \eqref{K1lowest-}.  
		
$
K(t,x) = \F^{-1}[|\xi|^{s_1} \hat K_0^+(t,\cdot) \chi_{low}]  	
$	
$(K_0^+$ is defined by \eqref{K0+-def}) 
 has the form \eqref{Kdef} with 	
		$$
		\upsilon(\rho) = \frac{-\rho^{s_1} \lambda_-(\rho)}{\lambda_+(\rho) - \lambda_-(\rho)} \chi_{low}, \quad
		\lambda(\rho) =  \lambda_{+}(\rho),
		$$
	By using \eqref{lambda-}, we easily see that $\upsilon$ and $\lambda$ satisfy  the assumption of Corollary \ref{Kweightlemma} with  $\alpha = s_1, \beta = 2(1 -\sigma)$.   
	Definition of $\alpha$ and \eqref{qass2} means the condition \eqref{qass} of $q_1$ for $q_1 = q_3$.  Since $q_4 < 2$, the condition of $q_2$ of \eqref{qass} holds for $q_2 = q_4$.  Hence,  
	we can apply Corollary \ref{Kweightlemma} to obtain \eqref{K0lowest+}. 
			
Kernel $K = \F^{-1}[|\xi|^{s_1} \hat K_0^-(t,\cdot) \chi_{low}]$ $(K_0^-$ is defined by \eqref{K0+-def}) has the form \eqref{Kdef} with 	
		$$
		\upsilon(\rho) = \frac{\rho^{s_1} \lambda_+(\rho)}{\lambda_+(\rho)   - \lambda_-(\rho)} \chi_{low}, \quad 
		\lambda(\rho) =  \lambda_{-}(\rho),
		$$
	By using \eqref{lambda+} -- \eqref{lambda+-}, we easily see that $\upsilon$ and $\lambda$ above	satisfy  the assumption of Corollary \ref{Kweightlemma} with  
		$\alpha =s_1 + 2(1- 2 \sigma) (> 0), \beta = 2 \sigma (> 0)$.   
		Definition of $\alpha$ and \eqref{qass2} means the condition of $q_1$ of \eqref{qass} for $q_1 = q_3$.  The condition of $q_2$ of \eqref{qass} holds for $q_2 = q_4$ in the same reason as above.  Hence, \eqref{K0lowest-} holds by  \eqref{Kgamma0}.  
		 		
Since $\sigma \in (0,1/2)$, inequality \eqref{low0weight} follows from  \eqref{K0lowest+} and \eqref{K0lowest-}. 
	\end{proof}

	\begin{lemma}\label{Kweightdiff}
			Let 	$\vartheta \ge 0$ and $0 \le s_2 \le s_1 $ 
				 such that 
		$\vartheta - s_1 + s_2 < \frac{n}{2} - 2 \sigma $.  		
		Assume that $q_j \in [1,2)$  $(j = 1,2,3,4)$ satisfy 
		\begin{equation}\label{qass2diss}
		\frac{1}{q_1} \ge \frac{1}{2} + \frac{2\sigma + \vartheta - s_1 + s_2 }{n}, \;
		\frac{1}{q_2} \ge \frac{1}{2} + \frac{2 \sigma - s_1 + s_2}{n}, \;
		\frac{1}{q_3} \ge \frac{1}{2} + \frac{ \vartheta - s_1 + s_2 }{n}.  \;
			\end{equation}		
		Then the following hold provided the right-hand sides are finite:
		\begin{align}
&\Norm{|\cdot|^{\vartheta} (-\Delta)^{\frac{s_1}{2}}   
			\left(\F^{-1}[(\hat K_1^+(t,\cdot) 
			- |\xi|^{-2 \sigma} e^{- |\xi|^{2(1 - \sigma)} t} )\chi_{low}]  * \varphi \right)}{{2}}
	\nonumber	\\
		&\lesssim 
		\langle t \rangle ^{{ \frac{1}{1 - \sigma} \left(-\frac{n}{2}(\frac{1}{q_1} - \frac{1}{2}) + \frac{{\vartheta - s_1 + s_2} }{2} + 3\sigma - 1 \right) }} 
		\| (-\Delta)^{\frac{s_2}{2}}\varphi \|_{{{q_1}}}^\prime
	\nonumber \\
 & \qquad		+	\langle t \rangle ^{{ \frac{1}{1 - \sigma} \left(-\frac{n}{2}(\frac{1}{{q_2}} - \frac{1}{2}) +  3\sigma - 1\right) }} 
	\||\cdot|^{\vartheta} (-\Delta)^{\frac{s_2}{2}}\varphi \|_{q_2}^\prime,	
		\label{K1lowest+diff} 
			\\
			&\Norm{|\cdot|^{\vartheta} (-\Delta)^{\frac{s_1}{2}}   
			\left(\F^{-1}[(\hat K_0^+(t,\xi) 
			-  e^{- |\xi|^{2(1 - \sigma)} t} )\chi_{low}]  * \varphi \right)}{{2}}
		\nonumber	\\
		&\lesssim 
		\langle t \rangle ^{{ \frac{1}{1 - \sigma} \left(-\frac{n}{2}(\frac{1}{q_1} - \frac{1}{2}) + \frac{{\vartheta - s_1 + s_2} }{2} + 2\sigma - 1 \right) }} \|(-\Delta)^{\frac{s_2}{2}}\varphi\|_{q_1}^\prime
	\nonumber \\
& \qquad			+
		\langle t \rangle ^{{ \frac{1}{1 - \sigma} \left(-\frac{n}{2}(\frac{1}{{q_2}} - \frac{1}{2}) +  2\sigma - 1 \right) }} 
		\||\cdot|^{\vartheta} (-\Delta)^{\frac{s_2}{2}}\varphi\|_{q_2}^\prime,
	\label{K0lowest+diff} 
				\end{align}	
where $\| \cdot \|^\prime_q$ is defined by \eqref{normprime}. 	
	\end{lemma}		

\begin{proof}
We first prove \eqref{K1lowest+diff}.  
	Let $\lambda_\pm$ be the functions defined by \eqref{lambda_def2}.  
By the same reason as in the proof of Lemma \ref{Kweight}, we may assume that $s_2 = 0$.  
It follows from the definition that 		
\begin{align}
	&|\xi|^{s_1} \hat K_1^+(t,\xi) 
	- |\xi|^{s_1 -2 \sigma} e^{- |\xi|^{2(1 - \sigma)} t} 
	\nonumber \\
	&	= |\xi|^{s_1 -2 \sigma}
	\Bigg(	\frac{1}{\sqrt{1-4|\xi|^{2(1 - 2\sigma)}}}
	\exp \Big(\frac{- 2 	|\xi|^{2(1 - \sigma)}t}{1 + \sqrt{1-4|\xi|^{2(1 - 2\sigma)}}	}
							\Big)
	- e^{- |\xi|^{2(1 - \sigma)} t} 
			\Bigg)		
\nonumber			\\
			&\quad + |\xi|^{s_1 -2 \sigma}
\left( \frac{1}{\sqrt{1-4|\xi|^{2(1 - 2\sigma)}}} - 1 \right)
- e^{- |\xi|^{2(1 - \sigma)} t} 
\nonumber		 \\
		&	= \frac{|\xi|^{s_1 -2 \sigma} e^{- |\xi|^{2(1 - \sigma)} t} }{\sqrt{1-4|\xi|^{2(1 - 2\sigma)}}}
		\left(
				\exp \Big(\frac{ - 4|\xi|^{2(2 - 3 \sigma) } t}{(1+\sqrt{1-4|\xi|^{2(1 - 2\sigma)}})^2	}  \Big) -1
		\right)
\nonumber		\\
		&\quad -
	\frac{4 |\xi|^{s_1 + 2(1 - 2 \sigma) } e^{- |\xi|^{2(1 - \sigma)} t}  }
	{(1 + \sqrt{1-4|\xi|^{2(1 - 2\sigma)}})\sqrt{1-4|\xi|^{2(1 - 2\sigma)}}}. 
\nonumber
	\\
	&:= M_{1,1} + M_{1,2} \quad \text{(we put)}.  
		\label{weight_diss}	\end{align}

We easily see that 
		$$
		\upsilon(\rho) =  
		\frac{- \rho^{s_1 -2 \sigma}}{\sqrt{1-4\rho^{2(1 - 2\sigma)}}},
		\quad
		\lambda(\rho) = -\rho^{2(1 - \sigma)},
		\quad
		\mu(\rho) = \frac{4\rho^{2(2 - 3 \sigma) }}
		{(1+\sqrt{1-4\rho^{2(1 - 2\sigma)}})^2	}
		$$
		satisfy  the assumption 
		\eqref{fassump} 
			-- \eqref{muassump} 
				of Corollary \ref{Kweightlemmadiff} with $\alpha =  s_1 - 2 \sigma, \beta = 2(1 -\sigma), 
		\gamma = 2(2 - 3 \sigma)$.  
			Then the assumption $n \ge 2$ and $2 \sigma < 1$ implies 
		$\alpha - \beta + \gamma = s_1 + 2(1 - 3\sigma) > -\frac{n}{2}$ and $\beta > 0$, that is, $\alpha$, $\beta$ and $\gamma$ satisfy the assumption of Corollary \ref{Kweightlemmadiff}.  		 
		The assumption \eqref{qass2diss} and the definition of $\alpha, \beta, \gamma$ above imply
\begin{align*}
&\frac{1}{q_1} \ge \frac{1}{2} + \frac{2\sigma + \vartheta - s_1 }{n}
\ge \frac{1}{2} + \frac{ \vartheta - \alpha + \beta - \gamma }{n}, \;
\\
&\frac{1}{q_2} \ge \frac{1}{2} + \frac{2 \sigma - s_1 + s_2}{n} 
\ge \frac{1}{2} + \frac{ - \alpha + \beta - \gamma }{n}, \;
\end{align*}	
that is,  \eqref{qassdiff} holds.  Hence, we can apply Corollary \ref{Kweightlemmadiff} for the above choice to obtain
\begin{equation}\label{M11}
\begin{aligned}
|M_{1,1}| \lesssim 
& \langle t \rangle ^{{ \frac{1}{1 - \sigma} \left(-\frac{n}{2}(\frac{1}{q_1} - \frac{1}{2}) + \frac{\vartheta - s_1  }{2} + 2\sigma - 1  \right) }} 
\| (-\Delta)^{\frac{s_2}{2}}\varphi \|_{{q_1}}^\prime
\\
& \qquad			+
\langle t \rangle ^{{ \frac{1}{1 - \sigma} \left(-\frac{n}{2}(\frac{1}{{q_2}} - \frac{1}{2}) - \frac{s}{2} +  2\sigma - 1  \right) }} 
\||\cdot|^{\vartheta} (-\Delta)^{\frac{s_2}{2}}\varphi\|_{q_2}^\prime. 	
\end{aligned}\end{equation}

We also see that 
$$
\upsilon(\rho) =  
\frac{4 \rho^{s_1 + 2(1 - 3 \sigma) }}{(1+\sqrt{1-4\rho^{2(1 - 2\sigma)}})\sqrt{1-4\rho^{2(1 - 2\sigma)}}},
\quad
\lambda(\rho) = -\rho^{2(1 - \sigma)}
$$
satisfy  the assumption 
of Corollary \ref{Kweightlemma} with    
$\alpha =  s_1 + 2(1 - 3 \sigma) (>- \frac{n}{2})$ and $\beta = 2(1 -\sigma) (> 0)$.   
The assumption \eqref{qass} is satisfied by \eqref{qass2diss} together with the definition of $\alpha$.  
Hence, we can apply Corollary
\ref{Kweightlemma} 
to obtain
\begin{equation}\label{M12}
\begin{aligned}
|M_{1,2}| \lesssim 
& \langle t \rangle ^{{ \frac{1}{1 - \sigma} \left(-\frac{n}{2}(\frac{1}{q_1} - \frac{1}{2}) + \frac{\vartheta - s_1  }{2} + 3\sigma - 1  \right) }} 
\| (-\Delta)^{\frac{s_2}{2}}\varphi \|_{{{q_1}}}^\prime
\\
& \qquad			+
\langle t \rangle ^{{ \frac{1}{1 - \sigma} \left(-\frac{n}{2}(\frac{1}{{q_2}} - \frac{1}{2}) - \frac{s_1}{2} +  3\sigma - 1  \right) }} 
\||\cdot|^{\vartheta} (-\Delta)^{\frac{s_2}{2}}\varphi\|_{q_2}^\prime. 	
\end{aligned}\end{equation}
Inequality  \eqref{K1lowest+diff} follows from \eqref{weight_diss}, \eqref{M11}  and \eqref{M12}. 

Next we prove \eqref{K0lowest+diff}.  
It follows from \eqref{lambda_def2}, \eqref{lambda+} and \eqref{lambda+-} that  		
\begin{align}
&|\xi|^{s_1} \hat K_0^+(t,\xi) 
- |\xi|^{s_1} e^{- |\xi|^{2(1 - \sigma)} t} 
\nonumber \\
&	= 
|\xi|^{s_1} \left( \frac{- \lambda_-(|\xi|)}{\lambda_+(|\xi|)  - \lambda_-(|\xi|) }
- 1 \right)
e^{\lambda_+ (|\xi|) t}
+ 
|\xi|^{s_1}	\left(e^{\lambda_+ (|\xi|)t }-  e^{-|\xi|^{-2 (1- \sigma)}t} \right) 
\nonumber \\
&	= -
\frac{2 |\xi|^{s_1 + 2(1 - 2 \sigma) }}{(1 + \sqrt{1-4|\xi|^{2(1 - 2\sigma)}})\sqrt{1-4|\xi|^{2(1 - 2\sigma)}}}
\exp(\lambda_+(|\xi|) t)	 
\nonumber \\
& \quad	+ |\xi|^{s_1 }
\exp(-|\xi|^{2(1 - \sigma)}t)
\left(
\exp \Big(\frac{- 4|\xi|^{2(2 - 3 \sigma) }t}{(1+\sqrt{1-4|\xi|^{2(1 - 2\sigma)}})^2	}  \Big) -1
\right)
\\
&=: M_{2,1} + M_{2,2} \quad (\text{we put}). 
\label{weight3_diss}
\end{align}
We easily see that 
\begin{align*}
\upsilon(\rho) &=  
\frac{2 \rho^{s_1 + 2(1 -2 \sigma) }}{(1+\sqrt{1-4\rho^{2(1 - 2\sigma)}})\sqrt{1-4\rho^{2(1 - 2\sigma)}}},
\quad
\\
\lambda(\rho) &=\lambda_+(\rho) 
= \frac{- 2  \rho^{2 - 2 \sigma}t}{1 +\sqrt{1-4\rho^{2(1 - 2\sigma)} } }
\end{align*}
satisfy  the assumption 
of Corollary \ref{Kweightlemma} with  
$\alpha =  s_1 + 2(1 - 2 \sigma) , \beta = 2(1 -\sigma)$.   
		The assumption \eqref{qass2diss} and the definition of $\alpha$ above yield
\begin{align*}
&\frac{1}{q_3} \ge \frac{1}{2} + \frac{ \vartheta - s_1 }{n}
\ge \frac{1}{2} + \frac{ \vartheta - \alpha }{n}, 
\\
&\frac{1}{q_4} > \frac{1}{2} 
\ge \frac{1}{2} + \frac{- \alpha }{n}, 
\end{align*}	
that is,  \eqref{qass} holds for $q_1 = q_3$ and $q_2 = q_4$.  
Hence, we can apply Corollary
\ref{Kweightlemma} 
to obtain 
\begin{align}\label{M21}
|M_{2,1}| 
		&\lesssim 
\langle t \rangle ^{{ \frac{1}{1 - \sigma} \left(-\frac{n}{2}(\frac{1}{q_3} - \frac{1}{2}) + \frac{{\vartheta - s_1 } }{2} + 2\sigma - 1  \right) }} \|\varphi\|_{q_3}^\prime
\nonumber \\
& \qquad			+
\langle t \rangle ^{{ \frac{1}{1 - \sigma} \left(-\frac{n}{2}(\frac{1}{{q_4}} - \frac{1}{2}) - \frac{s_1}{2}+ 2\sigma - 1 \right) }} 
\||\cdot|^{\vartheta} (-\Delta)^{\frac{s_2}{2}}\varphi\|_{q_4}^\prime. 	
\end{align}

We also see that 
$$
\upsilon(\rho) = \rho^{s_1}, 
\quad
\lambda(\rho) = -\rho^{2(1 - \sigma)},
\quad
\mu(\rho) = \frac{4\rho^{2(2 - 3 \sigma) }}
{(1+\sqrt{1-4\rho^{2(1 - 2\sigma)}})^2	}
$$
satisfy  the assumption 
of Corollary \ref{Kweightlemmadiff} with  
$\alpha =  s_1, \beta = 2(1 -\sigma), \gamma = 2(2 - 3 \sigma)$.  
		The assumption \eqref{qass2diss} and the definition of $\alpha, \beta, \gamma$ above yield
\begin{align*}
&\frac{1}{q_3} \ge \frac{1}{2} + \frac{\vartheta - s_1 }{n}
>  
\frac{1}{2}  + \frac{\vartheta - s_1-2(1 - 2\sigma) }{n}
= \frac{1}{2} + \frac{ \vartheta - \alpha + \beta - \gamma }{n}, \;
\\
&\frac{1}{q_4} > \frac{1}{2} 
\ge \frac{1}{2}  + \frac{- s_1-2(1 - 2\sigma) }{n}
= \frac{1}{2} + \frac{ - \alpha + \beta - \gamma }{n}, \;
\end{align*}	
that is,  \eqref{qassdiff} holds for $q_1 = q_3, q_2 = q_4$.   
Hence, we can apply Corollary \ref{Kweightlemmadiff} 
to obtain													
\begin{align}\label{M22}
|M_{2,2}| 
		&\lesssim 
\langle t \rangle ^{{ \frac{1}{1 - \sigma} \left(-\frac{n}{2}(\frac{1}{q_3} - \frac{1}{2}) + \frac{{\vartheta - s_1 } }{2} + 2\sigma - 1  \right) }} \|\varphi\|_{q_3}^\prime
\nonumber \\
& \qquad			+
\langle t \rangle ^{{ \frac{1}{1 - \sigma} \left(-\frac{n}{2}(\frac{1}{{q_4}} - \frac{1}{2}) - \frac{s_1}{2}+ 2\sigma - 1 \right) }} 
\||\cdot|^{\vartheta} (-\Delta)^{\frac{s_2}{2}}\varphi\|_{q_4}^\prime. 	
\end{align}
Inequality \eqref{K0lowest+diff} follows from \eqref{weight3_diss}, \eqref{M21} and \eqref{M22}. 
 
\end{proof}

			\subsection{Estimate of the kernels for high frequency part$(|\xi| \geq 1)$}
						
			In this subsection, we consider high frequency region: 
			$|\xi| \geq 1$. 
					
		\begin{lemma}	\label{high}
	For every $s, \delta \ge 0$, the following hold.  
						\begin{align}
&			\Norm{(-\Delta)^{\frac{s }{2}}  \left(K_{1,high}(t, \cdot) * \varphi \right) }{{2}}
			\lesssim
			e^{-\frac{t}{2}} \| 
			(-\Delta)^{\frac{s }{2}} (1 - \Delta)^{-\frac{1}{2}}  \varphi \|_{{2}},	
			\label{high_der}
	\\
	&				\Norm{\langle \cdot \rangle^\delta    \left(K_{1,high}(t, \cdot)  * \varphi \right) }{{2}}
			\lesssim
			e^{-\frac{t}{2}} \|   (1 - \Delta)^{-\frac{1}{2}} \langle \cdot \rangle^\delta \varphi\|_{{2}},
					\label{high_weight}
					\\
		&			\Norm{(-\Delta)^{\frac{s }{2}}  \left(K_{0,high}(t, \cdot) * \varphi \right) }{{2}}
		\lesssim
		e^{-\frac{t}{2}} \| (-\Delta)^{\frac{s }{2}} \varphi \|_{{2}},	
		\label{high0_der}
		\\
		&				\Norm{\langle \cdot \rangle^\delta    \left(K_{0,high}(t, \cdot)  * \varphi \right) }{{2}}
		\lesssim
		e^{-\frac{t}{2}} \|\langle \cdot \rangle^\delta  \varphi\|_{{2}},	
		\label{high0_weight}			
			\end{align}
			provided the right-hand sides are finite.  
		\end{lemma}
		
		\begin{proof}	
	We easily see that 
				\begin{align}\label{K1high1}
			\langle \xi \rangle |\widehat{K_1^{\pm}}(t,\xi) |
				&\lesssim  e^{-\frac{|\xi|^{2 \sigma}}{2}t}  \le e^{-\frac{t}{2}},
				\end{align}
			on the support of $ \chi_{high}$. 
						Hence, 			
			\begin{equation*}
			\begin{aligned}
			&		\Norm{(-\Delta)^{\frac{s }{2}} \left(K_{1,high}(t, \cdot) * \varphi \right)}{{2}}
			=\Norm{|\xi|^{s } \hat K_{1,high}(t,\xi)  \hat \varphi(\xi)  }{{2}}
			\\
			&\le \Norm{\hat K_{1,high}(t,\xi) \langle \xi \rangle} {{\infty}}
			\Norm{|\xi|^{s } \langle \xi \rangle ^{ - 1 } \hat \varphi(\xi) }{{2}}
			\lesssim e^{-\frac{t}{2}} \| (-\Delta)^{\frac{s }{2}} (1 - \Delta)^{-\frac{1}{2}}  \varphi \|_{{2}},	
			\end{aligned}
			\end{equation*}
			that is, \eqref{high_der} holds.  
			
			In the proof of \cite[p.  10]{IIW} (see also \cite[ p. 643]{Hy}), the following Leibniz rule is shown:
			\begin{equation}\label{Lr}
			\|(1 - \Delta)^{\frac{\vartheta  }{2} } (\varphi \psi) \|_{{2}} 
							\le
		\sum_{j = 1}^n \sum_{k = 1}^{[\vartheta]+1}\|	\partial_j^k \varphi  \|_{{\infty}}
		\|(1 - \Delta)^{\frac{\vartheta  }{2} } \psi \|_{{2}}. 
			\end{equation}
		Since 
			$		|\partial_j^k (\hat K_{1,high}|\xi|)(t,\xi) |
			\le C_{k} e^{-\frac{t}{2}}$ 
					for every nonnegative integer $k$, 
								we have
			\begin{equation*}
			\begin{aligned}
		&	\Norm{|\cdot|^\vartheta 
				     \left(K_{1,high} * \varphi \right) }{{2}}
			=\Norm{|(-\Delta)^{\frac{\vartheta  }{2}} 
				\left( (\hat K_{1,high}(t,\xi) \langle \xi \rangle ) 
			(	\langle \xi \rangle^{-1}
				\hat \varphi(\xi) )
				\right) }{{2}}
					\\			
& \qquad \qquad 
\lesssim e^{-\frac{t}{2}}
								\Norm{\langle \xi \rangle^{-1} (1 - \Delta)^{\vartheta   /2} \hat \varphi(\xi) }{{2}}
				\sim e^{-\frac{t}{2}} \| (1 - \Delta)^{-\frac{1}{2}}  
			\langle \cdot \rangle^\vartheta   \varphi\|_{{2}}.   
						\end{aligned}
			\end{equation*}
	Taking $\vartheta   = 0$ and $\delta$ in this inequality,  we obtain 
	\eqref{high_weight}.  

	We can prove  \eqref{high0_der} and \eqref{high0_weight} in the same way.  	
	\end{proof}
			
	\subsection{Estimate of the kernels for middle frequency part}		

	In this subsection, we consider the region:	
	$ |\xi| \in [2^{{-\frac{3}{1 - 2\sigma} }}, 2]$. 		
	
	\begin{lemma}	\label{middle}
		There is a constant $\vare_{\sigma} \in (0, \frac{1}{2})$ such that 
		the following hold for every $s \ge 0$, 
			$ \vartheta \ge 0$:
		\begin{align}
		&			\Norm{(-\Delta)^{\frac{s}{2}}  \left(K_{1,mid}(t, \cdot) * \varphi \right)}{{2}}
		\lesssim
		e^{-\vare_{\sigma} t} 
		\| (1 - \Delta)^{-\frac{1}{2}}  \varphi \|_{{2}},	
		\label{mid_der1}
		\\
		&				\Norm{\langle \cdot \rangle^\vartheta      \left(K_{1,mid}(t, \cdot)  * \varphi \right)}{{2}}
		\lesssim
		e^{-\vare_{\sigma} t}\| (1 - \Delta)^{-\frac{1}{2}}  
		\langle \cdot \rangle^\vartheta    \varphi \|_{{2}},
			\label{mid_weight1}
		\\
		&			\Norm{(-\Delta)^{\frac{s}{2}}  \left(K_{0,mid}(t, \cdot) * \varphi \right)}{{2}}
		\lesssim
		e^{-\vare_{\sigma} t} 
		\| \varphi \|_{{2}},	
		\label{mid0_der1}
		\\
		&		\Norm{
					\langle \cdot \rangle^\vartheta      \left(K_{0,mid}(t, \cdot) * \varphi \right)}{{2}}
		\lesssim
		e^{-\vare_{\sigma} t} \| \langle \cdot \rangle^\vartheta    
	 \varphi \|_{{2}},
		\label{mid0_weight1}			
		\end{align}
		provided the right-hand sides are finite. 		
		\end{lemma}
	
	\begin{proof}
		By definitions \eqref{K1def}, \eqref{lambda_def1} and \eqref{lambda_def2}, we have
		\begin{equation}
		\begin{aligned}
		|	\widehat K_1(t,\xi) | & = 
	\left|	 \frac{e^{\lambda_+(|\xi|) t}- e^{\lambda_-(|\xi|) t}}{\lambda_+(|\xi|) - \lambda_-(|\xi|)} 
	\right|
		=  	t 	\left|e^{(\theta \lambda_+(|\xi|) + (1 - \theta)\lambda_-(|\xi|))t } \right|
		\\
		&	=  t \left| e^{\frac{t}{2} (- |\xi|^{2 \sigma}  + (2 \theta - 1)
				 |\xi|^{2 \sigma}	\sqrt{1-4|\xi|^{2(1 - 2\sigma)}}  )} 
			 \right|
			 		\end{aligned}				\end{equation}
		for some $\theta \in (0,1)$.  						
		Hence, 			
		\begin{equation}\label{K1mid1}
		|	\widehat K_1(t,\xi) | =	t e^{-\frac{1}{2} |\xi|^{2 \sigma}t} 
		\le	t e^{-2^{- \frac{1}{1 - 2\sigma}}t} 
		\lesssim 
			 e^{-2 \vare_{\sigma}t}
		\end{equation}
		in the case $2|\xi|^{1 - 2 \sigma} \ge 1$, 
		where 
		$\vare_{\sigma} = 2^{- \frac{6}{1 - 2 \sigma}- 1}$.  		 	
		Next, we consider the case $2|\xi|^{1 - 2 \sigma} \le 1$.  
		Then 
		\begin{align*}
		&	\frac{1}{2}	|\xi|^{2 \sigma} (-1 + (2 \theta - 1)
		\sqrt{1-4|\xi|^{2(1 - 2\sigma)}}  )		
		\le 
		\frac{1}{2}	 |\xi|^{2 \sigma}(-1 +  \sqrt{1-4|\xi|^{2(1 - 2\sigma)}} ) 
		\\ 
		& \quad = 
		\frac { - 2|\xi|^{2 - 2\sigma}}{ 1 +  \sqrt{1-4|\xi|^{2(1 - 2\sigma)}} }
		\le  - |\xi|^{2 - 2\sigma} \le -2^{- \frac{6(1-\sigma)}{1 - 2 \sigma}} 
	\le -  2 \vare_{\sigma},
			\end{align*}	
and thus,  
			\begin{equation*}
		|	\widehat K_1(t,\xi) | 
			\le	t e^{-2  \vare_{\sigma} t }	\lesssim	 e^{- \vare_{\sigma} t}
		\end{equation*}		
on $[2^{-\frac{3}{1-2\sigma}}, 2^{-\frac{1}{1-2\sigma}}]$, 	
which together with \eqref{K1mid1} yields 
			\begin{equation}\label{mid2}
\langle \xi \rangle |	\widehat K_1(t,\xi)\chi_{mid}(\xi) | \lesssim 
			e^{- \vare_{\sigma} t}. 
		\end{equation}		
Calculating in the same way as in the proof of \eqref{high_der} by using \eqref{mid2} instead of \eqref{K1high1}, and noting that $-\Delta$ is bounded operator on the support $\chi_{mid}$,  we obtain \eqref{mid_der1}.

	In the same way as in the proof of \eqref{mid2}, we see that
		\begin{equation}\label{mid3}
		\|  (1 - \Delta)^{\frac{k}{2}} \hat K_1(t,\cdot) \chi_{mid} \|_{{\infty}} 
		\le C_k e^{- \vare_{\sigma} t}, 
		\end{equation}
		for every $k \in \N \cup \{0\}$.  
		Then by the same calculation as in the proof of \eqref{high_weight},
	we obtain \eqref{mid_weight1}.  		
		
		We can estimate
		\begin{align*}
		\hat K_0
				&=\frac{ \lambda_+ e^{\lambda_- t}- \lambda_-  e^{\lambda_+ t}}{\lambda_+ - \lambda_-}
		= \frac{ e^{\lambda_+}  -   e^{\lambda_-} }				{\lambda_+ - \lambda_-}\lambda_+
		+ \lambda_-,
		\end{align*}
		in the same way, and obtain the assertion for $K_{0,mid}$.   
			\end{proof}	

	\section{Asymptotic profile of the solutions of Linear equation}
	
In this section, we prove Theorem \ref{thm_lin_diff}.    

Since the solution $u$ of \eqref{LW}  is written as
	\begin{align*}
	u(t,x) 
	&=
	\left( K_0(t,\cdot)\ast u_0 \right)(x)
	+
	\left( K_1(t,\cdot)\ast u_1 \right)(x),	
		\end{align*}
the conclusion of Theorem \ref{thm_lin_diff} follows from the following lemma.

\begin{lemma}\label{lemma_lin_diff}
	Let $u_j \in L^1 \cap L^2$ for $j = 0,1$.  Then the following hold.  
\begin{align}	
&
	\left\| K_1(t,\cdot)\ast u_1 
- G_{\sigma}(t,x) \int_{\R^n}u_1 (y)dy  \right\|_2
\nonumber \\
&	\lesssim \langle t \rangle 
^{{\max \{
		\frac{1}{1 - \sigma}\left(-\frac{n}{4} +  3\sigma - 1 \right),
		\frac{1}{\sigma}		 \left(-\frac{n}{4} +  \sigma \right) 
		\} }}
\| u_1 \|_{{{1}}} 
+ e^{{-\vare_\sigma t}} \|u_1 \|_{{H^{-2 \sigma} }}
\nonumber \\& \qquad 
+	 {t}^{ \frac{1}{1-\sigma}(-\frac{n}{4} + \sigma -\frac{\theta}{2} )}
\left\| | \cdot |^\theta  u_1 \right\|_1,
\label{I2est}
\\
&\left\| K_0(t,\cdot)\ast u_0
- H_{\sigma}(t,x) \int_{\R^n} u_0(y)dy  \right\|_2 
\nonumber \\
&\lesssim
\langle t \rangle 
^{{		\frac{1}{1 - \sigma}\left(-\frac{n}{4} +  2\sigma - 1 \right)}}
\|u_0\|_1
+ 
e^{-\vare_{\sigma} t} \| u_0 \|_2
+	 	{t}^{ \frac{1}{1-\sigma}(-\frac{n}{4}  - \frac{\theta}{2} )}
\left\| | \cdot |^\theta  u_0 \right\|_1. 
\label{I0est}
\end{align}
\end{lemma}

\begin{proof}
First we prove \eqref{I2est}.  We have
\begin{equation}\label{I2_2}
\begin{aligned}
&\left\| K_1(t,\cdot)\ast u_1 
- G_{\sigma}(t,\cdot) \int_{\R^n}u_1 (y)dy  \right\|	
\\	
&\le 
\Norm{
	\F^{-1}[\hat K_1^+(t,\xi) \chi_{low}]  * u_1 
	-  \F^{-1}[ |\xi|^{-2\sigma} e^{- |\xi|^{2(1 - \sigma)} t} \chi_{low}]  * u_1}{2}
\\
&\quad + \Norm{\F^{-1}[\hat K_1^-(t,\xi) \chi_{low}]  * u_1}{{2}} 
+\Norm{ K_{1,mh}(t,\cdot) * u_1 }{{2}}		
\\
& \quad + \Norm{ \F^{-1}\left[\hat G_{\sigma}(t,\xi)\chi_{mh}                  \right]*  u_1 }{{2}}		
+	\|	G_{\sigma}(t,\cdot)\ast u_1 - G_{\sigma}(t,\cdot) \int_{\R^n}u_1 (y)dy \|_2
\\
&=: I_{1,1} + I_{1,2}+ I_{1,3} + I_{1,4} + I_{1,5}. 
\end{aligned}
\end{equation}
By \eqref{K1lowest+diff} and \eqref{K1lowest-} for $q_1 = q_2 = 1$ and $\vartheta = s_1 = s_2 = 0$, we have 
\begin{align}\label{I21}
I_{1,1} + I_{1,2}
\lesssim 
\langle t \rangle ^{{ \frac{1}{1 - \sigma} \left(-\frac{n}{4}+ 3\sigma - 1 \right) }} \|u_1\|_{{1}}
+ \langle t \rangle ^{{ \frac{1}{\sigma} \left(-\frac{n}{4} + \sigma \right) }} \| u_1 \|_1. 
\end{align}	
By \eqref{high_der} and \eqref{mid_der1},
we have
\begin{align}
&		
I_{1,3}		\lesssim e^{-\vare_{\sigma} t} \|u_1 \|_{{H^{-1}}}.  
\label{K1hm}\end{align}
Since the support of $\chi_{mh} $ is included in $[2^{-\frac{3}{1 - 2\sigma}}, \infty)$ and 
$2^{{ -\frac{6(1 - \sigma)}{1 - 2\sigma} }} > \vare_\sigma$ , we have
\begin{align}
I_{1,4}	
&=\Norm{ \hat G_{\sigma}(t,\cdot)\chi_{mh}   \hat u_1 }{{2}}
= 
\Norm{ |\xi|^{- 2 \sigma}e^{-|\xi|^{2(1-\sigma)}t }\chi_{mh}  \hat u_1 }{{2}}
\lesssim e^{-\vare_{\sigma} t} \|u_1 \|_{{H^{-2 \sigma}}}. 
\label{G1hm}\end{align}
It is written that
\begin{align}
&I_{1,5} = \| \hat G_{\sigma}(t,\cdot)(\hat u_1(\cdot)- \hat u_1(0)) \|_{{2}}.  
\label{Gest01}\end{align}
Since $\hat G_{\sigma}(t,\xi) = |\xi|^{-2\sigma}e^{-|\xi|^{2(1-\sigma)}t }$, 
we have by the transformation $t^{\frac{1}{2(1-\sigma)}} r = \rho$ that
\begin{equation*}
\begin{aligned}
\| \hat G_{\sigma}(t,\cdot)|\cdot|^{\theta} \|_{{2}}^2  
&= \int_0^\infty r^{2( \theta-2\sigma) + n - 1}
e^{-2 r^{2(1-\sigma)} t} dr
\\	
&
=t^{- \frac{n+2( \theta-2\sigma)}{2(1-\sigma)}} 
\int_0^\infty \rho^{2(\theta -2\sigma) + n-1}
e^{-2 \rho^{2(1-\sigma)} } d\rho
\sim t^{- \frac{n+2(\theta-2\sigma)}{2(1-\sigma)}},
\end{aligned}
\end{equation*}
that is, 
\begin{equation}
\| \hat G_{\sigma}(t,\cdot)|\cdot|^{\theta} \|_{{2}}		
\sim {t}^{ \frac{1}{1-\sigma}(-\frac{n}{4}-\frac{\theta}{2} + \sigma)}. 
\label{Gest02}
\end{equation}
On the other hand, since $\theta \in [0,1]$, we have 
\begin{align}\label{phiest1}
|\hat u_1(\xi)- \hat u_1(0) | 
&\le \int_{\R^n} \left|(e^{i x \cdot \xi} - 1)u_1(x) \right| dx
=  \int_{\R^n}\left|  (e^{\frac{i x \cdot \xi}{2}} - e^{\frac{-i x \cdot \xi}{2}} )u_1(x)  \right| dx
\nonumber \\
&  = 2 \int_{\R^n}\left| \sin \left( \frac{ x \cdot \xi}{2} \right) u_1(x)  \right| dx
\nonumber \\
&
\le 2  \int_{\R^n}  \left(\frac{|x \cdot \xi|}{2}\right)^\theta
|u_1(x)| dx 
= 2^{1-\theta} |\xi|^\theta \left\|| \cdot |^\theta  u_1 \right\|_1
\end{align}
for every $\xi \in \R^n$.  
From  \eqref{Gest01}, \eqref{Gest02} and \eqref{phiest1}, it follows that
\begin{equation}
\begin{aligned}\label{I25}
I_{1,5}
&\le 2 \left\| \hat G_{\sigma}(t,\cdot)|\cdot|^{\theta}  \;
\left\|| \cdot |^\theta  u_1 \right\|_1 \right\|_{{2}} 
= 2 \| \hat G_{\sigma}(t,\cdot)|\cdot|^{\theta} \|_{{2}}  \;
\left\|| \cdot |^\theta  u_1 \right\|_1
\\
&\lesssim {t}^{ \frac{1}{1-\sigma}(-\frac{n}{4}-\frac{\theta}{2} + \sigma)}
\left\|| \cdot |^{\theta}  u_1 \right\|_1. 
\end{aligned}\end{equation}	
Substituting \eqref{I21}--\eqref{I25} into \eqref{I2_2}, we obtain \eqref{I2est}.  

Next we prove \eqref{I0est}.  We have
\begin{equation}\label{I0_1}
\begin{aligned}
& \left\| K_0(t,\cdot)\ast u_0
- H_{\sigma}(t,x) \int_{\R^n}u_0(y)dy  \right\|_2  
\\ 
&\le 
\Norm{
	\F^{-1}[\hat K_0^+(t,\cdot) \chi_{low}]  * u_0 
	-  \F^{-1}[ \hat H_{\sigma}(t,\cdot)\chi_{low}]  * u_0}{2}
\\
&\quad + \Norm{\F^{-1}[\hat K_0^-(t,\cdot) \chi_{low}]  * u_0}{{2}} 
+\Norm{ K_{0,mh}(t,\cdot) *  u_0 }{{2}}		
\\
&\quad + \Norm{ \F^{-1}\left[\hat H_{\sigma}(t,\cdot)\chi_{mh}  \right]
	* u_0 }{{2}}		
+	\|	H_{\sigma}(t,\cdot)\ast u_0 - H_{\sigma}(t,\cdot) \int_{\R^n}u_0 (y)dy \|_2
\\
&=: I_{0,1} + I_{0,2}+ I_{0,3} + I_{0,4} + I_{0,5}. 
\end{aligned}
\end{equation}
By \eqref{K0lowest+diff} with $q_j =1, s_j = 0 \;(j = 1,2)$ and $\vartheta = 0$, we have 
\begin{equation}
\begin{aligned}\label{I01}
I_{0,1} 
& = \Norm{
	\F^{-1}\left[
	\left(
	\hat K_0^+(t,\xi) -   e^{- |\xi|^{2(1 - \sigma)} t} \right)
	\chi_{low}\right]  *u_0}{2}
  \\& 
\lesssim 
\langle t \rangle ^{{ \frac{1}{1 - \sigma} \left(-\frac{n}{4}+ 2\sigma - 1  \right) }} \|u_0\|_{{1}}. 
\end{aligned}
\end{equation}
Inequality \eqref{K0lowest-} implies
\begin{align}\label{I02}
& I_{0,2}  \lesssim
\langle t \rangle ^{{ \frac{1}{\sigma} \left(-\frac{n}{4} + 2 \sigma - 1 \right) }} \| u_0 \|_1,	
\end{align}	
and inequalities \eqref{high0_der} and \eqref{mid0_der1} imply
\begin{align}\label{I03}
&  I_{0,3} \lesssim
e^{-\vare_\sigma t} \|u_0 \|_2. 
\end{align}	
Since the support of $\chi_{mh} $ is included in $[2^{-\frac{3}{1 - 2\sigma}}, \infty)$ and 
$2^{{ -\frac{6(1 - \sigma)}{1 - 2\sigma} }} > \vare_\sigma$ , we have
\begin{align}
I_{0,4}	
&	=\Norm{ e^{-|\xi|^{2(1-\sigma)}t }\chi_{mh}  \hat u_0 }{{2}}
\lesssim e^{-\vare_\sigma t} \|u_0 \|_2. 
\label{G00hm}
\end{align}
By \eqref{Gest02} with $\theta$ replaced by $2 \sigma + \theta$, we have	
\begin{equation}
\begin{aligned}
\| \hat H_{\sigma}(t,\cdot)|\cdot|^{\theta} \|_{{2}}
= \| \hat G_{\sigma}(t,\cdot)|\cdot|^{2 \sigma + \theta} \|_{{2}}
&
\sim {t}^{ \frac{1}{1-\sigma}(-\frac{n}{4}-\frac{ \theta}{2} )}. 
\end{aligned}\label{Hest02}
\end{equation}
Then, in the same way as in the proof of \eqref{I25}, by using \eqref{Hest02} instead of \eqref{Gest02}, we have
\begin{align}\label{I05}	
I_{0,5}
&\lesssim
{t}^{ \frac{1}{1-\sigma}(-\frac{n}{4}  - \frac{\theta}{2} )}
\left\| | \cdot |^\theta  u_0 \right\|_1. 
\end{align}
Since $\sigma < 1 - \sigma$, \eqref{I0est} follows from \eqref{I0_1} -- \eqref{I05}.  
\end{proof}

\section{Estimate of the nonlinear term}
		
Throughout this section, we suppose the assumption \eqref{fass}, and  
estimate nonlinear terms by using the argument of \cite{Hy} and \cite{IIW}.  
For $r \in [1,2)$, $\delta \in [ 0,\frac{n}{2} - 2 \sigma)$ and $\bar{s} \ge 1$, 
we define 
\begin{equation}\label{Xdef}
X_{r,\delta,\bar{s}} :=\{ u \in  C((0,\infty);H^{\bar{s}} \cap H^{0,\delta} ); 
\Norm{u}X_{r,\delta,\bar{s}} < \infty \},
\end{equation}
 where
 \begin{equation}\label{Xnormdef}
 	\begin{aligned}
\Norm{\varphi}{{X_{r,\delta,\bar{s}}}} 
	:= \sup_{t > 0} 
&	\Big(
\langle t \rangle ^{{\frac{1}{1 - \sigma} \left(\frac{n}{2}(\frac{1}{r} - \frac{1}{2}) -  \sigma 
		 + \frac{\bar{s}}{2}\right)	}}
	\Norm{(-\Delta)^{{\bar{s}/2}} \varphi(t)}{{2}} 
\\& \qquad 	
+
	\langle t \rangle ^{{\frac{1}{1 - \sigma} \left(\frac{n}{2}(\frac{1}{r} - \frac{1}{2}) - \frac{\delta}{2} - \sigma \right) }}
	\Norm{\langle\cdot \rangle^\delta \varphi(t)}{{2}}			
		\Big). 
	\end{aligned}
	\end{equation}

For $\vartheta \in [0, \frac{n}{2}- 2 \sigma)$, we put
				\begin{align}	\label{zeta_def}			
				\zeta_{r,\vartheta} :&=  \frac{1}{1 - \sigma} 
				\left(-\frac{n}{2}(\frac{1}{r}  - \frac{1}{2} ) + \sigma + \frac{\vartheta}{2}
				- (p-1)\left(\frac{n}{2r}- \sigma \right) +  \frac{1}{2} \right)
				\\
						&=	\frac{1}{1 - \sigma} 
				\left(	-(\frac{n}{2r}  - \sigma )p  + \frac{n}{4} +  \frac{\vartheta}{2} +  \frac{1}{2}	 \right). 
				\nonumber 
				\end{align}
			For $s \ge 0$, we define				
					\begin{align}
					\tilde q_{s} :&= \frac{2n}{n + 2  +2 [s] - 2 s}, 
					\quad \text{that is,}\quad 
					\;	\frac{1}{\tilde q_s} 	= \frac{1}{2} + \frac{1 + [s] - s }{n}. 
					\label{tildeqdef}	 
									\end{align}
		
	\begin{lemma}	\label{function}
Let $r \in [1,2)$, $\delta \in [ 0,\frac{n}{2} - 2 \sigma)$ and $\bar{s} > 2 \sigma$.  
Let $X = X_{r,\delta,\bar{s}}$.  
			Then the following holds for every $\vartheta   \in [0,\delta]$, 
$s \in [0,\bar{s}]$ 
and $u \in X$:
\begin{enumerate}
\item We have
		\begin{align}
	& \Norm{(- \Delta)^{\frac{s}{2}} u(t,\cdot)}{{2}} \lesssim
	\langle t \rangle ^{{\frac{1}{1 - \sigma} \left(-\frac{n}{2}(\frac{1}{r} - \frac{1}{2}) +  \sigma  - \frac{s }{2}
			\right) }}\Norm{u}{X},	
    \label{beta_der}			\\
		&\Norm{|\cdot|^\vartheta     u(t,\cdot)}{{2}} \lesssim 
			\langle t \rangle ^{{\frac{1}{1 - \sigma} \left(-\frac{n}{2}(\frac{1}{r} - \frac{1}{2}) + \frac{\vartheta  }{2} + \sigma \right) }}\Norm{u}{X}. 
			\label{udelta}
			\end{align}			
\item	We have	
	\begin{align}
	\label{u_q2}
	&\Norm{u(t,\cdot)}{{q,2}} \lesssim  
	\langle t \rangle ^{{\frac{1}{1 - \sigma} \left(-\frac{n}{2}(\frac{1}{r} - \frac{1}{q})  + \sigma \right) }}\|{u}\|_{X}, \;
\quad \text{if} \quad
	q = \frac{2n}{n + 2\delta },	
	\\
	&\label{u_q}
		 \Norm{u(t,\cdot)}{{q}} \lesssim
		\langle t \rangle ^{{\frac{1}{1 - \sigma} \left(-\frac{n}{2}(\frac{1}{r} - \frac{1}{q}) +  \sigma 
				\right) }}\Norm{u}{X} 
	\\
			&\nonumber \qquad	\qquad	\qquad	\qquad \qquad 
\quad \text{if} \quad
	q \in 
	\begin{cases}
	(\frac{2n}{n + 2\delta }, \frac{2n}{n - 2 \bar{s}}] \quad &(2\bar{s} < n),
	\\
	(\frac{2n}{n + 2\delta }, \infty) \qquad \quad &(2\bar{s} \ge n). 		
	\end{cases}
\end{align}

\item We have
								\begin{align}
				& \| (-\Delta)^{\frac{[s ]}{2}} f(u(t,\cdot))\|_{{{\tilde q_s }}} 
				\lesssim		
				\langle t \rangle^
				{{	\frac{1}{1- \sigma}
						\left(
						(-\frac{n}{2r} + \sigma)p + \frac{n}{4} + \frac{1}{2}  - \frac{s }{2} 							
						\right)
					}}		
					\|u \|_X^p,
					\label{nonlinear1} 	\\
					& \Norm{(- \Delta)^{\frac{s}{2}} (1 - \Delta)^{-\frac{1}{2}} f(u(t,\cdot))}{{2}} 
					\lesssim
					\langle t \rangle ^
					{{
							\frac{1}{1 - \sigma} 
							\left( (-\frac{n}{2r}+\sigma)p + \frac{n}{4} + \frac{1}{2}  - \frac{s }{2}
							\right) 
						}}
						\|{u} \|_{X}^p
						\label{nabla_beta_-}
						\\
						& 	\Norm{\langle \cdot \rangle^\vartheta  f(u(t,\cdot)) }{{ {2n/(n+2)} }  }
						\lesssim \langle t \rangle ^{\zeta_{r,\vartheta}} \|u\|_{X}^p.  
						\label{fweightq}
						\end{align}	
									\end{enumerate}						
					\end{lemma}     	
		
			\begin{proof}
	Except use of the weak $L^p$ estimate, we follow the argument of  
			\cite[Lemmas 2.3 and 2.5]{IIW}, which is originated in \cite[Lemma 2.1, 2.3 and 2.5]{Hy}.

		(i) By Plancherel's  theroem and H\"{o}lder's inequality, we have 
				\begin{align*}
				& \Norm{(- \Delta)^{\frac{s}{2}} u(t,\cdot)}{{2}}
				= \|{|\cdot|^{s}  \hat  u(t,\cdot)}\|_{{2}} 
				\\
				&\le  \|{|\cdot|^{\bar{s}} \hat  u(t,\cdot)}\|_{{2}}^{\frac{s }{\bar{s}}} 	
				\|\hat  u(t,\cdot)\|_{{2}}^{1 -\frac{s }{\bar{s}}}
				= \|{(-\Delta)^{\frac{\bar{s}}{2}}   u(t,\cdot)}\|_{{2}}^{\frac{s }{\bar{s}}} 	
				\|u(t,\cdot) \|_{{2}}^{1 -\frac{s }{\bar{s}}},
			\end{align*}
			which together with the definition of $\|\cdot\|_X$ implies \eqref{beta_der}.  
			In the same way,  we see that \eqref{udelta} holds  by H\"{o}lder's inequality. 
						
		(ii) We first consider the case 	$\frac{2n}{n + 2\delta } = q$, that is, 
				$n(\frac{1}{q} - \frac{1}{2}) = \delta$.  Then
					\begin{equation}\label{q_cri}
					\| u(t,\cdot) \|_{{q,2}} \lesssim   
					\||\cdot|^{-\delta} \|_{{\frac{n}{\delta}, \infty}}
					\||\cdot|^{\delta}u(t,\cdot)\|_{{2}} 
					\lesssim 	\||\cdot|^{\delta}u \|_{{2}},
					\end{equation}		 
which together with the definition of $\|\cdot\|_X$ implies \eqref{u_q2}.  						 
		
	Next we  consider the case $\frac{2n}{n + 2\delta } < q < 2$, that is,  
			$0 < n \left( \frac{1}{q} - \frac{1}{2}\right) < \delta$.     
			In \cite[(2.12)]{IIW} (see also \cite[(2.5)]{Hy}), the following is shown 
				\begin{equation}\label{2.12}
		\| u \|_{{q}} \lesssim   
		\|u\|_{{2}}^{1 - \frac{n}{\delta }\frac{2 - q}{2q}}
		\||\cdot|^\delta u\|_{{2}}^{\frac{n}{\delta }\frac{2 - q}{2q}},
			\end{equation}	
when $ 0 < n \left( \frac{1}{q} - \frac{1}{2}\right) < \delta$.  This together with the definition of $\|\cdot\|_X$ implies \eqref{u_q}.

We consider the case $2 \le q \le \frac{2n}{n - 2 \bar{s}}$ and $\bar{s} < 2 n$.  
Let $\tilde s = \frac{n}{2} - \frac{n}{q} (\le \bar{s})$.  Then Sobolev's embedding theorem together with  \eqref{beta_der}  implies 
		\begin{align*}
		\| u \|_{{q}} \le \| u \|_{{\dot H^{\tilde s}}} 
			\le 
		\langle t \rangle ^{{\frac{1}{1 - \sigma} \left(-\frac{n}{2}(\frac{1}{r} - \frac{1}{2})  + \sigma - \frac{s}{2} \right) }}\|{u}\|_{X}
		= \langle t \rangle ^{{\frac{1}{1 - \sigma} \left(-\frac{n}{2}(\frac{1}{r} - \frac{1}{q})  + \sigma \right) }}\|{u}\|_{X},
					\end{align*}
	that is, \eqref{u_q} holds. 	
	In the same way, \eqref{u_q} holds also in the case $2 \le q < \infty$ and $\bar{s} \le  2 n$. 
	  
(iii) 
We put
\begin{equation}
\label{kappadef}
\kappa := \frac{n}{2} - \frac{1}{p-1}. 
\end{equation}
By the Leibniz rule together with the assumption \eqref{fass}, we have  
\begin{equation}\label{Leibniz}
\begin{aligned}
\| \nabla^{[s]} f(u(t,\cdot)) \|_{{{\tilde q_s}}}
\le 
\Big\|u(t,\cdot)^{p-[s]} \sum_{{\sum_{j=1}^{[s]}|\nu_j |= [s]
	}}
	\prod_{j=1}^{[s]}
	| D_x^{\nu_j} u(t,\cdot)| \Big \|_{{{\tilde q_s}}},
	\end{aligned}
	\end{equation}	
	where $\nu_j$ is a multi index.   
Put $k_j = |\nu_j|$.   
Then, as in the proof of \cite[Lemma 2.5]{IIW},   
we can choose $s_j \in [0,k_j - \frac{1}{p-1})$ such that $q_j$ $(j =1,\cdots,[s])$ defined by 
 \begin{align}
 \frac{1}{q_0} &= \left( \frac{1}{2} - \frac{\kappa}{n} \right)(p - [s]),
\label{q0} \\
 \frac{1}{q_j} &=  \frac{1}{2} - \frac{\kappa + s_j - k_j}{n} \;\;
 (j = 1,\cdots,[s])
 \label{qj} 
 \end{align}
satisfies 
 \begin{align}\label{qj_cond1}
&\sum_{j = 0}^{[s]}\frac{1}{q_j} = \frac{1}{\tilde q_s},
\\
\label{qj_cond2}
& (p - [s])q_0 \in [2,\infty) \; \text{and} \; q_j  \in [2,\infty) \; \text{for} \; j = 1,\cdots,[s]. 
\end{align}
Since
	\begin{align*}
&	\frac{1}{\tilde{q_{s}}}   - \frac{1}{\textbf{}q_0} 
= \frac{1}{2} + \frac{1}{n} + \frac{[s] - s}{n} 
	-  \left( \frac{1}{2} - \frac{\kappa}{n} \right)(p - [s]),
\\
& \sum_{j=1}^{[s]}\frac{1}{q_j} 
=  \sum_{j=1}^{[s]}
 \left(
  \frac{1}{2} - \frac{\kappa + s_j - k_j}{n} 
  \right)
 = (\frac{1}{2} - \frac{\kappa}{n})[s] 
  - \frac{1}{n}  \sum_{j=1}^{[s]}   s_j  + \frac{[s]}{n}, 
  	\end{align*}
the condition \eqref{qj_cond1} is equivalent to
\begin{equation}\label{betaj_cond}
\sum_{j=1}^{[s]} s_j	
= s - \kappa,
\end{equation}		
and thus, $\kappa + s_j \le s$.  
Taking \eqref{q0} -- \eqref{qj_cond2} into account, we apply  H\"{o}lder's inequality 
and Sobolev's embedding theorem to \eqref{Leibniz}. Then we obtain
\begin{equation}\label{nonlinearder}
\begin{aligned}
\| \nabla^{[s]} f(u(t,\cdot)) \|_{{{\tilde q_s}}}
&	\lesssim  
\|u^{p-[s]}\|_{{{q_0}}}
\sum_{{k_j \ge 0, 
		\sum_{j=1}^{[s]}k_j = [s]
	}}
	\prod_{j=1}^{[s]}
	\| |\nabla|^{k_j} u(t,\cdot) \|_{{{q_j}}} 
	\\
	&\lesssim  
		\| \nabla^{\kappa} u \|_{{2}}^{p - [s] }
	\sum_{{k_j \ge 0, 
			\sum_{j=1}^{[s]}k_j = [s]
		}}
		\prod_{j=1}^{[s]}
	\| |\nabla|^{\kappa + s_j} u(t,\cdot) \|_{{2}},						
		\end{aligned}
		\end{equation}
		where $|\nabla| := (-\Delta)^{1/2}$.   
Then estimating the right-hand side of \eqref{nonlinearder} by the definition of $\| \cdot \|_X$,  and using \eqref{betaj_cond} and \eqref{kappadef}, we obtain
\begin{equation*}
\begin{aligned}
\| \nabla^{[s]} f(u(t,\cdot)) \|_{{{\tilde q_s}}}
& \lesssim \langle t \rangle^
{{	\frac{1}{1- \sigma}
		\left(
		(-\frac{n}{2}(\frac{1}{r} - \frac{1}{2}) + \sigma )p
		- \frac{\kappa}{2}(p - [s])
		- \frac{1}{2} \sum_{j=1}^{[s]}
		(\kappa + s_j)			
		\right)
	}}		
	\|u \|_X^p
	\\
	&= 		
	\langle t \rangle^
	{{	\frac{1}{1- \sigma}
			\left(
			(-\frac{n}{2}(\frac{1}{r} - \frac{1}{2}) + \sigma 
			- \frac{\kappa}{2})p - 
			\frac{1}{2}\sum_{j=1}^{[s]} s_j			
			\right)
		}}		
		\|u \|_X^p
		\\
		& = 		
			\langle t \rangle^
			{{	\frac{1}{1- \sigma}
					\left(
					(-\frac{n}{2r} + \sigma)p + \frac{n}{4} + \frac{1}{2}  - \frac{s}{2} 							
					\right)
				}}		
				\|u \|_X^p,
			\end{aligned}
			\end{equation*}
that is, \eqref{nonlinear1} holds.  

	In view of Sobolev's embedding theorem, inequality \eqref{nabla_beta_-} follows from \eqref{nonlinear1}.  

By H\"{o}lder's inequality together with the assumption \eqref{fass}, we have	
		\begin{equation}\label{fweightq1}
		\begin{aligned}
		& 	\|\langle \cdot \rangle^\vartheta  f(u(t,\cdot)) \|_{{2n/(n+2)} }  
		\lesssim 
		\| \langle \cdot \rangle^\vartheta u(t,\cdot)\|_2
		\|u(t,\cdot)\|^{p-1}_{{{{n(p-1)}} }}. 
			\end{aligned}				
		\end{equation}
The assumption \eqref{pass} implies 
		$$
		n(p-1) \ge \frac{2rn}{n - 2r\sigma} > 
\frac{2n}{n+ 2 \delta}, 
	$$ 
and \eqref{pass2} implies
$$
n(p-1) \le \frac{2n}{n - 2 \bar{s}} 
$$
if $2 \bar{s} < n$. 
Thus, we can apply \eqref{udelta} and \eqref{u_q} with $q = n(p-1)$, which together with \eqref{fweightq1} yields \eqref{fweightq}. 
					\end{proof}     
		
\section{Estimates of a convolution term}

Throughout this section, we suppose the assumption \eqref{fass}.
 
\subsection{Decay estimates}

Throughout this subsection, we suppose the assumption of Proposition 1.  

\begin{lemma}\label{f(u)mh}
Let $\vartheta  > 0$. 
For every $u  \in X = X_{r,\delta,\bar{s}}$, we have 
\begin{equation}\label{sol_high_weight}
		\begin{aligned}
&		\int_0^t 
\| |\cdot|^\vartheta \F^{-1} \left[\hat K^{\pm} \chi_{mh}\right](t - \tau,\cdot) * f(u(\tau,\cdot)) \|_2 d\tau
\lesssim \langle t \rangle^{\zeta_{r, \vartheta} }	\|u\|_X^p,
		\end{aligned}
		\end{equation}
		where $\zeta_{r,\vartheta}$ is the number defined by \eqref{zeta_def}.  
\end{lemma}

\begin{proof}
	By \eqref{high_weight} and \eqref{mid_weight1},  Sobolev's embedding theorem and \eqref{fweightq}, 
	we have 
		\begin{equation*}
		\begin{aligned}
	&\Norm{|\cdot|^\vartheta  	\F^{-1} \left[\hat K^{\pm} \chi_{mh}\right](t - \tau,\cdot) * f(u(\tau,\cdot)) }{{2}}
	\\
	& \quad \lesssim e^{-\vare_\sigma (t - \tau)} 
	\| (1 - \Delta)^{-\frac{1}{2}} \langle \cdot \rangle^\vartheta  f(u(\tau,\cdot))\|_{{2}} 
	\\
	& \quad \lesssim	 e^{-\vare_\sigma (t - \tau)} 
	\|\langle \cdot \rangle^\vartheta   f(u(\tau,\cdot))\|_{{{{\frac{2n}{n+2}}} }} 
	\lesssim  e^{-\vare_\sigma (t - \tau)} \langle \tau \rangle^{\zeta_{r,\vartheta}} \|u \|_X^p,
		\end{aligned}
		\end{equation*}
which yields \eqref{sol_high_weight}. 	
		\end{proof}

\begin{lemma}	\label{f(u)}
For every $u,v  \in X = X_{r,\delta,\bar{s}}$, we have 
\begin{align}
&\int_0^t \Norm{|\cdot|^\delta  \left(K_{1}(t - \tau,\cdot) * f(u(\tau,\cdot))\right) }{{2}}d\tau
\lesssim
	\langle t \rangle ^{{ \frac{1}{1 - \sigma} \left(-\frac{n}{2}(\frac{1}{r} - \frac{1}{2}) + \frac{\delta}{2} + \sigma \right) }} \|u\|_{X}^p. 
	\label{conv_weight}
	\\
	&\int_0^t \Norm{|\cdot|^\delta  \left(K_{1}(t - \tau,\cdot) *  
		(f(u(\tau,\cdot)) -  f(v(\tau,\cdot))) \right) }{{2}}d\tau
\nonumber 
\\
& \qquad \quad 	\lesssim
	\langle t \rangle ^{{ \frac{1}{1 - \sigma} \left(-\frac{n}{2}(\frac{1}{r} - \frac{1}{2}) + \frac{\delta}{2} + \sigma \right) }} 
	(\|u \|_X +\| v\|_{X})^{p-1} \|u - v\|_X.  
		\label{conv_weight2}	
\end{align}
\end{lemma}

\begin{proof}
	First, we estimate the low frequency part.  
	By \eqref{K1lowweight} with $q_1 = r$, $q_2 = \frac{2n}{n+2}$ and $\vartheta = \delta$, $s_1 = s_2 = 0$,  we have  
	\begin{align}\label{I1I2}
	&\int_0^t 
	\Norm{|\cdot|^\delta \left( K_{1,low  }(t - \tau,\cdot) * f(u(\tau,\cdot)) \right)}{{2}}d\tau
	\le I_1 + I_2,  
	\end{align}
	where
	\begin{align*}
&	I_1: =
	\begin{cases}
	& \int_0^{t} 
	\langle t -\tau \rangle ^{{ \frac{1}{1 - \sigma} \left(-\frac{n}{2}(\frac{1}{r} - \frac{1}{2}) 
			+ \frac{\delta}{2} + \sigma \right) }} 
	\Norm{ f(u(\tau,\cdot))}{{{r}}}d\tau 
\; \quad \text{if} \quad r \ge 1,
	\\
&  \int_0^{t} 
\langle t -\tau \rangle ^{{ \frac{1}{1 - \sigma} \left(-\frac{n}{2}(\frac{1}{r} - \frac{1}{2}) 
		+ \frac{\delta}{2} + \sigma \right) }} 
\Norm{ f(u(\tau,\cdot))}{{{r,2}}}d\tau
\quad \text{if} \quad r > 1,\end{cases}
	\\
	&I_2 := 
	\int_0^{t} 
	\langle t -\tau \rangle ^{{ \frac{1}{1 - \sigma} \left(- \frac{1}{2} + \sigma 	\right) }} 
	\Norm{|\cdot|^\delta  f(u(\tau,\cdot)) }{{{2n/n+2}}}
	d\tau. 
	\end{align*}
Since $\bar{s} \ge 1$, the assumption \eqref{pass2} implies that 
$pr \le 2 p  \le \frac{2n}{n - 2 \bar{s}}$ if $2 s_o < n$.   
From \eqref{delta2}, it follows that 
$\frac{2n}{n + 2 \delta }< pr$ in the case $r = 1$,  and 
$\frac{2n}{n + 2 \delta } \le pr$ in the case $r > 1$.  
Hence, in the case $r = 1$, we can apply  \eqref{u_q} with $q = pr$ to obtain
\begin{align}
 \Norm{  f(u(\tau,\cdot)) }{{{r}}} \lesssim \|u(\tau,\cdot)\|_{{{pr}}}^p
&\lesssim  	
 \langle \tau \rangle ^{{ \frac{1}{1 - \sigma} 
		\left( -\frac{n}{2}(\frac{p-1}{r}) + p \sigma  \right) }}
 \|u\|_{X}^p. 
\label{f(u)Lr} 
\end{align}
In the case $r > 1$, we can apply \eqref{u_q2} with $q = pr$ to obtain
\begin{align}
\Norm{  f(u(\tau,\cdot)) }{{{r,2}}} \lesssim \|u(\tau,\cdot)\|_{{{pr,2}}}^p
&\lesssim  	
\langle \tau \rangle ^{{ \frac{1}{1 - \sigma} 
		\left( -\frac{n}{2}(\frac{p-1}{r}) + p \sigma  \right) }}
\|u\|_{X}^p. 
\label{f(u)Lr2} 
\end{align}
Substituting \eqref{f(u)Lr}  or \eqref{f(u)Lr2} into $I_1$,  we obtain
\begin{align*}
I_1  \lesssim
& \int_0^{t} \langle t - \tau  \rangle^
{{  \frac{1}{1 - \sigma} \left(-\frac{n}{2}(\frac{1}{r} - \frac{1}{2}) + \frac{\delta}{2} + \sigma \right) }} 
 \langle \tau \rangle^{\frac{1}{1-\sigma}(-\frac{n}{2}(\frac{p}{r} - \frac{1}{r}) + p \sigma)} d\tau  \|u\|_{X}^p. 
\end{align*}
 
The following inequality is commonly used to estimate the nonlinear term.   
                               \begin{equation}\label{integ}
                               \int_0^t \langle t - \tau \rangle^{\rho} s^{\eta}d\tau
                               \\
                               \lesssim \begin{cases}
                               \langle t \rangle^{\max \{\rho,\eta\} }&  \text{if}\quad  \min \{\rho, \eta \} < - 1,\\
                               \langle t \rangle^{\max \{\rho,\eta\} }\log(2+t)& \text{if} \quad \min \{\rho, \eta \} = -1,
                               \\
                               \langle t \rangle^{1+\rho+\eta }& \text{if} \quad  \min \{\rho, \eta \} > -1. 
                               \end{cases}
                               \end{equation}
The assumption that $\delta \ge n(\frac{1}{r} - \frac{1}{2}) - 1$ implies 
\begin{equation}
 \frac{1}{1 - \sigma} 
 \left(
 -\frac{n}{2}(\frac{1}{r} - \frac{1}{2}) + \frac{\delta}{2} + \sigma 
 \right) \ge -1.
\end{equation}
The assumption \eqref{pass} is equivalent to 
$
p \left(\frac{n}{2r} - \sigma \right) > \frac{n}{2r} - \sigma + 1, 
$
which is equivalent to
\begin{equation}\label{less-1}
\frac{1}{1-\sigma}\left(-\frac{n}{2}(\frac{p - 1}{r}) + p \sigma \right) < -1. 
\end{equation}
Hence, by using \eqref{integ}, we obtain
\begin{align}\label{I1}
I_1 \lesssim
& \langle t \rangle^{{  \frac{1}{1 - \sigma} \left(-\frac{n}{2}(\frac{1}{r} - \frac{1}{2}) + \frac{\delta}{2} + \sigma \right) }} \|u\|_{X}^p.  
\end{align}
Since $\frac{1}{1 - \sigma} (- \frac{1}{2}+ \sigma) > -1$, it follows from 
\eqref{fweightq} and \eqref{integ} that
\begin{align*}
 I_2 
& \lesssim
 \int_0^{t} \langle t - \tau  \rangle ^{{  \frac{1}{1 - \sigma} (- \frac{1}{2}+ \sigma)}}
\langle \tau \rangle^{\zeta_{r,\delta}}d\tau  \|u\|_{X}^p
\\
& \lesssim 
\begin{cases}
\langle t  \rangle ^{{  \frac{1}{1 - \sigma} (- \frac{1}{2}+ \sigma)}}\|u\|_{X}^p
& \quad  \text{if} \; \zeta_{r,\delta} < - 1
\\
\langle t   \rangle ^{{  \frac{1}{1 - \sigma} (- \frac{1}{2}+ \sigma) + \zeta_{r,\delta} + 1}}
\log (t + 2) \|u\|_{X}^p
& \quad \text{if} \; \zeta_{r,\delta} \ge -1. 
\end{cases}
\end{align*}
The assumption that $\delta \ge n(\frac{1}{r} - \frac{1}{2}) - 1$ means 
$- \frac{1}{2}+ \sigma \le -\frac{n}{2}(\frac{1}{r} - \frac{1}{2}) + \frac{\delta}{2} + \sigma$. Hence, we have
\begin{align}\label{I2}
I_2 \lesssim
& \langle t \rangle^{{  \frac{1}{1 - \sigma} \left(-\frac{n}{2}(\frac{1}{r} - \frac{1}{2}) + \frac{\delta}{2} + \sigma \right) }} \|u\|_{X}^p 
\end{align}
in the case $\zeta_{r,\delta} < - 1$.  
By definition \eqref{zeta_def} and assumption \eqref{pass}, we have  
\begin{align*}
&\frac{1}{1 - \sigma} (- \frac{1}{2}+ \sigma) + \zeta_{r,\delta} + 1
\\
& \quad =
\frac{1}{1 - \sigma}
\left(
-\frac{n}{2}(\frac{1}{r} - \frac{1}{2}) + \frac{\delta}{2} + \sigma 
+ 1 - (p-1)\left(\frac{n}{2r} -\sigma \right)
\right)
\\
& \quad < 
\frac{1}{1 - \sigma}
\left(-\frac{n}{2}(\frac{1}{r} - \frac{1}{2}) + \frac{\delta}{2} + \sigma
\right).
\end{align*}
Hence,  \eqref{I2} holds also in the case $\zeta_{r,\delta} \ge -1$.  
Substituting \eqref{I1} and \eqref{I2} into \eqref{I1I2}, we obtain
\begin{equation*}
\begin{aligned}
\int_0^t \Norm{|\cdot|^\delta  
	\left( K_{1,low}(t - \tau,\cdot) * f(u(\tau,\cdot))\right) }{{2}}d\tau
\lesssim
\langle t \rangle ^{{ \frac{1}{1 - \sigma} \left(-\frac{n}{2}(\frac{1}{r} - \frac{1}{2}) + \frac{\delta}{2} + \sigma \right) }} \|u\|_{X}^p,
\end{aligned}
\end{equation*}
which together with \eqref{sol_high_weight} yields \eqref{conv_weight}.   	

The assumption \eqref{fass} implies 
$$
|f(u(\tau,x)) -  f(v(\tau,x))| \lesssim	(	|u(\tau,x)| + | v(\tau,x)|)^{p-1} |u(\tau,x)-  v(\tau,x)|,  
$$
and we can prove \eqref{conv_weight2} in the same way.  
\end{proof}
 
\begin{lemma} \label{f(u)der} 
Let $s  \in [0,\bar{s}]$.  
For every $u,v  \in X = X_{r,\delta,\bar{s}}$, we have  
\begin{align}
&	\int_0^t 
	\Norm{(- \Delta)^{\frac{s}{2}}  \left(K_{1}^+(t - \tau,\cdot) * f(u(\tau,\cdot))\right) }{{2}}d\tau
	\lesssim
	\langle t \rangle ^{{ \frac{1}{1 - \sigma} \left(-\frac{n}{2}(\frac{1}{r} - \frac{1}{2}) - \frac{s }{2} + \sigma \right) }} \|u\|_{X}^p. 
\label{conv+_der}
\\
&	\int_0^t 
\Norm{(- \Delta)^{\frac{s}{2}}  \left(K_{1}^-(t - \tau,\cdot) * f(u(\tau,\cdot))\right) }{{2}}d\tau
\nonumber \\
& 
\lesssim
\langle t \rangle ^{{ 
\max \{
	\frac{1}{ \sigma} \left(-\frac{n}{2}(\frac{1}{r} - \frac{1}{2}) - \frac{s }{2} + \sigma \right), 
		\frac{1}{1- \sigma}
		\left(-\frac{n}{2}(\frac{1}{r} - \frac{1}{2})- \frac{s }{2} + \sigma 
	- (p-1)(\frac{n}{2r} - \sigma) + 1							
			\right)
		\}		}} \|u\|_{X}^p
\label{conv-_der}
\\ 
&	\int_0^t 
\Norm{(- \Delta)^{\frac{s}{2}}  \left(K_{1}(t - \tau,\cdot) * f(u(\tau,\cdot))\right) }{{2}}d\tau
\lesssim
\langle t \rangle ^{{ \frac{1}{1 - \sigma} \left(-\frac{n}{2}(\frac{1}{r} - \frac{1}{2}) - \frac{s }{2} + \sigma \right) }} \|u\|_{X}^p. 
\label{conv_der}
\\
&\int_0^t 
\Norm{(- \Delta)^{\frac{s}{2}}  \left(K_{1}(t - \tau,\cdot) * (f(u(\tau,\cdot)) - f(v(\tau,\cdot)) ) \right) }{{2}}d\tau \nonumber
\\
& \qquad \qquad \lesssim
\langle t \rangle ^{{ \frac{1}{1 - \sigma} \left(-\frac{n}{2}(\frac{1}{r} - \frac{1}{2}) - \frac{s }{2} + \sigma \right) }} (\|u\|_{X} + \|v\|_{X})^{p-1} \|u - v\|_{X}. 
\label{conv_der2}
	\end{align}
for every $u,v  \in X$, 
\end{lemma}

\begin{proof}
	We first prove \eqref{conv+_der}.  
		We divide the left-hand side of  \eqref{conv+_der} into three parts:
			\begin{equation}\label{conv_der1}
		\begin{aligned}
&\int_0^t \Norm{(- \Delta)^{\frac{s}{2}}  \left(K_{1}^+(t - \tau,\cdot) * f(u(\tau,\cdot))\right) }{{2}}d\tau
		\\
&=	\int_0^{t/2} \Norm{(- \Delta)^{\frac{s}{2}}  
	\left(\F^{-1}[\hat K_{1}^+(t - \tau,\cdot) \chi_{low }] (t - \tau,\cdot) * f(u(\tau,\cdot))\right) }{{2}}d\tau
\\
&\quad 	+
	\int_{t/2}^t \Norm{(- \Delta)^{\frac{s}{2}}  \left(\F^{-1}[\hat K_{1}^+(t - \tau,\cdot) \chi_{low }] (t - \tau,\cdot) * f(u(\tau,\cdot))\right) }{{2}}d\tau
	\\
&	\quad +	\int_0^t \Norm{(- \Delta)^{\frac{s}{2}}  
	\left(\F^{-1}[\hat K_{1}^+(t - \tau,\cdot) \chi_{mh}                                   ](t - \tau,\cdot) * f(u(\tau,\cdot))\right) }{{2}}d\tau
	\\
&	:= J_1^+ + J^+_2+ J^+_3 \;(\text{we put}). 
		\end{aligned}
	\end{equation}
Substituting \eqref{K1lowest+}  with $\vartheta = 0$, $q_1 = q_2 = r$, $s_1 = s $ and $s_2 = 0$ and \eqref{f(u)Lr} into $J_1^+$, and using \eqref{less-1}, we obtain
		\begin{equation}
		\begin{aligned}
		J_1^+ 
		&	\lesssim
			\langle t \rangle ^{{ \frac{1}{1 - \sigma} \left(-\frac{n}{2}(\frac{1}{r} - \frac{1}{2}) - \frac{s }{2} + \sigma \right) }} 
				\int_0^{t/2} 
		\langle \tau \rangle ^{{ \frac{1}{1 - \sigma} 
				\left( -\frac{n}{2}(\frac{p-1}{r}) + p \sigma  \right) }}d\tau 
								\|u\|_{X}^p
						\\
		\label{J1ineq}
&	\lesssim
		\langle t \rangle ^{{ \frac{1}{1 - \sigma} \left(-\frac{n}{2}(\frac{1}{r} - \frac{1}{2}) - \frac{s }{2} + \sigma \right) }}	\|u\|_{X}^p. 
		\end{aligned}
		\end{equation}
By \eqref{K1lowest+}  with $\vartheta = 0$, $s_1 = s, \; s_2 =[s ]$, $q_1 = q_2 = \tilde {q}_s $
(defined by \eqref{tildeqdef}),  and  \eqref{nonlinear1}, we have
\begin{align}
J_2^+ & \lesssim \int_{t/2}^{t}
\langle t - \tau  \rangle ^
{{ \frac{1}{2(1 - \sigma)} \left( - n (\frac{1}{\tilde q_s } - \frac{1}{2})  - s  + [s ] + 2 \sigma \right) }} 
	\Norm{(-\Delta)^{\frac{[s ]}{2}}   f(u(\tau,\cdot) )}{{{\tilde q_s } }}d\tau	
\nonumber\\
& \lesssim
\langle t \rangle ^
	{{	\frac{1}{1- \sigma}
			\left(
			(-\frac{n}{2r} + \sigma)p + \frac{n}{4} + \frac{1}{2}  - \frac{s }{2} 							
			\right)
		}}		
		\|u\|_{X}^p
		\int_0^{t/2} 
\langle \tau \rangle ^	
{{ \frac{1}{2(1 - \sigma)} \left( - 1 + 2 \sigma \right) }} d\tau
\nonumber \\
& \sim  \langle t \rangle ^
{{	\frac{1}{1- \sigma}
		\left(
		(-\frac{n}{2r} + \sigma)p + \frac{n}{4}  - \frac{s }{2} + 1 							
		\right)
	}}		
	\|u\|_{X}^p 
\nonumber\\
& 	= \langle t \rangle ^
	{{	\frac{1}{1- \sigma}
			\left(-\frac{n}{2}(\frac{1}{r} - \frac{1}{2})- \frac{s }{2} + \sigma 
			- (p-1)(\frac{n}{2r} - \sigma) + 1							
			\right)
		}}		
		\|u\|_{X}^p. 
\label{J2+}\end{align}
	Last we estimate $J^+_3$. 
Combining \eqref{high_der}, \eqref{mid_der1} and \eqref{nabla_beta_-}, we have
\begin{align}
J^+_3
& \lesssim
\int_0^t 
e^{-\vare_\sigma(t-\tau)}
\langle \tau  \rangle ^
{{
		\frac{1}{1 - \sigma} 
		\left( (-\frac{n}{2r}+\sigma)p + \frac{n}{4} + \frac{1}{2}  - \frac{s }{2}
		\right) 
	}} d \tau
	\|{u} \|_{X}^p 
\nonumber \\
& \lesssim \langle t \rangle ^
{{
		\frac{1}{1 - \sigma} 
		\left( (-\frac{n}{2r}+\sigma)p + \frac{n}{4} + \frac{1}{2}  - \frac{s }{2}
		\right) 
	}} 
		\|u\|_{X}^p
\nonumber\\
& 		= \langle t \rangle ^
		{{
				\frac{1}{1 - \sigma} 
				\left(-\frac{n}{2}(\frac{1}{r} - \frac{1}{2})+\sigma - \frac{s }{2}
			- (p-1)(\frac{n}{2r} - \sigma)	+ \frac{1}{2} 
				\right) 
			}} 
			\|u\|_{X}^p. 
\label{J3ineq}	
\end{align}
The assumption \eqref{pass} means $ - (p-1)(\frac{n}{2r} - \sigma) + 1 < 0$.  
Thus, \eqref{conv+_der} follows from \eqref{conv_der1} -- \eqref{J3ineq}. 

	We divide the left-hand side of  \eqref{conv-_der} into three parts:
		\begin{equation}\label{conv-_der1}
	\begin{aligned}
	&\int_0^t \Norm{(- \Delta)^{\frac{s}{2}}  \left(K_{1}^-(t - \tau,\cdot) * f(u(\tau,\cdot))\right) }{{2}}d\tau
	\\
	&=	\int_0^{t/2} \Norm{(- \Delta)^{\frac{s}{2}}  
		\left(\F^{-1}[\hat K_{1}^-(t - \tau,\cdot) \chi_{low }] (t - \tau,\cdot) * f(u(\tau,\cdot))\right) }{{2}}d\tau
	\\
	&\quad 	+
	\int_{t/2}^t \Norm{(- \Delta)^{\frac{s}{2}}  \left(\F^{-1}[\hat K_{1}^-(t - \tau,\cdot) \chi_{low }] (t - \tau,\cdot) * f(u(\tau,\cdot))\right) }{{2}}d\tau
	\\
	&	\quad +	\int_0^t \Norm{(- \Delta)^{\frac{s}{2}}  
		\left(\F^{-1}[\hat K_{1}^-(t - \tau,\cdot) \chi_{mh}                                   ](t - \tau,\cdot) * f(u(\tau,\cdot))\right) }{{2}}d\tau
	\\
	&	=: J^-_1 + J^-_2+ J^-_3 \;(\text{we put}). 
	\end{aligned}
	\end{equation}
	Substituting \eqref{K1lowest-}  with $\vartheta = 0$, $q_1 = q_2 = r$, $s_1 = s $ and $s_2 = 0$ and \eqref{f(u)Lr} into $J^-_1$, and using \eqref{less-1}, we obtain
	\begin{equation}
	\begin{aligned}
	J^-_1 
	&	\lesssim
	\langle t \rangle ^{{ \frac{1}{ \sigma} \left(-\frac{n}{2}(\frac{1}{r} - \frac{1}{2}) - \frac{s }{2} + \sigma \right) }} 
	\int_0^{t/2} 
	\langle \tau \rangle ^{{ \frac{1}{1 - \sigma} 
			\left( -\frac{n}{2}(\frac{p-1}{r}) + p \sigma  \right) }}d\tau 
	\|u\|_{X}^p
	\\
	\label{J1ineq-}
	&	\lesssim
	\langle t \rangle ^{{ \frac{1}{\sigma} \left(-\frac{n}{2}(\frac{1}{r} - \frac{1}{2}) - \frac{s }{2} + \sigma \right) }}	\|u\|_{X}^p. 
	\end{aligned}
	\end{equation}
Since $\sigma < 1-\sigma$, 
the right-hand side of \eqref{K1lowest-} is dominated by that of \eqref{K1lowest+}.  
Hence, $J_2^-$ and $J_3^-$ are estimated by the right-hand sides of \eqref{J2+} and \eqref{J3ineq}, respectively, and thus,
\begin{align}
\label{J2ineq-}
J^-_2  \lesssim
\langle t \rangle ^
	{{	\frac{1}{1- \sigma}
			\left(-\frac{n}{2}(\frac{1}{r} - \frac{1}{2})- \frac{s }{2} + \sigma 
			- (p-1)(\frac{n}{2r} - \sigma) + 1							
			\right)
		}}		
		\|u\|_{X}^p, 
		\\
 J^-_3 \lesssim
 \langle t \rangle ^
{{	\frac{1}{1- \sigma}
		\left(-\frac{n}{2}(\frac{1}{r} - \frac{1}{2})- \frac{s }{2} + \sigma 
		- (p-1)(\frac{n}{2r} - \sigma) + \frac{1}{2}							
		\right)
	}}		
	\|u\|_{X}^p. 		
\label{J3ineq-}
\end{align}  	
Substituting \eqref{J1ineq-}, \eqref{J2ineq-} and \eqref{J3ineq-} into \eqref{conv-_der1}, we obtain   \eqref{conv-_der}.  

Inequality \eqref{conv_der} follows from \eqref{conv+_der} and \eqref{conv-_der}, since $\sigma < 1 -\sigma$ and  
$- (p-1)(\frac{n}{2r} - \sigma) + 1	< 0$.  
		
By using the assumption \eqref{fass}, we can prove \eqref{conv_der2} in the same way.  
	\end{proof}

              \subsection{Diffusion estimate}
	
\begin{lemma} \label{f(u)derdiff2} 
		Let $\delta$ and $\nu$ be an arbitrary number satisfying the assumption of 
	Theorem \ref{thmdiff}.  
Let $X = X_{1,\delta,\bar{s}}$, where $X_{1,\delta,\bar{s}}$ is defined by \eqref{Xdef}.  
Then we have 
\begin{equation}
	\begin{aligned}
\Big\|
			&	\int_0^t 	
K_1(t-\tau,\cdot)		* f(u(\tau,\cdot)) d\tau
 -  G_{\sigma}(t,\cdot) 
		\int_0^\infty \int_{\R^n}f(u(\tau,y))dy d\tau		
\Big\|_{{2}}
\\	
& \lesssim
t 
^{{\max \{
			 \frac{1}{1 - \sigma} \left(-\frac{n}{4} + \sigma  
		- \min \{( p - 1)\left(\frac{n}{2} 
			\right)  - 1, 1 - 2 \sigma, \nu \} 
		\right),  
	\frac{1}{\sigma}		 \left(-\frac{n}{4} +  \sigma \right) 
	\} }}
\|u\|_{X}^p. 
\label{heat+}	
\end{aligned}
\end{equation}
\end{lemma}

\begin{proof}
We have
	\begin{align}\label{L1234}
&\Big\|\int_0^t 
		K_1(t -\tau,\cdot) * f(u(\tau,\cdot)) d\tau			- G_{\sigma}(t,\cdot)\int_0^\infty \int_{\R^n}f(u(\tau,y))dy d\tau	\Big\|_2 
\nonumber 
\\
&\quad \le L_1 + L_2 + L_3 + L_4,
\end{align}
where
\begin{align*}
L_1	&:=  \int_{\frac{t}{2}}^{t}	
			\|K_1(t-\tau,\cdot)*f(u(\tau,\cdot)) \|_2 d\tau, 
			\\
L_{2} &:= \int_0^{t/2} 
\left\|K_1(t-\tau,\cdot) * f(u(\tau,\cdot))
  - G_{\sigma}(t - \tau,\cdot) \int_{\R^n} f(u(\tau,y))dy \right\|_2 d\tau,
\\
L_3 &:= \int_0^{t/2} 
\Big\|	G_{\sigma}(t-\tau,\cdot) -  G_{\sigma}(t,\cdot)
\Big\|_2
\left|\int_{\R^n} f(u(\tau,y))dy \right|
 d\tau, 
\\
L_4 &:= 
	\| G_{\sigma}(t,\cdot) \|_2 
				\left|\int_{t/2}^\infty \int_{\R^n}f(u(\tau,y))dy d\tau \right|. 
									\end{align*}
				
				First we estimate $L_1$ by dividing the integrand as
\begin{align*}
\| K_1(t-\tau,\cdot)*f(u(\tau,\cdot)) \|_2 
&= \|K_{1,low}(t-\tau,\cdot)* f(u(\tau,\cdot)) \|_2 
\\
& \quad  + \|K_{1,mh}(t-\tau,\cdot) * f(u(\tau,\cdot)) \|_2.  
\end{align*}		
Taking $q_1 = q_2 = \frac{2n}{n + 4 \sigma}$
and $s_1 = s_2 = \vartheta = 0$ in \eqref{K1lowweight}, we obtain 
\begin{equation}
\| K_{1,low}(t-\tau, \cdot)* f(u(\tau,\cdot)) \|_2 \lesssim \|f(u(\tau,\cdot))\|_{\frac{2n}{n + 4 \sigma}}
\lesssim \|u(\tau,\cdot)\|_{{\frac{2np}{n + 4 \sigma}}}^p. 
\label{L1lowest}\end{equation}
Since $\bar{s} \ge 1$, the assumption \eqref{pass}  implies 
$$
\frac{2np}{n + 4 \sigma}  > \frac{2n}{n + 2\delta }, 
$$
and \eqref{pass2} implies 
$$
\frac{2np}{n + 4 \sigma} \le \frac{2n}{n - 2 \bar{s}} 
$$
if $2 \bar{s} < n$. 
Thus, we can apply \eqref{u_q} with $r = 1$ and 
$q = \frac{2np}{n + 4 \sigma}$ to obtain
\begin{equation*}
\begin{aligned}
\|u(\tau,\cdot)\|_{\frac{2np}{n + 4 \sigma}}^p
&\lesssim 
\langle \tau \rangle ^{{ \frac{1}{1 - \sigma} 
		\left(
		-\frac{n}{2}( p - \frac{1}{2} - \frac{2 \sigma}{n}) +  p \sigma 
		\right)	
	}} \|u\|_{X}^p
\\
&=	\langle \tau \rangle ^{{ \frac{1}{1 - \sigma} 
			\left(
					- ( p - 1)\left(\frac{n}{2} -  \sigma \right)  - \frac{n}{4} +2 \sigma
			\right)	
		}}\|u\|_{X}^p. 
\end{aligned}\end{equation*}
From the inequality above and \eqref{L1lowest}, it follows that
\begin{equation}\begin{aligned}
\label{L1lowest2}
&\int_{\frac{t}{2}}^{t}	
\|K_{1,low}(t - \tau,\cdot)*f(u(\tau,\cdot)) \|_2 d\tau
\lesssim  
\langle t \rangle ^{{ \frac{1}{1 - \sigma} 
		\left(
		-\frac{n}{4} +  \sigma 
		- ( p - 1)\left(\frac{n}{2} - \sigma \right)  + 1
		\right)	
	}}\|u\|_{X}^p. 
\end{aligned}	\end{equation}	
	By \eqref{J3ineq} and \eqref{J3ineq-} with $r = 1$ and $s = 0$, we have 
	\begin{align*}
\int_{\frac{t}{2}}^{t}	
\|K_{1,mh} (t-\tau, \cdot)*f(u(\tau,\cdot)) \|_2 d\tau
		&	\lesssim
		\langle t \rangle ^
		{{
				\frac{1}{1 - \sigma} 
				\left(-\frac{n}{4}+\sigma
				+ \frac{1}{2} - (p-1)(\frac{n}{2} - \sigma)
				\right) 
			}} 
			\|u\|_{X}^p,
					\end{align*}
	which together with \eqref{L1lowest2} yields
	\begin{equation}\label{L1est}
L_1 \lesssim  
	\langle t \rangle ^{{ \frac{1}{1 - \sigma} 
			\left(
			-\frac{n}{4} +  \sigma 
			- ( p - 1)\left(\frac{n}{2} - \sigma \right)  + 1
			\right)	
		}}\|u\|_{X}^p. 
		\end{equation}		

By \eqref{I2est} with $u_1 =f(u(\tau,\cdot))$ and 
$\theta = 2 \tilde \nu$, we have
\begin{equation}
	\begin{aligned}
	L_2 \lesssim 
& \langle t \rangle^{{
			\max \{
			\frac{1}{1 - \sigma}\left(-\frac{n}{4} +  3\sigma - 1 \right),
			\frac{1}{\sigma}		 \left(-\frac{n}{4} +  \sigma \right) 
			\} }}
\int_0^{t/2}  \| f(u(\tau,\cdot)) \|_{1} d\tau
	\\
	&  + e^{{-\frac{\vare_\sigma t}{2}}} \int_0^{t/2}
	\| f(u(\tau,\cdot)) \|_{{H^{-2 \sigma}}} d\tau
\\
&	+	 	 t^{ \frac{1}{1-\sigma}(-\frac{n}{4} + \sigma - \tilde \nu )}
	\int_0^{t/2} \| |\cdot|^{2 \tilde \nu} f(u(\tau,\cdot)) \|_{{{1}}} d\tau
		\\
 =: & L_{2,1} + L_{2,2} + L_{2,3} \: (\text{we put}).  			\end{aligned}
\end{equation}
Inequality \eqref{f(u)Lr} with $r = 1$ and \eqref{less-1} yield 
\begin{align*}
\int_0^{t/2} \Norm{  f(u(\tau,\cdot)) }{1} d\tau
&\lesssim  	
\int_0^{t/2}\langle \tau \rangle
^{{ \frac{1}{1 - \sigma} 
		\left( -\frac{n}{2}(p-1) + p \sigma  \right) }}d \tau
\|u\|_{X}^p
 \lesssim  \|u\|_{X}^p. 
\end{align*}
Thus
\begin{equation}\label{L21}
L_{2,1} \lesssim \langle t \rangle^{{
		\max \{
		\frac{1}{1 - \sigma}\left(-\frac{n}{4} +  3\sigma - 1 \right),
		\frac{1}{\sigma}		 \left(-\frac{n}{4} +  \sigma \right) 
		\} }}\|u\|_{X}^p. 
\end{equation}
Since $\bar{s} \ge 1$, \eqref{pass2} and \eqref{delta2} with $r = 1$ imply  
$2p  > \frac{2n}{n + 2\delta }$ and $2p \le \frac{2n}{n - 2 \bar{s}}$ if $2 \bar{s} < n$.   
Hence, we can use \eqref{u_q} with $r = 1$ and $q = 2p$ to obtain 
\begin{equation*}
\begin{aligned}
\| f(u(\tau,\cdot)) \|_{{H^{-2 \sigma}}}
& \lesssim \|u(t,\cdot)\|_{2p}^p \lesssim
\langle t \rangle ^{{\frac{p}{1 - \sigma} \left(-\frac{n}{2}(1 - \frac{1}{2p}) +  \sigma 	\right) }}\|u\|_{X}^p \lesssim \|u\|_{X}^p. 
\end{aligned}	
\end{equation*}
Thus
\begin{equation}\label{L22}
L_{2,2} 
\le 
e^{{-\vare_\sigma t/2}} \|u\|_{X}^p. 
\end{equation}
We estimate $L_{2,3}$.  
Let $\tilde{\nu}$ be an arbitrary number satisfying
\begin{equation}
\label{nu'def}
0 < 
\nu  < \tilde{\nu} < 
\min  
\big \{\frac{n}{4}(p-2) + \frac{1}{2}p \delta, 
\delta 
\big\}.   
\end{equation}
Assume moreover that 
\begin{equation}\label{nu'def2}
\tilde{\nu} \le \frac{\delta}{2 \bar{s}}(n - \frac{p}{2}(n - 2 \bar{s})), 
\end{equation}
if $\bar{s} < \frac{n}{2}$. 
By H\"{o}lder's inequality, we have 
\begin{equation}
\begin{aligned}
\left\|| \cdot |^{\tilde{2 \nu}}  f(u(\tau,\cdot)) \right\|_1 
\lesssim \||\cdot|^{\frac{2 \tilde{\nu}}{p} }u(\tau,\cdot) \|_p^p 
& \le 
\|(|\cdot|^\delta |u(\tau,\cdot)|)^{\frac{2 \tilde \nu}{p\delta}}\|_{\frac{p\delta}{\tilde{\nu}}}^p   \;
\||u(\tau,\cdot)|^{1 -\frac{2 \tilde \nu}{p\delta}} \|_q^p
\\
&
= \| |\cdot|^\delta u(\tau,\cdot)\|_{2}^{\frac{2 \tilde \nu}{\delta}} \;
\|u(\tau,\cdot)\|_{\tilde q}^{p -\frac{2 \tilde \nu}{\delta}},
\label{xnu/pu} \end{aligned}\end{equation}
where $q = \frac{ p \delta}{ \delta - \tilde{\nu}}$ and 
$\tilde q =  q(1 -\frac{2 \tilde \nu}{p\delta})$. 
The assumption \eqref{nu'def} implies
$$
\tilde q = q (1 -\frac{2\tilde{\nu}}{p\delta})= \frac{p \delta -2 \tilde{\nu}}{ \delta - \tilde{\nu}} 
> \frac{2n}{n + 2\delta }. 
$$ 
In fact, the condition 
$
\tilde{\nu} < \frac{n}{4}(p-2) + \frac{p \delta}{2} 
$
is equivalent to 
$
\tilde q = \frac{p \delta - 2 \tilde{\nu}}{ \delta - \tilde{\nu}} 
> \frac{2n}{n + 2\delta }
$. 
The condition \eqref{nu'def2} is equivalent to 
$
\tilde q= \frac{p \delta - 2 \tilde{\nu}}{ \delta - \tilde{\nu}} \le \frac{2n}{n - 2 \bar{s}}
$ in the case $n > 2 \bar{s}$.   
Hence, using \eqref{u_q} with taking $q$ as $\tilde q$, and definition of $\|\cdot \|_X$ with $r = 1$ (see \eqref{Xnormdef}) in the right-hand side of \eqref{xnu/pu}, we obtain
\begin{equation}\begin{aligned}
\left\|| \cdot |^{\tilde{2 \nu}}  f(u(\tau,\cdot)) \right\|_1 
& \lesssim 
\langle \tau \rangle ^
{{\frac{1}{1 - \sigma} 
		\left(-\frac{n}{2}(1 - \frac{1}{2}) +  \sigma + \frac{\delta}{2} \right) 
		\frac{2 \tilde \nu}{\delta}	
}}
\langle \tau \rangle ^{{\frac{1}{1 - \sigma} 
		\left(-\frac{n}{2}(1 - \frac{1}{\tilde q})  + \sigma \right) 
		(p -\frac{2 \tilde \nu}{\delta})
}}
\|u\|_X^p
\\
&= \langle \tau \rangle 
^{{\frac{1}{1 - \sigma} \left((-\frac{n}{2} + \sigma )p 
		+ \frac{n}{2} + \tilde \nu \right)}}
\|{u}\|_{X}^p.  \;
\label{gest2} \end{aligned}\end{equation}
Thus,
\begin{equation*}
L_{2,3} \lesssim 
\langle t \rangle^{\frac{1}{1-\sigma}(-\frac{n}{4} + \sigma - \tilde \nu)} \; 
\int_0^{t/2}\langle \tau \rangle ^{{\frac{1}{1 - \sigma} 
		\left( \left(-\frac{n}{2} + \sigma \right)p + \frac{n}{2} + \tilde \nu\right)}} 
d\tau
\|{u}\|_{X}^p,
\end{equation*}
which yields
\begin{align}
L_{2,3}  
&\lesssim 
\begin{cases}
\langle t \rangle ^
{{ \frac{1}{1 - \sigma} 
		\left(
		-\frac{n}{4} +  \sigma 
		- ( p - 1)\left(\frac{n}{2} - \sigma \right)  + 1
		\right)	
}}
&\text{if} \quad  
\frac{1}{1 - \sigma} \left(-\frac{n}{2}
(p-1) + p \sigma + \tilde \nu\right) > -1, 
\nonumber	\\
\langle t \rangle ^{{ \frac{1}{1 - \sigma} \left(-\frac{n}{4} + \sigma - \tilde \nu\right) }} 
\log (\langle t \rangle + 1)
&\text{if} \quad 
\frac{1}{1 - \sigma} \left(-\frac{n}{2}
(p-1) + p \sigma + \tilde \nu\right) = -1,
\nonumber	\\
\langle t \rangle ^{{ \frac{1}{1 - \sigma} \left(-\frac{n}{4} + \sigma - \tilde \nu\right) }} 
&\text{if} \quad 
\frac{1}{1 - \sigma} \left(-\frac{n}{2}
(p-1) + p \sigma + \tilde \nu\right)< -1. 
\end{cases}
\nonumber	\\
& \lesssim
\langle t \rangle ^{{ \frac{1}{1 - \sigma} \left(-\frac{n}{4} + \sigma 
		- \min \{( p - 1)\left(\frac{n}{2} - \sigma \right)  - 1, \nu \} 
		\right) }}.  
\label{L23}
\end{align}
Inequalities \eqref{L21}, \eqref{L22} and \eqref{L23} yield
\begin{equation}\label{L2est}
L_{2} \lesssim \langle t \rangle^{{
		\max \{
	\frac{1}{1 - \sigma} 
	\left(-\frac{n}{4} + \sigma 
			- \min \{1 - 2\sigma, ( p - 1)\left(\frac{n}{2} - \sigma \right)  - 1, \nu \} 
			\right),\;
		\frac{1}{\sigma}		 \left(-\frac{n}{4} +  \sigma \right) 
	\} 	
	}}
	\|u\|_{X}^p. 
\end{equation}

We estimate $L_3$.  
By the definition of $G_\sigma$,
\begin{align*}
\F( G_{\sigma}(t-\tau,\cdot) -  G_{\sigma}(t,\cdot))
 & =
 \tau  \int_0^{1}\spd{\hat G_{\sigma}}{t}(t- \theta \tau, \xi) 	d \theta 
\\
& =   \tau \int_0^{1} |\xi|^{2(1-2\sigma) } e^{- |\xi|^{2(1 - \sigma)} (t- \theta \tau)}
 d \theta. 
\end{align*}
Inequality \eqref{Gest02} with $\theta = 2 - 2 \sigma$ means
\begin{equation}
\begin{aligned}
 \left\| |\xi|^{2(1-2\sigma) } e^{- |\xi|^{2(1 - \sigma)} (t- \theta \tau)} \right\|_2 
& \sim  
 t^{{ \frac{1}{1 - \sigma} \left(-\frac{n}{4} -  1 + 2 \sigma \right) }} 
 \end{aligned}
\end{equation}
uniformly to $\theta \in [0,1]$ and $\tau \in [0,t/2]$.  
Then   by using \eqref{f(u)Lr} with $r = 1$ together, we have
\begin{equation}
	\begin{aligned}
	L_3
&\lesssim 	t^{{ \frac{1}{1 - \sigma} \left(-\frac{n}{4} -  1 + 2 \sigma \right) }} 
\int_0^{t/2} \tau \|f(u(\tau,\cdot))\|_1 d\tau
			\\
	&	\lesssim  
	t^{{ \frac{1}{1 - \sigma} 
	\left(-\frac{n}{4} -  1 + 2 \sigma 
	\right) }} 
	\int_0^{t/2} \tau
	\langle \tau \rangle ^{{ \frac{1}{1 - \sigma} 
			\left(-\frac{n}{2}
			(p-1) + p \sigma \right) }}
	d\tau, 		
		\end{aligned}
\end{equation}
which yields
\begin{align}
L_3 
&\lesssim  
\begin{cases}
{t}^
 {{ \frac{1}{1 - \sigma} 
 		\left(
 	-\frac{n}{4} +  \sigma 
 		- ( p - 1)\left(\frac{n}{2} - \sigma \right)  + 1
       \right)	
    }}
&\text{if} \quad 
\frac{1}{1 - \sigma} \left(-\frac{n}{2}
(p-1) + p \sigma \right) > -2
\\
{t}^{{ \frac{1}{1 - \sigma} \left(-\frac{n}{4} -  1 + 2 \sigma \right) }} 
\log (\langle t \rangle + 1)
\quad 
&\text{if} \quad 
\frac{1}{1 - \sigma} \left(-\frac{n}{2}
	(p-1) + p \sigma \right) = -2
\\
{t}^{{ \frac{1}{1 - \sigma} \left(-\frac{n}{4}  -1 + 2 \sigma \right) }} 
\quad 
&\text{if} \quad 
\frac{1}{1 - \sigma} \left(-\frac{n}{2}
 (p-1) + p \sigma \right)< - 2. 
\end{cases}
\nonumber \\
&\le 
{t}^{{ \frac{1}{1 - \sigma} \left(-\frac{n}{4} + \sigma - 
\min \{	 ( p - 1)\left(\frac{n}{2} - \sigma \right) - 1, 1 - 2 \sigma \}
				\right) }}. 		
\label{L3est} \end{align}

Last we estimate $L_4$.  
Since $n \ge 2$, the assumption \eqref{delta} implies
$$
\frac{2n}{n + 2\delta} \le \frac{2n}{2n - 2} = 1 + \frac{1}{n-1} \le 1 + \frac{2}{n - 2\sigma},
$$
which together with \eqref{pass} yields $p > \frac{2n}{n + 2\delta}$. 
By this fact and assumption \eqref{pass2}, we can apply \eqref{u_q} with 
$r = 1$ and $q = p$ to obtain
$$
\| u(\tau,\cdot) \|_p \lesssim \langle \tau \rangle ^{{\frac{1}{1 - \sigma} \left(-\frac{n}{2}(1 - \frac{1}{p})  + \sigma \right) }}\|{u}\|_{X}.  
$$
This together with \eqref{less-1} yields
\begin{align}\label{heat+2}
\int_{t/2}^\infty \left|\int_{\R^n}f(u(\tau,y))dy  \right| d\tau
&\lesssim 
\int_{t/2}^\infty \langle \tau \rangle ^{{\frac{1}{1 - \sigma} \left(-\frac{n}{2}(p - 1)  + \sigma p \right) }} d\tau \|{u}\|_{X}^p
\nonumber \\
&\sim  \langle t \rangle ^{{\frac{1}{1 - \sigma} \left(-(\frac{n}{2} - \sigma)(p - 1)  + 1 \right) }}
\|{u}\|_{X}^p. 
\end{align} 
Taking the product of \eqref{Gsigmadecay} and \eqref{heat+2}, we obtain
\begin{equation}
L_4 \lesssim 
{t}^{{ \frac{1}{1 - \sigma} 
		\left(
		-\frac{n}{4} +  \sigma 
		- ( p - 1)\left(\frac{n}{2} - \sigma \right)  + 1
		\right)	
	}}. 
\label{L4est}
\end{equation}	
Substituting  \eqref{L1est}, \eqref{L2est}, \eqref{L3est} and \eqref{L4est} 
into \eqref{L1234}, we obtain \eqref{heat+}.  
\end{proof}

\section{Proof of Proposition and Theorems}

\subsection{Proof of Proposition \ref{existence}}

Let $\vare> 0$, and 
$$
X(\varepsilon) = X_{r,\delta, \bar{s}}(\vare) 
:= \{ \varphi \in X_{r,\delta, \bar{s}}; \| \varphi \|_{{X_{r,\delta, \bar{s}}}} < \vare \},
$$
where $X_{r,\delta, \bar{s}}$ and $\| \varphi \|_{{X_{r,\delta, \bar{s}}}}$ are defined by \eqref{Xdef} and \eqref{Xnormdef}, respectively.  
We put $X = X_{r,\delta, \bar{s}}$ and $\|\cdot\|_X =\| \varphi \|_{{X_{r,\delta, \bar{s}}}}$,   throughout this subsection.   

If $u$ is a solution of \eqref{NW}, then Duhamel's principles implies 
\begin{align*}
u(t,x)=K_0(t,\cdot) \ast u_0 +K_1(t,\cdot) \ast u_1 + \int _0 ^t K_1(t- \tau,\cdot)\ast f(u( \tau,\cdot) ) d\tau, 
\end{align*}	                    
where $K_0$ and $K_1$ are defined by \eqref{K0def} and \eqref{K1def}. 
Taking account of the formula above,
we define the mapping $\Phi$ on $X(\vare)$ by
\begin{equation}
\label{Phidef}
	(\Phi u)(t) := K_0(t,\cdot)\ast u_0+K_1(t,\cdot)\ast u_1 
	+ \int_0^t K_1(t - \tau,\cdot)* f(u(\tau,\cdot)) d \tau. 
\end{equation}
	We prove that $\Phi$ is a contraction mapping on $X(\vare)$ provided $\vare$ and initial data are sufficiently small.   
		
First we estimate $K_1(t,\dot)*u_1$. 
By \eqref{delta}, we see that the assumption \eqref{qass2} of Lemma \ref{Kweight} is satisfied for    
$\vartheta = 0$ and $\delta$, $s_1 = s_2  = 0$, $q_1 = r$ and 
$q_2 = \frac{nr}{n - r \vartheta}(\in [r,2))$ (that is, $\frac{1}{q_2} = \frac{1}{r} - \frac{\vartheta}{n}$).   
Then, \eqref{K1lowweight}                                
gives  estimate of low frequency part.  
The high and middle frequency parts are given by \eqref{high_weight} and  \eqref{mid_weight1}.   
	Then we have
		\begin{equation}
	\begin{aligned}
\Norm{\langle \cdot \rangle^\delta  \left(K_{1}(t,\cdot)*u_1 \right) }{{2}}
& 	\lesssim
	\langle t \rangle ^{{ \frac{1}{1 - \sigma} 
			\left(
			-\frac{n}{2}(\frac{1}{r} - \frac{1}{2}) + \frac{\delta}{2} + \sigma \right) }} 	
	(\|u_1 \|_{r}^\prime + 
	\| \langle \cdot \rangle^\delta u_1 \|_{{{{ \frac{nr}{n - r \delta}, 2}} }})
\\
&	\qquad + 	e^{-\vare_{\sigma} t} 
	\|(1 - \Delta)^{-\frac{1}{2}}
           \langle \cdot \rangle^\delta u_1 \|_{{{{2}} }}. 
	\end{aligned}
	\label{u11weight1}
	\end{equation}			
	Assumption \eqref{delta} implies  
	$
	\frac{ n - r \delta}{r} - 1 \le \frac{n}{2}
	$, 
	and therefore, sharp Sobolev's embedding theorem (Lemma E) yields 
	\begin{equation*}
		\|(1 - \Delta)^{-\frac{1}{2}} \langle \cdot \rangle^\delta  u_1 \|_{{2}} 
	\le	\| \langle \cdot \rangle^\delta  u_1 \|_{{{{ \frac{nr}{n - r \delta}},2 }}}.  
	\end{equation*}
	Substituting this inequality into \eqref{u11weight1}, we obtain
			\begin{equation}
	\begin{aligned}
	&	\Norm{\langle \cdot \rangle^\delta  \left(K_{1}(t,\cdot)*u_1 \right) }{{2}}
	\lesssim
	\langle t \rangle ^{{ \frac{1}{1 - \sigma} 
			\left(
			-\frac{n}{2}(\frac{1}{r} - \frac{1}{2}) + \frac{\delta}{2} + \sigma \right) }} 	
	(\|u_1 \|_{r}^\prime + 
	\| \langle \cdot \rangle^\delta u_1 \|_{{ \frac{nr}{n - r \delta}, 2}} ). 
	\end{aligned}
	\label{u11weight}
	\end{equation}	
Inequality \eqref{K1lowweight} with $\vartheta = 0$, $s_1 = \bar{s}$, $s_2 = 0$,  
	$q_1 =q_2 = r$, and inequalities  \eqref{high_der} and \eqref{mid_der1}  with $s = \bar{s}$ yield
		\begin{equation}\label{K1u1s0}
		\begin{aligned}
		&\Norm{(-\Delta)^{\frac{\bar{s}}{2}}  \left(K_{1}(t,\cdot) * u_1 \right) }{{2}}
		\\
		&\lesssim
			\langle t \rangle ^{{ \frac{1}{1 - \sigma} 
				\left(
				-\frac{n}{2}(\frac{1}{r} - \frac{1}{2}) - \frac{\bar{s}}{2} + 
				\sigma \right) }} 	
			\|u_1\|_{{r}}^\prime 
			+ 	
		e^{-\vare_{\sigma} t} 
		\|  (-\Delta)^{\frac{\bar{s}}{2}} (1 - \Delta)^{-\frac{1}{2}}  u_1 \|_{{2}}. 	
						\end{aligned}
			\end{equation}

Next we estimate $K_0(t,\dot)*u_0$. 
By \eqref{delta}, we see that the assumption  of Lemma \ref{Kweight} is satisfied for    
$$
\vartheta   = 0 \;\text{and}\; \delta, \quad s_1 = s_2 = 0, \quad
q_3 = \frac{nr}{n - 2 r \sigma}, \quad 
q_4 = \frac{nr}{n - r(\vartheta + 2\sigma)}. 
$$
Then, \eqref{low0weight}                      
gives estimate of low frequency part.  
We estimate high and middle frequency parts by \eqref{high0_weight} and \eqref{mid0_weight1}.  
Then  we obtain
							\begin{equation*}
							\begin{aligned}
							&\Norm{\langle \cdot \rangle^\delta  \left(K_{0}(t,\cdot) *u_0 \right) }{{2}}
\\
				&	\lesssim
	\langle t \rangle ^{{ \frac{1}{1 - \sigma} 
			\left(-\frac{n}{2}(\frac{1}{r} - \frac{1}{2}) 
			+ \frac{\delta}{2} + \sigma \right) }} 	
					(\|u_0 \|_{{ \frac{nr}{n -2 r \sigma}, 2}}	+ 											
			 \|\langle \cdot \rangle^\delta u_0 \|_{{{{\frac{nr}{n - r(\delta + 2\sigma)},2}} }})
							+							
					e^{-\vare_{\sigma} t} 
				\| \langle \cdot \rangle^\delta  	u_0 \|_{{{{2}} }}. 
							\end{aligned}
													\end{equation*}		
By Corollary A, we have
\begin{equation}\label{weightdeosaeru0}
\|u_0 \|_{{ \frac{nr}{n -2 r \sigma},2 }}
\lesssim \| | x |^{-\delta}\|_{{{\frac{n}{\delta}, \infty} }}
\|\langle \cdot \rangle^\delta u_0 \|_{{{{\frac{nr}{n - r(\delta + 2\sigma)},2}} }}
\lesssim  
\|\langle \cdot \rangle^\delta u_0 \|_{{{{\frac{nr}{n - r(\delta + 2\sigma)},2}} }}. 
\end{equation}	
Hence, we have
			\begin{equation}
			\begin{aligned}
			&\Norm{\langle \cdot \rangle^\delta  \left(K_{0}(t,\cdot) *u_0 \right) }{{2}}
			\\
			&	\lesssim
			\langle t \rangle ^{{ \frac{1}{1 - \sigma} 
					\left(-\frac{n}{2}(\frac{1}{r} - \frac{1}{2}) 
					+ \frac{\delta}{2} + \sigma \right) }} 	
			\|\langle \cdot \rangle^\delta u_0 \|_{{{{\frac{nr}{n - r(\delta + 2\sigma)},2}} }}
					+							
			e^{-\vare_{\sigma} t} 
			\| \langle \cdot \rangle^\delta  	u_0 \|_{{{{2}} }}. 
			\end{aligned}
			\label{u0weight}
			\end{equation}		
	Inequality \eqref{low0weight} with $\vartheta = 0$, $s_1 = \bar{s}$, $s_2 = 0$ and 
	$q_3 = q_4 =\frac{nr}{n - 2 \sigma r}$ and inequalities \eqref{high0_der} and \eqref{mid0_der1} with $s = \bar{s}$ imply
				\begin{equation}
					\begin{aligned}
&		\Norm{(-\Delta)^{\frac{\bar{s}}{2}}  \left(K_{0}(t,\cdot) * u_0 \right) }{{2}}
\\
&		\lesssim
			\langle t \rangle ^{{ \frac{1}{1 - \sigma} 
					\left(-\frac{n}{2}(\frac{1}{r} - \frac{1}{2}) - \frac{\bar{s}}{2} + \sigma  \right) }}
		\Norm{u_0}{{{\frac{nr}{n - 2 \sigma r} },2 }}
								+ 	
								e^{-\vare_{\sigma} t} \|  u_0 \|_{{H^{\bar{s}} }}	
								\\
&\lesssim
\langle t \rangle ^{{ \frac{1}{1 - \sigma} 
		\left(-\frac{n}{2}(\frac{1}{r} - \frac{1}{2}) - \frac{\bar{s}}{2} + \sigma  \right) }}
	\|\langle \cdot \rangle^\delta u_0 \|_{{{{\frac{nr}{n - r(\delta + 2\sigma)},2}} }}
+ 	
e^{-\vare_{\sigma} t} \|  u_0 \|_{{H^{\bar{s}} }}. 							
								\end{aligned}								
								\label{u0der}
								\end{equation}
For the last inequality, we used \eqref{weightdeosaeru0}.  

 By \eqref{Phidef}, \eqref{u11weight},  \eqref{K1u1s0}, \eqref{u0weight}, \eqref{u0der},  \eqref{conv_weight}  and \eqref{conv_der} with $s  = 0,\bar{s}$, we have
\begin{equation}\label{Phiest}
\begin{aligned}
\|\Phi u \|_X 
& \lesssim  
\|\langle \cdot \rangle^\delta u_0 \|_{{{{\frac{nr}{n - r(\delta + 2\sigma)}},2} }}
+ \|\langle \cdot \rangle^\delta u_0 \|_{{{{2}} }}
+  \|  u_0 \|_{{H^{\bar{s}} }}
\\
& \quad + 
	\|u_1\|_{r}^\prime + 
\|\langle \cdot \rangle^\delta u_1 \|_{{{{ \frac{nr}{n - r \delta},2}} }}
+ \|  (-\Delta)^{\frac{\bar{s}}{2}} (1 - \Delta)^{-\frac{1}{2}}  u_1 \|_{{2}} 
+  \|u \|_X^p.  
\end{aligned}
\end{equation}	

(i) First we consider the case $r = 1$.  
By \eqref{Phiest}, there is a positive constant $C_1$ independent of initial data such that
\begin{equation*}
\begin{aligned}
\|\Phi u \|_X 
& \le C_1 
\Big(
\|\langle \cdot \rangle^\delta u_0 \|_{{\frac{n}{n - (\delta + 2\sigma)}, 2}} 
+ \|\langle \cdot \rangle^\delta u_0 \|_{{{{2}} }}
+  \|  u_0 \|_{{H^{\bar{s}} }}
\\
& \quad + 
\|u_1\|_{{1}}
+ \|\langle \cdot \rangle^\delta u_1 \|_{{ \frac{n}{n -  \delta},2 }}
+ \|  (-\Delta)^{\frac{\bar{s}}{2}} (1 - \Delta)^{-\frac{1}{2}}  u_1 \|_{{2}} 
+  \|u \|_X^p \Big). 	
\end{aligned}
\end{equation*}	
Hence, taking $\vare_1 > 0$ such that $C_1 \vare_1^{p-1} \le \frac{1}{2}$, and assuming $u_0$ and $u_1$ satisfy
\begin{equation*}
\begin{aligned}
C_1  
&\Big(
\|\langle \cdot \rangle^\delta u_0 \|_{{\frac{n}{n - (\delta + 2\sigma)},2 }}
+ \|\langle \cdot \rangle^\delta u_0 \|_{2} 
\| (-\Delta)^{\frac{\bar{s}}{2}} u_0 \|_{{2}}
\\
& \quad + 
\|u_1\|_{{{1}}}
+ \|\langle \cdot \rangle^\delta u_1 \|_{{ \frac{n}{n -  \delta}, 2}}
+ \|  (-\Delta)^{\frac{\bar{s}}{2}} (1 - \Delta)^{-\frac{1}{2}}  u_1 \|_{{2}} 
\Big)
\le \frac{\vare_1}{2}, 
\end{aligned}
\end{equation*}		
$\Phi$ becomes a mapping $X(\vare_1)$ to $X(\vare_1)$.  

By \eqref{conv_weight2}   and \eqref{conv_der2}, there is a positive constant $C_2$ independent of initial data such that 
\begin{equation}\label{Phiu-v}
\|\Phi u - \Phi v \|_X \le C_2 (\|u \|_X + \|v \|_X)^{p-1} \|u - v \|_X
\end{equation}
for every $u, v \in X$.  
Thus, by taking $\vare \in (0,\vare_1)$ such that 
$$
C_2  (2 \vare)^{p-1} < 1,
$$
$\Phi$ becomes a contraction mapping from $X(\vare)$ to $X(\vare)$,  
and therefore $\Phi$ has the only one fixed point $u$, which is the unique solution.  

(ii) Next we consider the case $r \in (1,\frac{2n}{n + 2 \sigma}]$.  	
By Corollary A, we have
\begin{equation*}
\|u_1 \|_{{{r}}}^\prime = \|u_1 \|_{{{r,2}}} 
\lesssim \| | x |^{-\delta}\|_{{{\frac{n}{\delta}, \infty} }}
\| \langle \cdot \rangle^\delta u_1 \|_{{{{ \frac{nr}{n - r \delta},2}} }}
\lesssim \| \langle \cdot \rangle^\delta u_1 \|_{{{{ \frac{nr}{n - r \delta}},2} }}. 
\end{equation*}	
Substituting this inequality into 
  \eqref{Phiest}, we obtain 
  				  \begin{equation}\label{Phi_est}
				 \begin{aligned}
	\|\Phi u \|_X 
	& \le C_3 
	\Big(
 \|\langle \cdot \rangle^\delta u_0 \|_{{{{\frac{nr}{n - r(\delta + 2\sigma)}},2} }}
	+ \|\langle \cdot \rangle^\delta u_0 \|_{{{{2}} }}
	+  \|  u_0 \|_{{H^{\bar{s}} }}
	\\
	& \quad + 
\|\langle \cdot \rangle^\delta u_1 \|_{{{{ \frac{nr}{n - r \delta},2}} }}
	+ \|  (-\Delta)^{\frac{\bar{s}}{2}} (1 - \Delta)^{-\frac{1}{2}}  u_1 \|_{{2}} 
	+  \|u \|_X^p \Big),	
	 \end{aligned}
	 \end{equation}	
for a positive constant $C_3$ independent of initial data.  
		Hence, taking $\vare_2 > 0$ such that $C_3 \vare_2^{p-1} \le \frac{1}{2}$, and assuming $u_0$ and $u_1$ satisfy
\begin{equation*}
\begin{aligned}
C_3	&\Big(
	\|\langle \cdot \rangle^\delta u_0 \|_{{{{\frac{nr}{n - r(\delta + 2\sigma)}},2} }}
	+ \|\langle \cdot \rangle^\delta u_0 \|_{{{{2}} }}
	+  \|  u_0 \|_{{H^{\bar{s}} }}
	\\
	& \quad \quad + 
	\|\langle \cdot \rangle^\delta u_1 \|_{{{{ \frac{nr}{n - r \delta},2}} }}
	+ \|  (-\Delta)^{\frac{\bar{s}}{2}} (1 - \Delta)^{-\frac{1}{2}}  u_1 \|_{{2}}
	\Big)
\le \frac{\vare_2}{2}, 
\end{aligned}
\end{equation*}		
	$\Phi$ becomes a mapping from  $X(\vare_2)$ to $X(\vare_2)$.  
	In the same way as \eqref{Phiu-v}, 
there is a positive constant $C_4$ independent of initial data such that 
	\begin{equation*}
	 \|\Phi u - \Phi v \|_X \le C_4 (\|u \| + \|v \|)^{p-1} \|u - v \|_X
	\end{equation*}
	for every $u, v \in X$.  
	Thus, taking $\vare \in (0,\vare_2)$ such that 
	$
	C_4  (2 \vare)^{p-1} < 1
	$, 
	$\Phi$ becomes a contraction mapping $X(\vare)$ to $X(\vare)$,  
	and therefore $\Phi$ has the only one fixed point $u$, which is the unique solution.

\subsection{Proof of Theorems 2 and 3}
	
\begin{proof}[Proof of Theorems \ref{existence2}]
	
	We prove Theorem \ref{existence2} by reducing it to Proposition \ref{existence}.  
  
In the case $p_\sigma < p \le 1 + \frac{4}{n}$, 
we can define $r$ by
\begin{equation}\label{rdef}
\frac{1}{r} = \frac{2}{n} \left(\frac{1}{p-1} + \sigma \right) + \eta
\end{equation}
satisfying 
\begin{equation}\label{rass}
\frac{1}{2} + \frac{2 \sigma}{n} < \frac{1}{r} < 1, \quad \text{that is}, \quad r \in (1, \frac{2n}{n + 4 \sigma}),
\end{equation}
 if $\eta >0$ is sufficiently small.   
Then $\frac{1}{r} > \frac{2}{n} \left(\frac{1}{p-1} + \sigma \right) $ means the condition \eqref{pass}.  
  
(Case 1) First we consider the case $p_\sigma < p \le  1 + \frac{4}{n + 2 - 4 \sigma}$.   
We define $r$ by \eqref{rdef} satisfying \eqref{rass}, and we put
  \begin{equation}\label{deltaprimedef}
  \delta^\prime= 2 \left(\frac{1}{p-1} + \sigma \right) - \frac{n}{2} - 1 + n \eta,
    \end{equation}
    which is equal to 
    \begin{equation}\label{deltaprime1}
    \delta^\prime= n \left(\frac{1}{r} - \frac{1}{2} \right) - 1, 
    \end{equation}
  by the definition \eqref{rdef}.  
Comparing \eqref{deltaass1} and \eqref{deltaprimedef}, we see that
\begin{equation}\label{deltaprime}
\delta^\prime \le \delta 
\end{equation}
if $\eta >0$ is sufficiently small.   
Hence, taking $\eta > 0$ sufficiently small, we can assume that $r$ and $\delta^\prime$ defined above satisfy \eqref{rass} and \eqref{deltaprime}.    
We check that the conditions  \eqref{delta} and \eqref{delta2} of Proposition \ref{existence} are satisfied 
with $\delta$ replaced $\delta^\prime$.  
Since $2 \sigma < 1$, \eqref{delta} is trivial by \eqref{deltaprime1}.  
Since $n \ge 2$ and $2 \sigma < 1$, we have
$$
p \le  1 + \frac{4}{n + 2 - 4 \sigma} < 1 + \frac{1}{1 - 2\sigma} = \frac{2 - 2\sigma}{1 - 2\sigma},  
$$
from which it follows that 
$$
2\left(\frac{1}{p - 1} + \sigma \right) - 1 > \frac{2}{p} \left(\frac{1}{p - 1} + \sigma \right). 
$$
From this and the definition of $\delta^\prime$ and $r$, it follows that
\begin{equation}
\delta^\prime= 2 \left(\frac{1}{p-1} + \sigma \right) -1 - \frac{n}{2} + n \eta
> \frac{2}{p} \left(\frac{1}{p - 1} + \sigma \right) + n \eta - \frac{n}{2}
= n\left( \frac{1}{pr} - \frac{1}{2} \right), 
\end{equation}
that is, \eqref{delta2} is satisfied with $\delta$ replaced by $\delta^\prime$.  
Hence the assumption of Proposition \ref{existence} is satisfied.  
Let $\hat q_j$ $(j = 0,1)$ be the constants defined by \eqref{hat_qdef} with $\delta = \delta^\prime$ and $r$ defined above.  
Then 
\begin{equation}
\frac{1}{\hat q_0   } = \frac{n - r(\delta^\prime + 2 \sigma)}{nr} = \frac{n + 2 - 4 \sigma}{2 n},
\quad 
\frac{1}{\hat q_1   } = \frac{n - r \delta^\prime}{nr} = \frac{n+2}{2n},
\end{equation}
that is, $\hat q_j    = q_j$ $(j = 0,1)$.  
Since $\delta^\prime \le \delta$, the conditions \eqref{initial1} implies \eqref{initialass_r} with $\delta$ replaced by $\delta^\prime$.  
Thus, Proposition \ref{existence} guarantees the existence of the solution $u \in C^1([0,\infty), H^{\bar{s}}) \cap C([0,\infty), H^{\bar{s} - 1})$ if $\vare$ is sufficiently small.  
By the standard argument, the uniqueness holds in the class 
$C^1([0,\infty), H^{\bar{s}}) \cap C([0,\infty), H^{\bar{s} - 1})$.  

(Case 2-1) Next we consider the case 
\begin{equation}\label{pass3}
1 + \frac{4}{n + 2 - 4\sigma} < p \le 1 + \frac{4}{n}. 
\end{equation} 
We define $r$ by \eqref{rdef} satisfying \eqref{rass}.  
We show that $\delta = 0$ satisfies the conditions  \eqref{delta} and \eqref{delta2}, that is, 
\begin{align}
&  2 \left(\frac{1}{p-1} + \sigma \right) + n \eta  - \frac{n}{2} - 1
=	n(\frac{1}{r} - \frac{1}{2}) - 1 
\le 0 
< n(\frac{1}{r} - \frac{1}{2}) - 2\sigma, 
\label{cond3}
\\
& \frac{2}{p}\left(\frac{1}{p-1} + \sigma \right) + \frac{n \eta}{p} - \frac{n}{2}
=  	n(\frac{1}{pr} - \frac{1}{2}) \le 0,
\label{cond4}
\end{align}
if $\eta > 0$ is sufficiently small.  

The condition \eqref{rass} implies $n(\frac{1}{r} - \frac{1}{2}) - 2\sigma > 0$.   The assumption 
\\
$1~+~\frac{4}{n + 2 - 4\sigma} < p$ is equivalent to 
\begin{equation}
  2 \left(\frac{1}{p-1} + \sigma \right)  - \frac{n}{2} - 1 < 0. 
  \end{equation}
  Hence \eqref{cond3} holds if $\eta$ is sufficiently small.   

From the assumption that $p \ge 1 + \frac{4}{n + 2 - 4 \sigma}$, it follows that
\begin{align*}
 \frac{2}{p}\left(\frac{1}{p-1} + \sigma \right) - \frac{n}{2} 
&\le \frac{2(n + 2 - 4\sigma)}{n + 6 - 4\sigma}\left(\frac{n + 2 - 4\sigma}{4} + \sigma \right) - \frac{n}{2} 
\\
&=  -\frac{n+ 4 \sigma - 2}{n + 6 - 4\sigma} < 0,
\end{align*}
and thus the condition \eqref{cond4} holds if $\eta > 0$ is sufficiently small.   
Hence the assumption of Proposition \ref{existence} is satisfied with $\delta = 0$.  
Let $\hat q_j$ $(j = 0,1)$ be the constants defined by \eqref{hat_qdef} with $\delta = 0$, and $r$ be defined by \eqref{rdef} and \eqref{rass}.  
Then  
\begin{equation}
\frac{1}{\hat q_0   } = \frac{1}{r} - \frac{2 \sigma}{n} = \frac{2}{ n(p-1)} + \eta,
\quad 
\frac{1}{\hat q_1   } = \frac{1}{r} =  = \frac{2}{n(p-1)} + \frac{2 \sigma}{n} + \eta. 
\end{equation}
Then by the assumption of $q_j$ ($j = 0,1$), we have
\begin{equation}
\label{qeta_ass}
q_j < \hat q_j    \quad (j = 0,1),
\end{equation}
if $\eta > 0$ is sufficiently small.  
Since $p < 1 + \frac{4}{n}$, 
\begin{equation}\label{qj2}
\frac{1}{\hat q_j} > \frac{2}{n(p-1)} > \frac{1}{2}, \quad \text{that is,} \quad \hat q_j < 2 \quad (j = 0,1). 
\end{equation}
By \eqref{qeta_ass} and \eqref{qj2}, 
$$
L_{\hat q_j,2   } \subset L_{q_j,2} \cap L_{2}.  
$$
Hence \eqref{initial2} implies \eqref{initialass_r}. 
Thus the conclusion holds by Proposition \ref{existence} in the same way as above.   
	
(Case 2-2) Last we consider the case $p \ge 1 + \frac{4}{n}$.  	
We define $r$ by
		\begin{equation}\label{rdef2}
	\frac{1}{r} = \frac{1}{2} + \frac{2\sigma}{n} + \eta \quad (\eta > 0).  
		\end{equation} 
	Since $2 \sigma < 1$,  \eqref{delta} holds for $\delta = 0$ if $\eta  > 0$ is sufficiently small.  
The condition
$$\
n(\frac{1}{pr} - \frac{1}{2}) = \frac{n}{p}\left(\frac{1}{2} + \frac{2\sigma}{n} + \eta \right) - \frac{n}{2} < 0
$$
is equivalent to
$$
1 + \frac{4\sigma}{n} + 2 \eta < p,
$$
which holds if $\eta > 0$ is sufficiently small, since $\sigma < 1$ and $p \ge 1 + \frac{4}{n}$.   
Hence, defining $r$ by \eqref{rdef2} with 
sufficiently small $\eta > 0$, we can take $\delta = 0$ in Proposition \ref{existence}.  
Let $\hat q_j$ $(j = 0,1)$ be the constants defined by \eqref{hat_qdef} with $\delta = 0$ and $r$ defined above.  Then, considering the asumption of $q_j$ ($j = 0,1$), we see that
\begin{equation}
\frac{1}{\hat q_0   } = \frac{1}{r} - \frac{2 \sigma}{n} = \frac{1}{2} + \eta < \frac{1}{q_0},
\quad 
\frac{1}{\hat q_1   } = \frac{1}{r} = \frac{1}{2} + \frac{2 \sigma}{n} + \eta < \frac{1}{q_1},  
\end{equation}
if $\eta > 0$ is sufficiently small.   
This imply that $q_j < \hat q_j < 2$ $(j = 0,1$).  
Hence \eqref{initial3} implies \eqref{initialass_r} with $\delta = 0$, and the conclusion holds by Proposition \ref{existence}.  
\end{proof}
	
\begin{proof}[Proof of Theorem \ref{thmdiff}]	
Since $u =\Phi u$ ($\Phi$ is defined by \eqref{Phidef}), we can write
\begin{align}
&u(t,\cdot) - \varTheta G_{\sigma}(t,x) 
=
\left( K_0(t,\cdot)\ast u_0
+
K_1(t,\cdot)\ast u_1 
- \vartheta_1 G_{\sigma}(t,x) 
\right)
\nonumber \\
&+ \left(\int_0^t K_1(t - \tau,\cdot)* f(u(\tau,\cdot)) d \tau
	- G_{\sigma}(t,x)
		\int_0^\infty \int_{\R^n}f(u(\tau,y))dy d\tau
	\right),
\label{I1I2I3} \end{align}
where $\vartheta_1$ is defined by \eqref{theta1def}.  
Since $K_0(t,\cdot)\ast u_0 + K_1(t,\cdot)\ast u_1$ is a solution of the linear equation \eqref{LW}, 
the first term of the right-hand side of \eqref{I1I2I3} is estimated by Theorem \ref{thm_lin_diff}.   
The second term is estimated in \eqref{heat+}.  
Combining these estimates, we obtain the assertion. 
\end{proof}

\end{document}